\documentclass[a4paper,12pt]{article}
\usepackage[top=3cm, left=2.5cm, bottom=3cm, right=2.5cm]{geometry}
\usepackage{geometry}
\usepackage{graphicx}
\usepackage{amsmath, amsthm, amssymb} 
\usepackage{overpic}
\usepackage{dsfont}

\usepackage[ngerman, french, english]{babel}
\usepackage[latin1]{inputenc}
\usepackage[T1]{fontenc}
\usepackage{color}
\usepackage{natbib}
\usepackage{setspace}
\usepackage{longtable}
\usepackage{bm}

\newtheorem{theorem}{Theorem}[section]

\newtheorem{prop}[theorem]{Proposition}
\newtheorem{corollary}[theorem]{Corollary}
\newtheorem{cor}[theorem]{Corollary}
\newtheorem{lemma}[theorem]{Lemma}
\newtheorem{lem}[theorem]{Lemma}
\newtheorem{remark}[theorem]{Remark}
\newtheorem{example}[theorem]{Example}

\newtheorem{assumption}[theorem]{Assumption}

\def\R{\mathbb{R}}
\def\P{\mathbb{P}}
\def\E{\mathbb{E}}

\def\eps {\epsilon}

\DeclareMathOperator\dist{dist}

\DeclareMathOperator\e{e}
\DeclareMathOperator\di{d}
\DeclareMathOperator\Var{Var}
\DeclareMathOperator\Cov{Cov}
\DeclareMathOperator\argmax{argmax}
\DeclareMathOperator\argmin{argmin}

\allowdisplaybreaks

\title{\begin{Huge}Large deviations for infectious diseases models\end{Huge}}

\author{Peter Kratz and Etienne Pardoux\footnote{Aix-Marseille Universit\'e, CNRS, Centrale Marseille, I2M, UMR 7373, 13453 Marseille, France, e-mail adress: etienne.pardoux@univ-amu.fr}}

\begin{document}
%\thispagestyle{empty}
%\setmargins{2cm}{0cm}{0cm}{0cm}%
%{7mm}{1.2cm}{4mm}{1.5cm}\sffamily

\maketitle

\section{Introduction} \label{Introduction}

Consider a model of infectious disease dynamics where the total number of individuals is constant over time,
equal to $N$, and we denote by $Z^N(t)$ the vector of proportions of this population in each compartment
(susceptible, infectious, removed, etc.). Our probabilistic model takes into account each event of infection,
removal, etc. It takes the form
\begin{equation}\label{EqPoisson1}
Z^{N,x}(t)  :=Z^N(t):= x +\frac{1}{N} \sum_{j=1}^k h_j P_j\Big(\int_0^t N \beta_j(Z^N(s)) ds \Big).
\end{equation}
Here, $(P_j)_{1\le j\le k}$ are i.i.d. standard Poisson processes. The $h_j\in \mathds{Z}^d$ denote the $k$ respective jump directions with jump rates $\beta_j(x)$ and $x\in A$ (where $A$ is the ``domain'' of the process). The $d$ components of the process denote the ``proportions'' of individuals in the various compartments. Usually $A$ is a compact or at least a bounded set. For example, in the models we have in mind  the compartment sizes are non-negative, hence $A\subset\mathds{R}^d_+$. 

As we shall prove below, $Z^{N,x}_t\to Y^x_t$ as $N\to\infty$, where $Y^x_t$ is the solution of the ODE
\begin{equation}\label{ODE}
Y^{x}(t)  :=  Y(t):=x +\int_0^t b(Y^x(s)) ds,
\end{equation}
with 
\[
b(z):=\sum_{j=1}^k \beta_j(z) h_j.
\]
This  Law of Large Numbers result goes back to~\cite{Kurtz1978} (see the version in Theorem~\ref{ThLLN} below, where a rate of convergence is given).

Most of the literature on mathematical models of disease dynamics  treats deterministic models of the type of \eqref{ODE}. When an epidemics is established, and each compartment of the model contains a significant proportion of the total population, if $N$ is large enough, the ODE \eqref{ODE} is a good model for the epidemics. The original stochastic model \eqref{EqPoisson1}, which we believe to be more realistic than the \eqref{ODE}, can be considered as a stochastic perturbation of \eqref{ODE}. However, we know from the work of \cite{Freidlin2012}, that small Brownian perturbations of an ODE will eventually produce a large deviation from its law of large numbers limit. For instance, if the ODE starts in the basin of attraction of an locally stable equilibrium,
the solution of the ODE converges to that equilibrium, and stays for ever close to that equilibrium. The theory of Freidlin and Wentzell, based upon the theory of Large Deviations, predicts that soon or later the solution of a random perturbation of that ODE will exit the basin of attraction of the equilibrium. The aim of this paper is to show that the Poissonian perturbation \eqref{EqPoisson1} of \eqref{ODE} behaves similarly. This should allow us to predicts the time taken by an endemic equilibrium to cease, and a disease--free equilibrium to replace it. 

We shall apply at the end of this paper our results to the following example.
\begin{example} \label{ExampleSIRS}
We consider a so-called $SIRS$ model without demography ($S(t)$ being the number of susceptible individuals, $I(t)$ the number of infectious individuals and $R(t)$ the number of removed/immune individuals at time $t$). We let $\beta>0$ and assume that the average number of new infections per unit time is $\beta S(t) I(t)/N$.\footnote{The reasoning behind this is the following. Assume that an infectious individuals meets on average $\alpha>0$ other individuals in unit time. If each contact of a susceptible and an infectious individual yields a new infection with probability $p$, the average number of new infections per unit time is $\beta S(t) I(t)/N$, where $\beta=p \alpha$ since all individuals are contacted with the same probability (hence $S(t)/N$ is the probability that a contacted individual is susceptible).}
For $\gamma, \nu>0$, we assume that the average number of recoveries per unit time is $\gamma I(t)$ and the average number of individuals who lose immunity is $\nu R(t)$. As population size is constant, we can reduce the dimension of the model by solely considering the proportion of infectious and removed at time $t$. Using the notation of equations~\eqref{EqPoisson1} and~\eqref{ODE}, we have 
\begin{align*}
&A=\{x \in \mathds{R}^2_+| 0 \leq x_1+x_2\leq 1\}, \quad h_1=(1,0)^\top, \quad h_2=(-1,1)^\top, \quad  h_3=(0,-1)^\top, \\*
& \beta_1(z)= \beta z_1 (1-z_1-z_2), \quad \beta_2(z)=\gamma z_1, \quad \beta_3(z)=\nu z_2.
\end{align*}
%\color{red}The text needs to be changed. This is taken from CEMRACS paper.
%\color{black}
It is easy to see that in this example the ODE~\eqref{ODE} has a disease free equilibrium $\bar x=(0,0)^\top$. This equilibrium is asymptotically stable if $R_0=\beta/\gamma<1$.
%\color{red}(CHECK)\color{black}
$R_0$ is the so-called \emph{basic reproduction number}. It denotes the average number of secondary cases infected by one primary case during its infectious period at the start of the epidemic (while essentially everybody is susceptible).
If $R_0>1$, $\bar x$ is unstable and there exists a second, endemic equilibrium 
$x^*=(\frac{\nu(\beta-\gamma)}{\beta(\gamma+\nu)},\frac{\gamma(\beta-\gamma)}{\beta(\gamma+\nu})$ 
which is asymptotically stable. While in the deterministic model the proportion of infectious and removed individuals converges to the endemic equilibrium $x^*$, the disease will go extinct soon or later in the stochastic model.

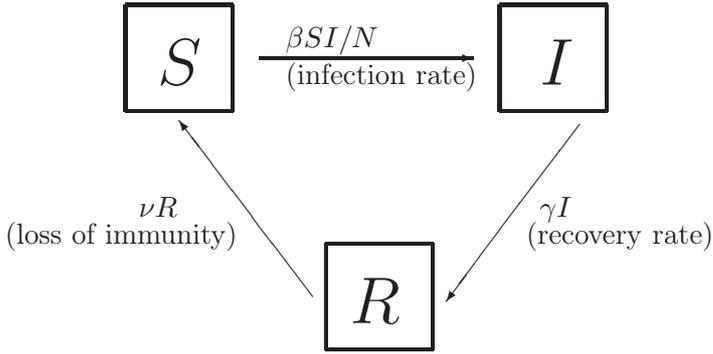
\begin{figure} %\label{FigSIV}
\begin{picture} (290,140)
\linethickness{0.4mm} 

\put(135,0){\line(1,0){40}}
\put(135,0){\line(0,1){40}} 
\put(175,0){\line(0,1){40}}
\put(135,40){\line(1,0){40}}

\put(60,90){\line(1,0){40}}
 \put(60,90){\line(0,1){40}}
\put(100,90){\line(0,1){40}} 
\put(60,130){\line(1,0){40}}

\put(200,90){\line(1,0){40}} 
\put(200,90){\line(0,1){40}}
\put(240,90){\line(0,1){40}} 
\put(200,130){\line(1,0){40}}

\put(110,110){\vector(1,0){80}}
\put(130,20){\vector(-3,4){50}} 
\put(230,85){\vector(-3,-4){50}}

\put(72,100){\Huge{$S$}} 
\put(215,100){\Huge{$I$}}
\put(144,10){\Huge{$R$}}

\put(120,114){\small{$\beta S I/N$}}
\put(120,100){\small{(infection rate)}} 

\put(65,50){\small{$\nu R$}} 
\put(15,40){\small{(loss of immunity})} 

\put(215,50){\small{$\gamma I$}}
\put(210,40){\small{(recovery rate)}}

\end{picture}
\caption{Transmission diagram of the $SIRS$ model without demography.}
\end{figure}
\end{example}

Our results also apply to other models like the $SIV$ model ($V$ like  \emph{vaccinated}) and the $S_0IS_1$
(with two levels of susceptibility), see \cite{kratz2015} and the references therein. These two models have the property that for certain values of their parameters, both the disease--free equilibrium and one of the endemic equilibria are locally stable.
Our results predict the time taken by the solution of the stochastic equation to leave the basin of attraction of the endemic
equilibrium. We shall discuss those and other applications elsewhere in the future.

There is already a vast literature on the theory of large deviations for systems with Poissonian inputs, see
\cite{Dupuis1997}, \cite{Dupuis1991}, \cite{Feng2006}, \cite{Shwartz1995}, among others.

However, the assumptions in those works are not satisfied in our case. The difficulty is the following. For obvious reasons, the solution of our SDE \eqref{EqPoisson1} must remain in $\mathds{R}^d_+$. This implies that some rates vanish when one of the components of $Z^{N,x}(t)$ vanishes. However, the expression of the large deviation fucntional (as well as the ratio of the probabilities in the Girsanov theorem) involves the logarithm of those rates, which hence explodes as the rate vanishes. The same happens with the computer network models which was the motivation of the work of  \cite{Shwartz1995}, and this led them to consider situations with vanishing rates in \cite{Shwartz2005}. However, even the assumptions in that paper are not fully satisfied in our models
(see our discussion below in section \ref{sec2.2}).
For that reason, in order to avoid the ackward situation where we would have to cite both \cite{Shwartz1995} and  \cite{Shwartz2005}, and add some arguments to cope with our specific situation, we prefered to rewrite the whole theory, so as to cover the situation of the epidemiological models in a self--consistent way. We must however 
recognize that the work of Shwartz and Weiss has been an importance source of inspiration for this work.

Let us now discuss one subtlety of our models. In the models without demography, i.e. the models where the total population  remains constant, then 
we choose $N$ as this total population, so that the various components of the vector $Z^{N,x}_t$ are the proportions of the total population in the various compartments of the model, that is each component of $Z^{N,x}_t$ at any time is of the form $k/N$, where $k\in\mathds{Z}_+$, and also
$\sum_{i=1}^dZ^{N,x}_i(t)=1$, if $Z^{N,x}_i(t)$ denotes the $i$--th component of the vector $Z^{N,x}(t)$, $1\le i\le d$. In this case, provided we start our SDE 
 from a point of the type $(k_1/N,\ldots,k_d/N)$, where $k_1,\ldots,k_d\in\mathds{Z}_+$, then the solution visits only such points, and cannot escape the set
 $\mathds{R}_+^d$ without hitting first its boundary, where the rates for exiting vanish. Consequently $Z^{N,x}_t$ remains in $\mathds{R}_+^d$ for all time
 a.s. However, if we start our process outside the above grid, or if the total population size does not remain constant,  the components of the vector 
 $Z^{N,x}_t$ multiplied by $N$ need not be integers. Then some of the components of $Z^{N,x}_t$ might become negative, and one can still continue to define
 $Z^{N,x}_t$ provided for any $1\le i\le d$ and $1\le j\le k$, the rule $x_i=0\Rightarrow \beta_j(x)=0$ is extended to $x_i\le0\Rightarrow \beta_j(x)=0$.
 However, in order to make things simpler, we restrict ourselves in this paper to the situation where all coordinates of the vector $NZ^{N,x}_0$ are integers,
 and the same is true with $NZ^{N,x}_t$ for all $t>0$. In particular, we shall consider equation \eqref{EqPoisson1} only with a starting point $x$ such that
 all coordinates of $Nx$ are integers. This will be explicitly recalled in the main statements, and implicitly assumed everywhere. We shall consider more general situations in a further publication.
 \bigskip
 
The paper is organized as follows Our set--up is made precise and the general assumptions are formulated in section \ref{sec2}. Section \ref{sec3} is devoted to the law of large numbers. In section \ref{SectionRateFunction} we study the rate function. The Large Deviations lower bound is established in section \ref{SecLower} and  the Large Deviations upper bound in section \ref{SecUpper}. Section \ref{SectionExitTime} treats the exit time from a domain, including the case of a characteristic boundary. Finally in section \ref{AppendixSIRS} we show how our results apply to the SIRS model (which requires an additional argument), and a Girsanov theorem is formulated in the Appendix.

\section{Set-up} \label{SectionSetup}\label{sec2}

We consider a set $A\subset \mathds{R}^d$ (whose properties we specify below) and define the grids
\[
\mathds{Z}^{d,N}:=\{x \in \mathds{R}^d| x_i=j/N \text{ for some } j \in \mathds{Z} \}, \quad A^N:=A \cap \mathds{Z}^{d,N}.
\]
%For $x\in A$, we let 
%\[
%x^N:=\argmin_{z \in A^N}  |z-x|.\footnotemark
%\]
%\footnotetext{We readily observe that $x^N$ is not necessarily unique. It is easy to specify a simple rule to render it unique. Note furthermore that $x^N\in \mathring A$ for large enough $N$ if $x \in \mathring A$.}

%\color{red}NEW
%\color{black}
We rewrite the process defined by Equation~\eqref{EqPoisson1} as
\begin{align} \label{EqPoisson}
Z^{N,x}(t) :=Z^N(t)&:= x+\frac{1}{N} \sum_{j=1}^k h_j P_j\Big(\int_0^t N \beta_j(Z^N(s))ds\Big) \\
&=  x+\int_0^t b(Z^{N}(s)) ds+ \frac{1}{N} \sum_j h_j M_j\Big(\int_0^t N \beta_j(Z^N(s)) ds \Big), \notag
\end{align}
where the $M_j(t)=P_j(t)-t$ are the compensated Poisson processes corresponding to the $P_j$ ($j=1,\dots,k$).  

Let us shortly comment on this definition. In the models we have in mind, the components of $Z^N$ usually denote the proportions of individuals in the respective compartments. It is hence plausible to demand that the starting point is in $\mathds Z^{d,N}$. However, it is not sufficient to simply restrict our analysis to those starting points with $x \in A^N$ as this does obviously not imply $x \in A^M$ for all $M>N$.
Note that $U^N(t)=NZ^N(t)$ would solve the SDE
\[ U^N(t)=Nx+\sum_{j=1}^k h_jP_j\left(\int_0^t\beta_{j,N}(U^N(s))\right),\]
where $N\beta_j(x)=\beta_{j,N}(Nx)$. Here the coefficients of the vector $U^N(t)$ are the numbers of individuals from the population in each compartment.
The equation for $U^N(t)$ is really the original model, where all events of infection,  recovery, loss of immunity, etc. are  modeled. Dividing by $N$ leads to a process which has  a law of large number limits as $N\to\infty$. The crucial assumption for this procedure to make sense is that $N^{-1}\beta_{j,N}(Nx)$ does not depend upon $N$, which is typically the case in the epidemics models, see in particular  Example \ref{ExampleSIRS}.

We first introduce the following notations.
For $x \in A$ and $y \in \mathds{R}^d$, let 
\begin{align*}
V_{x}&:= \Big\{\mu \in \mathds{R}^k_+ | \mu_j >0 \text{ only if } \beta_j(x)>0\Big\},\\
V_{x,y}&:= \Big\{\mu \in V_{x} | y=\sum_j \mu_j h_j\Big\}.
\end{align*}
As $V_{x,y}$ is sometimes independent of $x$ or $V_{x,y}=\emptyset$, we also define for $y\in \mathds{R}^d$,
\[
\tilde V_y:=\Big\{ \mu \in \mathds{R}^d_+| y=\sum_j \mu_j h_j \Big\}
\]

We define the cone spanned by a (finite) set of vectors $(v_j)_j$ ($v_j \in \mathds{R}^d$ by
\[
\mathcal{C}((v_j)_j):=\Big\{v=\sum_j \alpha_j v_j|\alpha_j\geq 0\Big\}.
\]
Similarly, we define the cone generated by the jump directions $(h_j)_j$ at $x \in A$ by
\[
\mathcal{C}_x:=\mathcal{C}((h_j)_{j:\beta_j(x)>0}) =\Big\{v=\sum_{j:\beta_j(x)>0} \mu_j h_j | \mu_j\geq 0 \Big\}.
\]
Note that 
\begin{equation}\label{C}
\mathcal{C}_x=\mathcal{C}=\{v=\sum_{j=1}^k\mu_jh_j| \mu_j\ge 0\}
\end{equation}
whenever $x\in\mathring{A}$, since $\beta_j(x)>0$ for all $1\le j\le k$ if $x\in\mathring{A}$. Also, in part of this paper, we shall assume
that the $\log \beta_j$'s are bounded, which then means that \eqref{C} is true for all $x\in A$.

We define the following upper and lower bounds of the rates. Let $\rho>0$.
\begin{align*}
\bar{\beta}&:=\sup_{x\in A,\, j=1,\dots,k} \beta_j(x)\in \bar{\mathds{R}}_+, \\
\underline{\beta}&:=\inf_{x\in A,\, j=1,\dots,k} \beta_j(x)\in \mathds{R}_+,\\
\underline{\beta}(\rho)&:=\inf \big\{ \beta_j(x)|j=1,\dots,k, \, x \in A \text{ and } |x-z|\geq \rho\forall z\in A \text{ with } \beta_j(z)=0\big\} \in \mathds{R}_+,\\
\bar{h}&:=\sup_{j=1,\dots,k}|h_j|\in \mathds{R}_+.
\end{align*}

\subsection{The Legendre-Fenchel transform and the rate function} \label{SubSecDefRateFct}

We define the following transforms. For $x\in A$, $y\in \R^d$, let
\begin{align*}
 \ell(x,\mu)&:= \sum_{j} \big\{ \beta_j(x) - \mu_j + \mu_j \log\big(\tfrac{\mu_j}{\beta_j(x)}\big)\big\},
\end{align*}
with  the convention $0 \log(0/\alpha)=0$ for all $\alpha \in \mathds{R}$,
and
\begin{equation}\label{Def-L}
\overline{L}(x,y):=\begin{cases}
\inf_{\mu \in V_{x,y}}  \ell (x,\mu)& \text{if } V_{x,y} \not=\emptyset \\
+\infty& \text{otherwise}.
\end{cases}
\end{equation}
Now let, for $x,y$ as above and $\theta\in\R^d$, 
\[
\tilde{\ell}(x,y,\theta) = \langle \theta,y \rangle - \sum_{j} \beta_j(x) \big( \e^{\langle \theta, h_j \rangle} -1 \big),
\]
and define
\begin{align}\label{autreDef-L}
\underline{L}(x,y)&:=\sup_{\theta \in \mathds{R}^d} \tilde{\ell} (x,y,\theta), 
\end{align}
%For the time being, we shall use the notation
%\begin{align*}
%\tilde{L}(x,y)=\sup_{\theta\in\R^d} \tilde{\ell} (\theta,x,y)
%\end{align*}

\begin{remark}
For $\mu \in \tilde V_y \setminus V_{x,y}$, we have $ \ell (x,\mu)=\infty$ and hence
\[
 L(x,y)=\inf_{\mu \in \tilde V_y}  \ell (x,\mu).
\]
\end{remark}
We first show
\begin{lem} \label{Lemma5.25}
Let $x \in A$, $y \in \mathcal{C}_x$, $\theta \in \mathds{R}^d$ and $\mu\in V_{x,y}$. Then
\[
\tilde{\ell}(x,y,\theta) \leq \ell(x,\mu),
\]
in particular
\[
\underline{L}(x,y) \leq \overline{L}(x,y).
\]
\end{lem}

\begin{proof} The result is obvious if $V_{x,y}=\emptyset$. If not, 
 for $\mu \in V_{x,y}$, with
\[ f_j(z)=\mu_j z-\beta_j(x)(e^z-1),\]
\begin{align*}
\tilde{\ell}(x,y,\theta)&=\sum_{j} \mu_j \langle \theta,h_j\rangle - \beta_j(x) (\exp (\langle \theta,h_j\rangle) - 1)\\
&=\sum_j f_j(\langle \theta,h_j\rangle) \\
&\leq \sum_j f_j(\log \mu_j/\beta_j(x))
\\&= \ell (x,\mu),
\end{align*}
since $f_j$ achieves its maximum at $z=\log[\mu_j/\beta_j(x)]$.
\end{proof}

We will show below that under appropriate assumptions, $\underline{L}(x,y)=\overline{L}(x,y)$, and we shall
write $L(x,y)$ for the common value of those two quantities.

For any $T>0$, we define
\begin{align*}
C([0,T];A) &:=\{\phi:[0,T]\rightarrow A| \phi \text{ continuous} \},\\
D([0,T];A) &:=\{\phi:[0,T]\rightarrow A| \phi \text{ c\`{a}dl\`{a}g} \}.\\
\end{align*}
On $C([0,T];A)$ (or $D([0,T];A)$), $\di_C$ denotes the metric corresponding to the supremum-norm, denoted by $\| \cdot \|$. Whenever the context is clear, we write $\di:=\di_C$. On $D([0,T];A)$ we denote by $\di_D$ the metric given, e.g., in~\cite{Billingsley1999}, Sections 12.1 and 12.2 which defines the Skorohod topology in such a way that the resulting space is Polish. The resulting metric spaces are denoted by $C([0,T];A;\di_C)$, $D([0,T];A;\di_C)$ and $D([0,T];A;\di_D)$, respectively (where the metrics are omitted, whenever they are clear from the context). 

We now introduce a candidate $I$ for the rate function. For $\phi:[0,T]\rightarrow A$, let
\begin{align*}
I_T(\phi)&:= \begin{cases}
\int_0^T L(\phi(t),\phi'(t))  dt& \text{ if } \phi \text{ is absolutely continuous} \\
\infty & \text{ otherwise.}
\end{cases}
\end{align*}
For $x\in A$ and $\phi:[0,T]\rightarrow A$, let
\begin{equation*}
I_{T,x}(\phi):=
\begin{cases}
I_{T}(\phi)& \text{ if } \phi(0)=x \\
\infty & \text{ otherwise.}
\end{cases}
\end{equation*}

\subsection{Assumptions on the process $Z^N$}\label{sec2.2}

We do not assume that the rates are bounded away form zero (as in~\cite{Shwartz1995}) and allow for them to vanish near the boundary (cf.~the discussion in the introduction). \cite{Shwartz2005} allow for vanishing rates. We generalize these assumptions as we outline below. The difference is essentially Assumption~\ref{MainAss}~(C) below.

%\color{red}
%Add assumption for LPD lower bound: Lipschitz extentsion outside ``$A^\delta$'' possible.
%\color{black}.

\begin{assumption} \label{MainAss}
\begin{enumerate}
\item[(A)] Assumptions on the set $A$.
\begin{enumerate}
\item[(A1)]
The set $A$ is compact and $A=\overline{\mathring{A}}$. Furthermore, there exists a constant $\lambda_0>0$ such that for all $N\in \mathds{N}$, $z\in A^N$ and $j=1,\dots,k$ with $\beta_j(z)>0$,
\[
z+\frac{h_j}{N} \in A^N \quad \text{and} \quad |\tilde z-z|\geq \frac{\lambda_0}{N} \text{ for all } \tilde z \text{ with } \beta_j(\tilde z)=0.\footnotemark
\]
\footnotetext{This implies that for $z \in A^N$ with $\beta_j(z)> 0$, we have $\beta_j(z) \geq \underline{\beta}(\lambda_0/N)$.}
\item[(A2)]
There exist open balls $B_i=B(x_i,r_i)$\footnote{Here (and later) $B(x,r)$ denotes the open ball around $x$ with radius $r$.}, $i=1,\dots,I_1, \dots,I$ ($0< I_1 < I$) such that 
\[
x_i\in \partial A \text{ for } i \leq I_1 \quad \text{and} \quad x_i\in \mathring A\text{  for } i>I_1
\]
and
\[
A \subset \bigcup_{i\leq I} B_i, \quad \partial A \subset \bigcup_{i \leq I_1} B_i \quad \text{and} \quad B_i \cap \partial A = \emptyset \text{ for } i>I_1.
\]
\item[(A3)]
There exist (universal) constants $\lambda_1,\lambda_2>0$ and vectors $v_i$ ($i\leq I_1$, w.l.o.g., we assume $0<|v_i|\leq 1$; for notational reasons, we set $v_i=0$ for $i>I_1$) such that
for all $x \in B_i \cap A$, 
\[
B(x+tv_i,\lambda_1 t)\subset A\quad \text{for all } t \in (0,\lambda_2).
\]
and $\dist(x+t v_i,\partial A)$ is increasing for $t\in (0,\lambda_2)$.
\item[(A4)] There exists a Lipschitz continuous mapping $\psi_A:\mathds{R}^d\to A$ such that $\psi_A(x)=x$ whenever
$x\in A$.
\end{enumerate}
\item[(B)]
Assumptions on the rates $\beta_j$.
\begin{enumerate}
\item[(B1)]
The rates $\beta_j:A \rightarrow \mathds{R}_+$ are Lipschitz continuous. 
\item[(B2)] 
For $x \in \mathring A$, $j=1,\dots,k$, $\beta_j(x)>0$  and $\mathcal{C}((h_j)_j) =\mathds{R}^d$.
\item[(B3)]
For all $x \in \partial A$ there exists a constant $\lambda_3=\lambda_3(x)>0$ such that
\[
y\in \mathcal{C}_x, |y|\leq \lambda_3 \Rightarrow x+y \in A.
\]
\item[(B4)]
There exists a (universal) constant $\lambda_4>0$ such that for all $i \leq I_1$, $x \in B_i\cap A$ and
\[
v \in \mathcal C_{1,i}:=\Big\{ \frac{\tilde v}{|\tilde v|} \Big| \tilde v =  v_i + w \text{ for } w\in \mathds{R}^d, |w|\leq \frac{\lambda_1}{3 - \lambda_1}\Big\},
\]
\[
\beta_j(x)<\lambda_4 \Rightarrow \beta_j(x + \cdot v) \text{ is increasing in } (0, \lambda_2).
\]
\end{enumerate}
\item[(C)] 
There exists an $\eta_0>0$ such that for all $N\in \mathds{N}$, $\epsilon>0$ there exists a constant $\delta(N,\epsilon)>0$ (decreasing in $N$ and in $\epsilon$) such that for all $i\leq I_1$, $x \in B_i$ there exists a $\mu^i=\mu^i(x)\in \tilde V_{v_i}$ and
\begin{equation} \label{EqAssC3.1}
\mathbb{P}\Big [\sup_{t \in [0,\eta_0]} |\tilde Z^{N,x}(t)-\phi^x(t)| \geq \epsilon\Big] \leq \delta(N,\epsilon),
\end{equation}
where $\tilde Z^{N,x}$ denotes the solution of~\eqref{EqPoisson} if the rates $\beta_j$ are replaced by the rates $\tilde \mu_j^i$ for
\[
\tilde \mu_j^i(z):=\begin{cases}
\mu_j^i &\text{ if } z+\epsilon h_j \in A \text{ for all } \epsilon \text{ small enough} \\
0 &\text{else}
\end{cases}
\]
and $\phi^x=x+t v_i$ as before.\footnote{We do not necessarily have $\mu^i\in V_{x,v_i}$ for all $x\in B_i$; this might not be the case if $x \in \partial A$. In such a case $V_{x,v_i}=\emptyset$ is possible, cf.~the discussion about $x=(1,0)^\top$ for the SIRS model below.}

Furthermore, there exists a constant $\alpha \in (0,1/2)$ and a sequence $\epsilon_N$ such that
\begin{equation} \label{EqAssC3.2a}
\epsilon_N < \frac{1}{N^\alpha} \quad \text{and} \quad \frac{\delta(N,\epsilon_N)}{\epsilon_N} \rightarrow 0 \text{ as } N\rightarrow \infty. 
\end{equation}
and
\begin{equation} \label{EqAssC3.3}
\rho^{\alpha} \log \underline\beta(\rho) \rightarrow 0 \quad \text{as } \rho \rightarrow 0.
\end{equation}

\end{enumerate}
\end{assumption}
%\color{red}
%Verify: Dependence on $x$ required.
%\color{black}

Let us comment on Assumption~\ref{MainAss}. Assumption~(A) is essentially Assumption~2.1 of~\cite{Shwartz2005}. We want to remark that Assumption 2.1~(iv) of~\cite{Shwartz2005} is not included here as it is redundant (see Lemma~3.5 of~\cite{Shwartz2005}; cf.~also the discussion preceding Lemma~\ref{Lemma2005.3.5}). 
%\color{red}
%Comment on (A1). Avoid technical difficulties concerning exit of $A$. In line with applications we consider. Now: %polytopes.
%\color{black}
In the epidemiological models we want to consider, $A$ is a compact, convex $d$-polytope, i.e., $\partial A$ is composed by $d-1$-dimensional hyperplanes. For example for the SIRS model in Example~\ref{ExampleSIRS}, 
\[
A=\{x\in \mathds{R}^2|0\leq x_1+x_2\leq 1\}.
\] 
In line with the Assumption (A1), let us note that we always want to choose the starting point $x$ of equation 
\eqref{EqPoisson1}  to belong to $A^N$. If that would not be the case, then in our simplest models the
solution $Z^{N,x}$ might exit the domain $A$. Choosing the starting point arbitrarily in $A$ would force us to let the rates $\beta_j$ depend upon $N$ (and vanish) near the boundary. Note that the coordinates of the vector $Z^{N,x}_t$
are proportions of the population in various compartments. The coordinates of the vector $NZ^{N,x}_t$ are integers,
while those of the vectors $h_j$ belong to the set $\{-1,0,1\}$.
%\color{red}
%Second comment on (A1): What is $A^N$. Obviously, rates are such that the domain is not left. Also parts of $A^N$ %on the boundary (in particular: corners in all $A^N$). Comment on distance to boundary if rates are positive. %Comment on magnitude of these rates in terms of $\underline \beta$. Then: as before.
%\color{black}

Note that in all situations we have in mind, both the set $A$ itself and its boundary can be covered by a finite number of balls. These balls can furthermore be chosen in such a way that those centered in the interior do not intersect with the boundary. For the SIRS model, we can for instance define the balls covering the boundary by $B(x,3/(4m))$ for large $m \in \mathds{N}$ and $x=(i/m,j/m)^\top$ for $i=0,\dots,m$, $j=0$ or $i=0$, $j=0,\dots,m$ or $i+j=m$. The vectors $v_i$ can be defined to be the inside normal vectors for those balls with $x \not\in\{(0,0)^\top,(1,0)^\top,(0,1)^\top\}$. For the remaining three balls, we define $v_i$ by the normalizations of
\[
(1/2,1/2)^\top, \quad (-1/2,1/4)^\top \quad \text{respectively} \quad (1/2,-1/2)^\top.
\]
In general, the constant $\lambda_1$ can be interpreted to be given via the ``angle'' of the vector $v_i$ to the boundary. We have $\lambda_1\leq 1$. It is straightforward that Assumption~(A) is satisfied for the SIRS model. We also note that Assumption~(A) is not very restrictive, see~\cite{Shwartz2005} Lemma 2.1. In particular, every convex, compact set with non-empty interior satisfies the assumption.
%\color{red}
%Comment: Notations/Setup as in S\&W to use their results and for showing relation to them. For our polytopes: %Assumption A trivial.
%\color{black}

Most of Assumption~(B) is taken from Assumption~2.2 of~\cite{Shwartz2005}. We outline the difference below. Assumption~(B1) is quite standard and ensures in particular that the ODE~\eqref{ODE} admits a unique solution. 
For the compartmental epidemiological models we consider, the rates are usually polynomials and hence this assumption is satisfied. Assumption~(B2) implies that within $\mathring A$ it is possible to move into all directions. Only by approaching the boundary the rates are allowed to vanish. (B3) implies that at least locally the convex cone $x+\mathcal{C}_x$ is included in $A$. In particular, it is not possible to exit the set $A$ from its boundary. 
Assumption~(B4) differs slightly from the corresponding assumption in~\cite{Shwartz2005}. While in~\cite{Shwartz2005}, it is implied that close to the boundary, ``small'' rates are increasing while following the vector $v_i$, we assume this for a set of vector in a ``cone'' around $v_i$. We note that for $i\leq I_1$, $x\in B_i$, $v\in \mathcal C_{1,i}$, we have (cf.~Assumption~(A3))
\[
\dist(x+t v, \partial A) \geq \dist(x+ t v_i,\partial A) - t \frac{\lambda_1}{3-\lambda_1}\ \\
\geq t \lambda_1 \Big(\frac{2-\lambda_1}{3-\lambda_1}\Big).
\]
It is easily seen that this assumption is satisfied for the SIRS model.
In addition to this, \cite{Shwartz2005} also require that (cf.~the meaning of $\lambda_4$ in Assumption~(B4))
\begin{equation} \label{AssSWRemoved}
v_i \in \mathcal{C}\big(\{h_j|\inf_{x \in B_i} \beta_j(x) >\lambda_4 \}\big).
\end{equation}
In order to apply the theory to epidemiological models, we have to remove this assumption. To see this, consider the SIRS model and the point $x=(1,0)^\top$ with corresponding ball $B$ containing it. We readily observe that a vector $v$ pointing ``inside'' $A$ (as required by Assumption~(A3)) which is generated by only those $h_j$ whose corresponding rates are bounded away from zero in $B$ does not exist. We hence replace this assumption by~Assumption~(C), which follows from~\eqref{AssSWRemoved}. Indeed, if Assumption~\eqref{AssSWRemoved} holds, the $\mu^i$ representing $v_i$ can be chosen in such a way that the directions corresponding to components $\mu^i_j>0$ do not point outside $A$ in $B_i$. Hence, $\tilde \mu^i\equiv \mu^i$ (as long as the process is in $B_i$) and the LLN Theorem~\ref{ThLLN} can be applied. In general, Theorem~\ref{ThLLN} cannot be applied as the rates $\tilde \mu^i$ can be discontinuous. Note that the assumption can only fail if $x\in \partial A$. Else, the process is equal to the process with constant rates $\mu^i$ on the set
\[
\Big\{\sup_{t \in [0,\eta_0]} |\tilde Z^{N,x}(t)-\phi^x(t)| < \epsilon \Big\}
\]
for all small enough $\epsilon>0$ and Theorem~\ref{ThLLN} is applicable. We note that Assumption~(C) implies that
\[
\delta(N,\epsilon)\rightarrow 0 \text{ as } N \rightarrow \infty  \text{ for all } \epsilon>0.
\]
Moreover, as $\delta(N,\cdot)$ is decreasing, we can choose $\epsilon_N$ in such a way that
\begin{equation} \label{EqEpsilonNAlpha}
\epsilon_N = \frac{1}{N^\alpha} \quad \text{for some } \alpha \in (0,1).
\end{equation}
We remark here (and further discuss this important issue below) that Assumption~\ref{MainAss}~(C) may well fail to be satisfied. To this end, we consider the SIRS model an $x\in A$ with $x_1=0$. We hence have $\beta_2(x)=0$ and hence the process $Z^{N,x}$ cannot enter the interior of $A$. Therefore, we have
\[
\mathbb{P} \Big[ \sup_{t\in [0,\eta_0]} |Z^{N,x}(t)-\phi^x(t)| \geq \epsilon\Big] =1
\]
for $\epsilon$ small enough. Assumption~(C) can hence be considered as a means to ensure that the process can enter the interior of $A$ from every point on the boundary. \eqref{EqAssC3.3} implies that  
\begin{equation} \label{EqAssC3.2}
\int_0^\eta |\log \underline\beta(\rho)| d\rho \rightarrow 0 \quad \text{as } \eta \rightarrow 0,
\end{equation}
since $\rho^{\alpha/2} |\log \underline \beta (\rho)|\leq C$ for appropriate $C$ and hence $|\log \underline \beta (\rho)|\leq C/\rho^{\alpha/2}$ is integrable,
and hence in particular that the rate $I(\phi)$ of linear functions $\phi$ is finite, as it is shown in Lemma~\ref{LemmaLPDLowerLinear} below. 

It remains to show that Assumption~\ref{MainAss}~(C) is satisfied for the SIRS model. This is accomplished in section~\ref{AppendixSIRS}.

Exploiting \eqref{EqAssC3.3}, it is easy to prove
\begin{lemma}
 Under the Assumption~\ref{MainAss}~(C),
  for all $i\leq I$, $x\in A\cap B_i$, let $\phi^x(t):=x+t v_i$. For all $\epsilon>0$, there exists an $\eta>0$ (independent of $i$, $x$) such that for all $i\leq I$ and all $x\in A\cap B_i$,
\[
I_{\eta,x}(\phi^x) <\epsilon.
\]
\end{lemma}\label{lemmaC1}
Note that for $i>I_1$, we have $\phi^x(t)=x$ for all $t$.
\section{Law of large numbers}\label{sec3}

We first prove the law of large numbers by~\cite{Kurtz1978} with the rate of convergence given as in~\cite{Shwartz1995} Theorem 5.3.

\begin{theorem} \label{ThLLN}
Let $Z^{N.x}$ and $Y^x$ be given as in Equation~\eqref{EqPoisson} and \eqref{ODE} respectively, and assume that the rates $\beta_j$ are bounded and Lipschitz continuous. Then there exist constants $\tilde C_1=\tilde C_1(T)>0$ (independent of $\epsilon$) and $\tilde C_2(\epsilon)=\tilde C_2(T,\epsilon)>0$ with $\tilde C_2(\epsilon)=O(\epsilon^2)$ as $\epsilon \rightarrow 0$ such that
\[
\mathbb{P}\Big [\sup_{t \in [0,T]} |Z^{N,x}(t)-Y^x(t)| \geq \epsilon\Big] \leq \tilde C_1 \exp\big(- N \tilde C_2(\epsilon)\big).
\]
$C_1$ and $C_2$ can be chosen independently of $x$.
\end{theorem}

Before we prove Theorem~\ref{ThLLN}, we require some auxiliary results. We first have

\begin{lemma} \label{Lem5.6}
Let $T>0$.
Suppose that $f:D([0,T];A) \times \mathds{R} \rightarrow \mathds{R}$ and $G:D([0,T];A) \times \mathds{R} \times \mathds{R} \rightarrow \mathds{R}$ are such that for all $\rho>0$,
\[
M(t):=\exp\big(\rho f(Z^{N,x},t)-G(Z^{N,x},t,\rho)\big)
\]
is a right-continuous martingale with mean one. Suppose furthermore that $R:\mathds{R} \times \mathds{R}\rightarrow \mathds{R}$ is increasing in the first argument and
\[
G(\phi,t,\rho) \leq R(t,\rho)
\]
for all $\phi \in D([0,T];A)$ and $\rho>0$. Then for all $\epsilon>0$
\[
\mathbb{P}\Big[\sup_{t\in [0,T]} f(Z^{N,x},t) \geq \epsilon\Big] \leq \inf_{\rho>0} \exp \big( R(T,\rho)- \rho\epsilon).
\] 
\end{lemma}

\begin{proof}
Fix $\rho>0$. Then by the assumptions of the lemma,
\begin{align*}
\mathbb{P}\Big[\sup_{t\in [0,T]} f(Z^{N,x},t) \geq \epsilon\Big] & =\mathbb{P}\Big[\sup_{t\in [0,T]} \exp\big(\rho f(Z^{N,x},t)\big) \geq \exp\big(\rho \epsilon\big)\Big] \\
&\leq \mathbb{P}\Big[\sup_{t\in [0,T]} \exp\big(\rho f(Z^{N,x},t) - G(Z^{N,x},t,\rho)\big) \geq \exp\big(\rho \epsilon-R(T,\rho)\big)\Big] \\
&\leq \exp\big(R(T,\rho)-\rho \epsilon\big)
\end{align*}
where the last inequality is  Doob's martingale inequality, see, e.g. Theorem II.1.7 in~\cite{RevuzYor1999}.
\end{proof}

The next result is an easy exercise which we leave to the reader.

\begin{lem} \label{Lem5.7}
Let $Y$ be a $d$--dimensional random vector. Suppose that there exist numbers $a>0$ and $\delta>0$ such that for all $\theta\in \mathds{R}^d$ with $|\theta|=1$
\[
\mathbb{P}\big[\langle \theta, Y \rangle \geq a\big] \leq \delta.
\]
Then
\[
\mathbb{P}\big[|Y|\geq a \sqrt{d}\big] \leq 2 d \delta.
\]
\end{lem}

The main step towards the proof of Theorem~\ref{ThLLN} is the following Lemma

\begin{lem} \label{Lem5.9}
Assume that $\beta_j$ ($j=1,\dots,k$) is bounded and that $Y^x$ is a solution of~\eqref{ODE}. Then for all $\theta \in \mathds{R}^d$ with $|\theta|=1$ and all $T>0$, there is a function $\tilde C:\mathds{R}_+\rightarrow \mathds{R}_+$ (independent of $x$) 
such that
\[
\mathbb{P} \Big[\sup_{t\in [0,T]} \Big\{
\langle Z^{N,x}(t)-Y^x(t),\theta \rangle -\int_0^t \sum_{j=1}^k \big(\beta_j(Z^{N,x}(s))- \beta_j(Y^{x}(s))\big) \langle h_j,\theta\rangle ds\Big\} \geq \epsilon\Big] \leq  \exp\big(-N \tilde C(\epsilon) \big),
\]
and moreover
\[
0<\lim_{\eps\rightarrow 0} \tilde C(\eps)/\eps^2 <\infty, \quad \text{and} \quad \lim_{\eps\rightarrow \infty} \tilde C(\eps)/\eps =\infty.
\]
%$\tilde C$ can be chosen independently of the initial value $x$.
\end{lem}
\begin{proof}
Let
\begin{align*}
 \mathcal{N}^{\theta}_t&=
 \langle Z^{N,x}(t)-Y^x(t),\theta \rangle -\int_0^t \sum_{j=1}^k \big(\beta_j(Z^{N,x}(s))- \beta_j(Y^{x}(s))\big) \langle \theta, h_j\rangle ds\\
 &=\frac{1}{N}\sum_{j=1}^k\langle h_j,\theta\rangle M_j\left(N\int_0^t\beta_j(Z^{N,x}_s)ds\right).
 \end{align*}
 We want to use Lemma \ref{Lem5.6}, with $f(Z^{N,x},t)=\mathcal{N}^{\theta}_t$. It is not hard to check that if we define
 \[ G(Z^{N,x},t,\rho)=N\sum_{j=1}^k\left(e^{\frac{\rho}{N}\langle h_j,\theta\rangle}-1-\frac{\rho}{N}\langle h_j,\theta\rangle\right)\int_0^t\beta_j(Z^{N,x}(s))ds,\]
 we have that
 \[ M(t)=\exp\left(\rho f(Z^{N,x},t)-G(Z^{N,x},t,\rho)\right)\]
 is a martingale. Hence from Lemma \ref{Lem5.6}, with $a=\rho/N$,
 \[Ê\P\left(\sup_{0\le t\le T}\mathcal{N}^{\theta}_t>\epsilon\right)\le
 \min_{a>0}\exp\left(N\overline{\beta}T\left[\sum_{j=1}^k\left\{e^{a\langle h_j,\theta\rangle}-1-a\langle h_j,\theta\rangle\right\}-a\epsilon\right]\right).
 \]
 The main inequality of the Lemma is established, with
 \[ \tilde{C}(\eps)=\overline{\beta}T\max_{a>0}\left[a\eps-\sum_{j=1}^k\left\{e^{a\langle h_j,\theta\rangle}-1-a\langle h_j,\theta\rangle\right\}\right].\]
 
 It is not hard to show that as $\eps\to0$,
 \[ \frac{\tilde{C}(\eps)}{\eps^2}\to\frac{\overline{\beta}T}{2\sum_{j=1}^k \langle h_j,\theta\rangle^2}.\]
 Consider now the case where $\eps$ is large. If $\langle h_j,\theta\rangle\le0$ for $1\le j\le k$, then
 for $\eps>-\sum_j \langle h_j,\theta\rangle$, $\tilde{C}(\eps)=+\infty$, which means that a certain event has
 probability zero. Now consider the more interesting case where $\langle h_j,\theta\rangle>0$
 for at least one $1\le j\le k$. If we choose $a_\eps$ such that 
 \[  \sum_{j=1}^k\left\{e^{a_\eps\langle h_j,\theta\rangle}-1-a_\eps\langle h_j,\theta\rangle\right\}=\eps,\]
 then $a_\eps\to\infty$ as $\eps\to\infty$, while $\tilde{C}(\eps)\ge\eps(a_\eps-1)$,  which completes the proof of the Lemma.
  \end{proof}

\begin{proof}[Proof of Theorem~\ref{ThLLN}]
We deduce from Lemma \ref{Lem5.9} and a variant of Lemma \ref{Lem5.7} that
\begin{equation}\label{estim+}
Ê\P\left(\sup_{0\le t\le T}\frac{1}{N}\left|\sum_{j=1}^k h_jM_j\left(N\int_0^t\beta_j(Z^{N,x}_s)ds\right)\right|>\eps\right)\le2d e^{-N\tilde{C}'(\eps)},
\end{equation}
where $\tilde{C}'(\eps)=\tilde{C}(\eps/\sqrt{d})$. In view of the Lipschitz property of $b$, we have
\begin{align*}
Z^{N,x}_t-Y^x_t&=\int_0^t\left[b(Z^{N,x}_s)-b(Y^x_s)\right]ds+\frac{1}{N}\sum_{j=1}^kh_jM_j\left(N\int_0^t\beta_j(Z^{N,x}_s)ds\right), \\
\sup_{0\le s\le t}\left|Z^{N,x}_s-Y^x_s\right|&\le K\int_0^t\sup_{0\le r\le s}\left|Z^{N,x}_r-Y^x_r\right|ds+
\sup_{0\le s\le t}\frac{1}{N}\left|\sum_{j=1}^kh_jM_j\left(N\int_0^s\beta_j(Z^{N,x}_r)dr\right)\right|.
\end{align*}
The result now follows from \eqref{estim+} and Gronwall's Lemma.
\end{proof}

We can deduce  from Theorem~\ref{ThLLN}.

\begin{cor} \label{CorLLN}
Let $M$ be a compensated standard Poisson process. Then there exist constants $C_1=C_1(T)>0$ (independent of $\epsilon$) and $C_2(\epsilon)=C_2(T,\epsilon)>0$ with $C_2(\epsilon)=O(\epsilon^2)$ as $\epsilon \rightarrow 0$ such that
\[
\mathbb{P}\Big [\sup_{t \in [0,T]} \frac{|M(tN)|}{N} \geq \epsilon\Big] \leq C_1\exp\big(- N C_2(\epsilon)\big).
\]
$C_1$ and $C_2$ can be chosen independently of $x$.
\end{cor}

\begin{proof}
We apply Theorem~\ref{ThLLN} to $d=k=1$, $\beta_1(x)\equiv 1$ and $h_1=1$. Hence,
\[
|Z^N(t)-Y(t)|  = \frac{|M(tN)|}{N}.
\]
The result follows directly. \end{proof}

We shall need below the
\begin{lem} \label{Cor5.55}
Let $\beta_j$ ($j=1,\dots,k$) be bounded. Then there exist positive constants $\tilde C_1$ and $\tilde C_2$ independent of $x$ such that for all $0\leq s < t \leq T$ and for all $\epsilon>0$,
\[
\mathbb{P}\Big[ \sup_{r \in [s,t]} |Z^{N,x}(r)-Z^{N,x}(s)|\geq \epsilon \Big] \leq \exp\Big(-N\epsilon\tilde C_1 \log \Big(\frac{\epsilon \tilde C_2}{t-s} \Big)\Big).
\]
\end{lem}

\begin{proof}
Let $\xi^N_{s,t}$ denote the number of jumps of the process $Z^{N,x}$ on the time interval $[s,t]$.  
It is plain that
$$\{\sup_{r\in[s,t]}|Z^{N,x}(r)-Z^{N,x}(s)|\ge\epsilon\}\subset\{\xi^N_{s,t}\ge CN\epsilon\},$$
for some universal constant $C>0$. Now  $\xi^N_{s,t}$ is stochastically dominated by
a Poisson random variable with parameter $C'N(t-s)$, for some other constant $C'>0$.
Now let $\Theta$ be a Poisson r.v. with parameter $\lambda$. For any $a,b>0$,
\begin{align*}
\mathbb{P}(\Theta>b)&=\P\left(e^{a\Theta}>e^{ab}\right)\\
&\le \exp\left(\lambda(e^a-1)-ab\right),
\end{align*}
which, with the optimal choice $a=\log(b/\lambda)$, reads
$$\mathbb{P}(\Theta>b)\le \exp\left(b-\lambda-b\log(b/\lambda)\right)\le \exp\left(-b\log\left(\frac{b}{e\lambda}\right)\right).$$
The result follows by applying this inequality with $\lambda=C'N(t-s)$, and $b=CN\epsilon$.
\end{proof}

\section{Properties of the rate function} \label{SectionRateFunction}

\subsection{Properties of the Legendre Fenchel transform}
In this subsection we assume that the $\beta_j$'s are bounded and continuous. We recall that $\ell$,  $\overline{L}$, $\tilde{\ell}$
and $\underline{L}$ have been defined in section \ref{SubSecDefRateFct}, and start with

\begin{lem} \label{Lemma5.12-5.14}
\begin{enumerate}
\item
For all $x \in A$, $\underline{L}(x,\cdot): \mathcal{C}_x \rightarrow \mathds{R}_+$ is convex and lower semicontinuous.
\item
For all $y \in \mathds{R}^d$,
\[
\underline{L}(x,y)\geq \underline{L}\Big(x, \sum_j \beta_j(x) h_j\Big)=0
\]
with strict inequality if $y \not=\sum_j \beta_j(x) h_j$.
%\item
%
\end{enumerate}
\end{lem}

\begin{proof}
\begin{enumerate}
\item
We readily observe that $\tilde\ell(x,\cdot,\theta)$ is linear and hence convex. As the supremum of these functions, the function $\underline{L}(x,\cdot)$ is convex.

Lower semicontinuity follows as $\underline{L}(x,\cdot)$ is the supremum of a family of continuous functions.
\item
Let first $y=\sum_j\beta_j(x) h_j$. We have
\begin{align*}
\underline{L}(x,y) &=\sup_\theta \Big\{ \sum_j \beta_j(x) \langle h_j, \theta \rangle - \sum_j \beta_j(x) \big(\exp\langle h_j,\theta\rangle - 1 \big) \Big\} \\
&=\sup_\theta \Big\{ \sum_j \beta_j(x)\big( \langle h_j, \theta \rangle -\exp\langle h_j,\theta\rangle + 1 \big) \Big\} \\ 
&=0
\end{align*}
as $\beta_j(x)\geq 0$ and $e^z\geq z+1$ for all $z\in \mathds{R}$ with equality for $z=0$.

Let now $y$ be such that $\underline{L}(x,y)=0$. This implies
\[
\langle y,\theta \rangle - \sum_j \beta_j(x) \big(\exp\langle h_j, \theta\rangle - 1 \big) \leq 0
\quad \text{for all } \theta \in \mathds{R}^d,
\]
in particular for $\theta=\epsilon e_i$ (where $e_i$ is the $i^{\text{th}}$ unit-vector and $\epsilon>0$; in the following $h_j^i$ is the $i^{\text{th}}$ component of $h_j$),
\[
\epsilon y_i \leq \sum_j \beta_j(x) \big(\exp(\epsilon h_j^i) - 1 \big).
\]
Dividing by $\eps$ and letting $\epsilon \rightarrow 0$, we deduce that
\[
y_i\leq \sum_j \beta_j(x) h_j^i.
\]
For $\theta=-\epsilon e_i$ the converse inequality follows accordingly.
%\item  The assertions follow from~2.
\end{enumerate}
\end{proof}

\begin{remark}
The function $\underline{L}(x,\cdot)$ is even strictly convex, see Corollary~\ref{Corollary5.27} below.
\end{remark}

\begin{lem} \label{Lemma5.17}
Assume that $\beta_j$ ($j=1,\dots,k$) is bounded. Then, there exist constants $C_1$ and $B_1$ such that for all $|y|\geq B_1$, 
$x\in\R^d$,
\[
\underline{L}(x,y) \geq C_1 |y| \log(|y|).
\]
\end{lem}

\begin{proof}
Let 
\[
\theta:=y \frac{\log|y|}{\bar h |y|},
\] 
hence provided $|y|\ge1$,
\[
\underline{L}(x,y) \geq \frac{|y| \log |y|}{\bar h}  - k \bar \beta |y|
\]
which grows like $|y|\log|y|$ as $|y| \rightarrow \infty$.
\end{proof}

We now have 

\begin{lem} \label{Lemma5.20}
There exists a constant $C_2$ such that for all $x \in A$, $y \in \mathcal{C}_x$, there exists a $\mu \in V_{x,y}$ with 
\[
|\mu| \leq C_2 |y|.
\] 
\end{lem}

\begin{proof}
We first note that there are only finitely many convex cones $\mathcal C_x$ and we can hence restrict our attention to a fixed $x \in A$.

We proceed by contradiction. Assume that for all $n$ there exists a $y^n\in \mathcal C_x$ such that for all $\mu \in V_{x,y^n}$, 
\[
|\mu|\geq n |y^n|.
\]
We note that for any $y \in \mathcal C_x$, there exists a minimal representation $\mu\in V_{x,y}$ (in the sense that $\tilde \mu \in V_{x,y} \Rightarrow \max_j\tilde \mu_j \geq \max_j \mu_j$). Indeed, let $\{\mu^n,: n\ge1\}\subset V_{x,y}$ be such that, as $n\to\infty$, 
\[
\max_j \mu^n_j\downarrow \inf_{\mu \in V_{x,y}}\left(\max_j \mu_j \right).
\]
There exists a subsequence along which $\mu^n\rightarrow \mu\in \mathds{R}^d_+$ as $n \rightarrow \infty$. If $\mu_j>0$, we have $\mu^n_j>0$ for $n$ large enough and hence $\beta_j(x)>0$. Hence $\mu \in V_{x,y}$ since
moreover
\[
\sum_j \mu_j h_j =  \lim_n \sum_j  \mu^n_j h_j =y.
\] 

Given $y^n$, we denote this minimal representation by $\bar \mu^n$. 
%By normalizing, we can assume without loss of generality that $|y^n|\leq 1$. 
We now define
\[
\tilde y^n:=\frac{y^n}{|\bar \mu^n|}, \quad \text{hence} \quad |\tilde y^n|\leq \frac 1 n.
\]
Furthermore, it is easy to see that minimal representations for the $\tilde y^n$ are given by
\[
\tilde \mu^n:=\frac{\bar \mu^n}{|\bar \mu^n|}, \quad \text{hence} \quad |\tilde \mu^n|=1.
\]
Boundedness implies (after possibly the extraction of a subsequence) $\tilde \mu^n \rightarrow \tilde \mu$ with $|\tilde \mu|=1$. We let $n$ large enough such that for all $j$
\[
\tilde \mu^n_j>0 \Rightarrow  \tilde \mu^n_j>\frac{\tilde \mu_j}{2}
\]
(note that for at least one $j$, $\tilde\mu_j>0$).
We have
\[
0=\lim_n \tilde y^n=\lim_n \sum_j \tilde\mu^n_j h_j =\sum_j \tilde \mu_j h_j
\]
and therefore
\[
\tilde y^n=\sum_{j} \tilde \mu_j^n h_j=\sum_{j; \tilde \mu_j>0} \underbrace{\Big(\tilde \mu_j^n-\frac{\tilde \mu_j}{2} \Big)}_{=:\hat \mu^n_j>0} h_j + \sum_{j; \tilde \mu_j=0} \underbrace{\tilde \mu_j^n}_{=:\hat \mu^n_j}h_j,
\]
a contradiction to the minimality of the $\tilde \mu^n_j$.
\end{proof}

We require the following result 
\begin{lemma} \label{Lemma3.3}
Let $x \in A$.
\begin{enumerate}
\item[1.]
$\ell(x,\mu) \geq 0$ for $\mu\in V_x$ and 
$\ell(x,\cdot): V_x \longrightarrow\mathds{R}_+$ is strictly convex and has compact level sets $\{\mu\in V_x| \ell (x,\mu) \leq \alpha\}$.
\item[2.]
Let $y\in \mathcal{C}_x$. Then there exists a unique $\mu^*=\mu^*(y)$ such that
\[
\ell(x,\mu^*)=\inf_{\mu\in V_{x,y}} \ell(x,\mu).
\]
\item[3.]
There exist constants 
$C_3,C_4,C_5,B_2>0$ (which depend only upon $\sup_{x\in A}\max_{i\le j\le k}\beta_j(x)$), such that
\begin{align}
|\mu^*(y)| & \leq C_3 |y| \quad \text{if } |y| > B_2, \label{IneqLemma3.3.1} \\ 
|\mu^*(y)| & \leq C_4 \quad \text{if } |y| \leq B_2, \label{IneqLemma3.3.2} \\
|\mu^*(y)| & \geq C_5 |y| \quad \text{for all } y. \label{IneqLemma3.3.3} 
\end{align}
\item[4.]
$\overline{L}(x,\cdot),\mu^*:\mathcal{C}_x \rightarrow \mathds{R}_+$ are continuous.
\end{enumerate}
\end{lemma}

\begin{proof}
\begin{enumerate}
\item[1.] 
We define the function $f(z)=1-z+z\log z$ for $z \geq 0$ and note that for $\mu \in V_x$,
\[
\ell(x,\mu)=\sum_{j, \beta_j(x)>0} \beta_j(x) f\Big(\tfrac{\mu_j}{\beta_j(x)}\Big). 
\]
We readily observe (by differentiation) that $f\geq 0$ and that $f$ is strictly convex. Thus the first two assertions follow.

As $V_x$ is closed and $\ell(x,\cdot)$ is continuous, the level sets are closed. Compactness follows form the fact that $\lim_{x\rightarrow \infty}f(x)=\infty$.
\item[2.]
Existence of a minimizer follows from the fact that $V_{x,y}$ is closed. Uniqueness follows from the strict convexity of $ \ell(x,\cdot)$. 
\item[3.]
By the definition of $ \ell$, there exists a $B_2=B_2(\bar\beta(x))>0$ and $C=C(\bar \beta(x))>0$ such that for $y\in \mathcal C_x$ with $|y|\geq B_2$ (and appropriate $\mu\in V_{x,y}$ according to Lemma~\ref{Lemma5.20}),
\[
 \ell(x,\mu^*(y))\leq  \ell(x,\mu) \leq C |y| \log |y|.
\]
On the other hand, assume that for all $n$ there exists an $y^n\in \mathcal{C}_x$ with $|y^n|\geq B_2$ such that
\[
|\mu^*(y^n)| \geq n |y^n|.
\]
This implies for an appropriate constant $\tilde C$ and $n$ large enough
\[
\ell(x,\mu^*(y^n))\geq n\tilde C |y| \log |y|,
\]
a contradiction. Hence Inequality~\eqref{IneqLemma3.3.1} follows.

Assume now that for all $n$ there exists an $y^n\in \mathcal C_x$ with $|y^n|\leq B_2$,
\[
|\mu^*(y^n)| \geq n, \quad \text{hence }\lim_{n\rightarrow \infty}  \ell(x,\mu^*(y^n)) \rightarrow \infty.
\]
However, Lemma~\ref{Lemma5.20} implies that there exists an $\mu^n\in V_{x,y^n}$ and a constant $C=C(\bar \beta (x),B_2)$ independent of $n$ with
\[
\ell (x,\mu^n) \leq C,
\]
a contradiction. Hence Inequality~\eqref{IneqLemma3.3.2} follows.

Finally, Inequality~\eqref{IneqLemma3.3.3} follows from the definition of $V_{x,y}$.

\item[4.]
Let $y,y^n\in \mathcal{C}_x$ with $y^n\rightarrow y$. By 3., the sequence $(\mu^*(y^n))_n$ is bounded and hence there exists a convergent subsequence, say (by abuse of notation)
\[
\mu^*(y^n) \rightarrow \mu^* \quad \text{with } \mu_j^*\geq 0\text{ for all } j.
\]
In particular, we have
\begin{equation}\label{ident}
\sum_j \mu^*_j h_j =y.
\end{equation}
We have
\begin{align}
y^n&=\sum_j \mu_j^*(y^n) h_j \notag \\
&=(1- \epsilon^n) \sum_j \mu_j^* h_j + \sum_j \big(\mu_j^*(y^n) - \mu_j^* + \epsilon^n \mu_j^* \big) h_j \notag\\
&=(1- \epsilon^n) \sum_j \mu_j^*(y) h_j + \sum_j \tilde{\mu}_j^n h_j, \label{EqLemma3.3.1}
\end{align}
where we have used \eqref{ident}, $\mu^\ast(y)=\argmax_\mu\ell(x,\mu)$, $\tilde{\mu}_j^n=\mu_j^*(y^n) - \mu_j^* + \epsilon^n \mu_j^*$, and
\[
\epsilon^n=\begin{cases} \frac{2 \max_j |\mu_j^*(y^n) - \mu_j^*|}{\min_{j; \mu_j^*>0} \mu_j^*},&
\text{if $\min_{j; \mu_j^*>0} \mu_j^*>0$};\\
1/n,&\text{otherwise}.\end{cases}
\]
In particular, we have $0 \leq \tilde \mu_j^n\rightarrow 0$ as $n \rightarrow \infty$.
By Equation~\eqref{EqLemma3.3.1}, 2.~and the continuity of $ \ell$, we have
\begin{align}
 \ell (x,\mu^*(y^n))) &\leq   \ell (x,(1-\epsilon^n) \mu^*(y) + \tilde \mu^n) \notag\\
&\leq  \ell (x,\mu^*(y)) + \delta(n)
\end{align}
with $\delta(n)\rightarrow \infty$ as $n \rightarrow \infty$. This implies (again by the continuity of $ \ell$)
\[
 \ell (x,\mu^*) \leq  \ell (x,\mu^*(y))
\]
and hence $\mu^*=\mu^*(y)$ by 2. As this holds true for all convergent subsequences of $(\mu^*(y^n))_n$, this establishes the continuity of $\mu^*(\cdot)$.

The continuity of $\overline{L}(x,\cdot)$ follows directly from this and the continuity of $ \ell$.
\end{enumerate}
\end{proof}

\begin{remark}
Assume that for $x\in A$, $\mathcal{C}_x=\mathcal C_{\tilde x}$ for all $\tilde x$ in some neighborhood $U$ of $x$. Then the function $\ell:  U\times V_{x,y}  \rightarrow \mathds{R}_+$ is continuous and hence we have that $\mu^*(y)=\mu^*(x,y)$ as given in Lemma~\ref{Lemma3.3} is also continuous in $x$ (as the $\argmin$ of a continuous function).
\end{remark}
We have moreover

\begin{lem} \label{Lemma5.21}
\begin{enumerate}
\item
Let $x \in A$. For all $B>0$, there exists a constant $C_6=C_6(x,B)>0$ such that for all $y \in \mathcal{C}_x$ with $|y| \leq B$ and $\theta \in \mathds{R}^d$ with $\tilde\ell(x,y,\theta)\geq -1$,\footnote{The constant $-1$ can be replaced by any other constant $-C$ ($C>0$). Note that $C_6$ then depends on $C$ with $C_6$ increasing in $C$.}
\[
\langle \theta, h_j \rangle \leq C_6 \quad \text{for all } j \text{ with } \beta_j(x)>0.
\]
If $\log\beta_j(\cdot)$ ($j=1\dots,k$) is bounded, $C_6$ can be chosen independently of $x$.
\item
Let $x \in A$ and $y \in \mathcal{C}_x$. If $(\theta_n)_n$ is a maximizing sequence of $\tilde\ell(x,y,\cdot)$ and for some $j=1,\dots,k$, 
\[
\liminf_{n\rightarrow \infty} \langle \theta_n,h_j \rangle = -\infty,
\] 
then
\[
\mu_j=0 \quad \text{for all } \mu \in V_{x,y}.
\]
Conversely, there exists a constant $\tilde C_6=\tilde C_6(B)>0$ such that if $|y|\leq B$ and $\mu_j>0$ for some $\mu \in V_{x,y}$, then 
\[
\liminf_{n\rightarrow \infty} \langle \theta_n,h_j \rangle > -\tilde C_6.
\] 
\end{enumerate}
\end{lem}

\begin{proof}
\begin{enumerate}
\item
Let $|y| \leq B$, $C_2$ and $\mu\in V_{x,y}$ be according to Lemma~\ref{Lemma5.20}. Define the functions from $\R$ into itself
\[
f_j(z):=\mu_j z - \beta_j(x) (\e^z-1).
\] 
Note that $f_j(z)=0$ if $\beta_j(x)=0$, and $\argmax_z f_j(z)=\log \mu_j/\beta_j(x)$ if $\beta_j(x)>0$. 
Let
\[Ê\tilde{C}(x,B)=\sup_{j;\ \beta_j(x)>0}\sup_{|\mu|\le C_2B}f_j\left(\log\frac{\mu_j}{\beta_j(x)}\right).\]
If $x$, $y$ and $\theta$ are as in the statement, and $1\le j\le k$ is such that $\beta_j(x)>0$ and $\langle \theta,h_j\rangle>0$, then
\[ \sum_{j'\not=j}f_{j'}(\langle\theta,h_{j'}\rangle)=\tilde{\ell}(x,y,\theta)-f_j(\langle\theta,h_j\rangle),\]
hence in view of the assumption,
\[f_j(\langle\theta,h_j\rangle)\ge-1-(k-1)\tilde{C}(x,B),\]
As $f_j(z) \rightarrow -\infty$ as $z \rightarrow \infty$, the assertion follows. 
\item
If $\liminf_{n\rightarrow \infty} \langle \theta_n,h_{j} \rangle =-\infty$ and $\mu \in V_{x,y}$ with $\mu_{j}>0$, then 1.~implies that 
$\tilde\ell(x,y,\theta_n)\rightarrow -\infty$, a contradiction.

The second assertion follows accordingly.
\end{enumerate}
\end{proof}

We now prove

\begin{lemma} \label{LemmaRemark5.21}
\begin{enumerate}
\item
Let $x \in A$ and $y \in \mathcal{C}_x$. Then there exists a maximizing sequence $(\theta_n)_n$ of $\tilde\ell(x,y,\cdot)$ and constants $\tilde s_j$ (for all $j=1,\dots,k$ for which there exists a $\mu\in V_{x,y}$ with $\mu_j>0$) such that
\[
\lim_{n\rightarrow \infty} \langle \theta_n,h_j \rangle = \tilde s_j \in \mathds{R}.
\]
The constants $\tilde s_j$ are bounded uniformly over bounded sets of $y\in \mathcal C_x$. 

In particular, there exists a maximizing sequence $(\theta_n)_n$ such that for all $j=1,\dots,k$ with $\beta_j(x)>0$, 
\[
\lim_{n\rightarrow \infty} \exp(\langle \theta_n,h_j \rangle) =  s_j \in \mathds{R}.\footnotemark
\] 
\footnotetext{Note that here, we also include those $j$ with $\beta_j(x)>0$ and $\mu_j=0$ for all $\mu\in V_{x,y}$.}
\item
Let $x\in A$ and $y \not \in \mathcal{C}_x$. Then $\underline{L}(x,y)=\infty$.
\end{enumerate}
\end{lemma}

\begin{proof}
\begin{enumerate}
\item
By Lemma~\ref{Lemma5.21}, 
\[
-\tilde C_6=-\tilde C_6(|y|) <\langle \theta_n,h_j \rangle \leq C_6=C_6(|y|)
\] 
for all $n$ and for all $j$ with $\mu_j>0$ for some $\mu \in V_{x,y}$. The first assertion follows by taking appropriate subsequences.

For the second assertion, we have to consider those $j$ with $\mu_j=0$ for all $\mu\in V_{x,y}$ although $\beta_j(x)>0$. 
If $\liminf_{n\rightarrow \infty} \langle \theta_n,h_j \rangle=- \infty$,
we take further subsequences  and obtain (with a slight abuse of notation) 
\[
\lim_{n\rightarrow \infty} \exp(\langle \theta_n,h_j \rangle)=0.
\]
\item
Let $y \not\in \mathcal C_x$ and $v$ be the projection of $y$ on $\mathcal C_x$. Hence, $0=\langle y-v,v \rangle \geq \langle y- v,\tilde v\rangle$ for all $\tilde v \in \mathcal C_x$. For $z=y-v$ ($\not=0$ as $y \not \in \mathcal C_x$), we have $\langle z,y\rangle=\langle z,z\rangle+\langle z,v\rangle>0$ and $\langle z,h_j\rangle\leq 0$ for all $j$ with $\beta_j(x)>0$. If we set $\theta_n=nz$, we obtain $\ell(x,y,\theta_n) \rightarrow \infty$.
\end{enumerate}
\end{proof}

\color{black}

\subsection{Equality of $\underline{L}$ and $\overline{L}$}

We can now finally establish

\begin{theorem} \label{Theorem5.26}
For all $x \in A$, $y \in \mathds{R}^d$,
\[
\underline{L}(x,y)=\overline{L}(x,y).
\]
\end{theorem}

\begin{proof}
In view of Lemma \ref{Lemma5.25}, it suffices to prove that $\overline{L}(x,y\le\underline{L}(x,y)$.
We first note that we have $\underline{L}(x,y)<\infty$ if and only if $y\in \mathcal{C}_x$ by Lemma~\ref{LemmaRemark5.21}~2.~and~Lemma~\ref{Lemma5.22}. As the same is true for $ \overline{L}(x,y)$ by definition, we can restrict our attention to the case $y \in \mathcal{C}_x$.

We choose a maximizing sequence $(\theta_n)_n$ according to Lemma~\ref{Lemma5.21}~and obtain
\begin{equation} \label{EqTheorem5.26.1}
\lim_n \langle \theta_n,y\rangle = \underline{L}(x,y) + \sum_{j} \beta_j(x) (s_j - 1);
\end{equation}
here we set $s_j=0$ if $\beta_j(x)=0$. We now differentiate with respect to $\theta$ and obtain for all $n$
\[
\nabla_{\theta} \tilde{\ell}(x,y,\theta_n) =y - \sum_{j;\beta_j(x)>0} \beta_j(x) h_j \exp(\langle \theta_n,h_j\rangle);
\]
hence (by the fact that $(\theta_n)_n$ is a maximizing sequence and the limit of $\nabla_{\theta} \tilde{\ell}(x,y,\theta_n)$ exists),
\[
\lim_n \nabla_{\theta}\tilde{\ell}(x,y,\theta_n) =y - \sum_{j;\beta_j(x)>0} \beta_j(x) s_j h_j =0.
\]
We set,
\begin{equation} \label{EqTheorem5.26.2}
\mu_j^*:=\beta_j(x) s_j,
\end{equation}
in particular 
\[
y=\sum_j \mu_j^* h_j \quad \text{and} \quad \mu^*\in V_{x,y}.
\]
Therefore,
\begin{align*}
 \overline{L}(x,y) &\leq  \ell(x,\mu^*) \\
&=\sum_{j} \beta_j(x) - \mu^*_j + \mu_j^* \log\big(\tfrac{\mu_j^*}{\beta_j(x)}\big) \\
&=\sum_{j} \beta_j(x) (1- s_j) +  \mu_j^* \log s_j \\
&= \underline{L}(x,y),
\end{align*}
where we have used \eqref{EqTheorem5.26.1} and \eqref{EqTheorem5.26.2} for the last identity.
The assertion follows.
\end{proof}

From now on, we shall write $L(x,y)$ for the quantity $\underline{L}(x,y)=\overline{L}(x,y)$.

We now prove the strict convexity of $L(x,\cdot)$. 
\begin{cor} \label{Corollary5.27}
For all $x \in A$, $L(x,\cdot): \mathcal{C}_x \rightarrow \mathds{R}_+$ is strictly convex.
\end{cor}

\begin{proof}
For strict convexity, we exclude the case that $\beta_j(x)=0$ for all $j$ (as then $L(x,y)=\infty$ for all $y\not=0$ and the assertion is trivial). 

Convexity was proven in Lemma~\ref{Lemma5.12-5.14}. Assume now that for $y,\tilde y\in \mathcal C_x$ and $\lambda\in (0,1)$,
\[
L(x,\lambda y + (1- \lambda) \tilde y )= \lambda L(x,y)+(1-\lambda)L(x,\tilde y).
\]
In other words,
\begin{align*}
&\sup_\theta \Big \{
\langle \theta, \lambda y +(1-\lambda) \tilde y\rangle - \sum_j \beta_j(x) (\e^{\langle \theta,h_j\rangle}-1)
\Big\} \\*
&\qquad =\sup_\theta \Big \{ \lambda \big[
\langle \theta, y\rangle - \sum_j \beta_j(x) (\e^{\langle \theta,h_j\rangle}-1)\big]
+ (1- \lambda) \big[\langle \theta, \tilde y \rangle - \sum_j \beta_j(x) (\e^{\langle \theta,h_j\rangle}-1)\big]
\Big\} \\
&\qquad =\lambda \sup_\theta \Big \{
\langle \theta, y \rangle - \sum_j \beta_j(x) (\e^{\langle \theta,h_j\rangle}-1)
\Big\} +
(1-\lambda) \sup_\theta \Big \{
\langle \theta, \tilde y \rangle - \sum_j \beta_j(x) (\e^{\langle \theta,h_j\rangle}-1)
\Big\}.
\end{align*}
Hence, if $(\theta_n)_n$ is a maximizing sequence for $\tilde\ell(x,\lambda y + (1-\lambda) \tilde y,\cdot)$, it is also a maximizing sequence for $\tilde\ell(x, y,\cdot)$ and $\tilde\ell(x,\tilde y,\cdot)$. As in the proof of Theorem~\ref{Theorem5.26}, this implies
\begin{align*}
\lim_{n\rightarrow \infty} \nabla_\theta\tilde \ell(x,y,\theta_n)&=y-\lim_{n\rightarrow \infty} \sum_j\beta_j(x) (\e^{\langle \theta_n,h_j\rangle}-1)=0, \\
\lim_{n\rightarrow \infty} \nabla_\theta\tilde \ell(x,\tilde y,\theta_n)
&=\tilde y-\lim_{n\rightarrow \infty} \sum_j \beta_j(x) (\e^{\langle \theta_n,h_j\rangle}-1)=0.
\end{align*}
Hence $y=\tilde y$ as required.
\end{proof}

\subsection{Further properties of the Legendre Fenchel transform}

In this subjection, we assume that the $\log \beta_j$'s are bounded. 
In this case $\mathcal{C}_x=\mathcal{C}=\R^d$ for all $x$.
%{\color{red} $=\R^d$ ??? SI OUI REVOIR TOUTES LES OCCURENCES DE $\mathcal{C}$}

We have
\begin{lem} \label{Lemma5.22}
\begin{enumerate}
Assume that $\log \beta_j$ ($j=1,\dots,k$) is bounded.
\item
For all $B>0$ exists a constant $C_7=C_7(B)>0$ such that for all $x \in A$, $y \in \mathcal{C}$ with $|y| \leq B$,
\[
L(x,y) \leq C_7.
\]
\item
For all $x \in A$, $L(x, \cdot): \mathcal{C} \rightarrow \mathds{R}_+$ is continuous.
\end{enumerate}
\end{lem}

\begin{proof}
\begin{enumerate}
\item
Let $x \in A$, $y \in \mathcal{C}$. By Lemma~\ref{Lemma5.20} and Theorem~\ref{Theorem5.26} below\footnote{note that this result is not used for the proof of Theorem~\ref{Theorem5.26}.}, we obtain
\begin{align*}
L(x,y) &\leq \sum_{j,\beta_j(x) >0} \beta_j(x)- \mu_j + \mu_j \log \mu_j - \mu_j \log \beta_j(x) \\
	&\leq k (\overline{\beta} +  C |y| \log C +  C |y| \log|y|+   C |y| |\log\underline{\beta}|).
\end{align*}
The assertion follows.
\item
The assertion follows directly from 1., Lemma~\ref{Lemma5.12-5.14}~1.
\end{enumerate}

\end{proof}

\color{black}

We have moreover

\begin{lem} \label{Lemma5.23}
Assume that $\log\beta_j$ ($j=1,\dots,k$) is bounded.
For all $\rho>0$, $\epsilon>0$, $C_8>0$, there exists a constant $B_3=B(C_8,\epsilon)$ such that for all $x \in A$,  $y \in \mathcal{C}$ with $|y|\leq C_8$,
\[
\sup_{|\theta|\leq B} \tilde\ell(x,y,\theta) \geq \sup_{\theta\in \mathds{R}^d} \tilde\ell(x,y,\theta)-\epsilon
=L(x,y)-\epsilon.
\]
\end{lem}

\begin{proof}
We first fix $x \in A$ and define the compact set
\[
\tilde{\mathcal{C}}:=\{y \in \mathcal{C}||y| \leq C_8\}.
\]
We fix $\delta >0$ and define for $y \in \tilde{\mathcal{C}}$,
\begin{align*}
z(y,\delta)&:=y + \sum_{j} \delta h_j, \\
N^{y,\delta}&:=\big\{ y+ \sum_{j}  \alpha_j h_j | \alpha_j \in (-\delta,  \delta) \big\}.
\end{align*}
For all $y \in \tilde{\mathcal{C}}$, $N^{y,\delta}$ is relatively open (with respect to $\mathcal{C}$) and $y\in N^{y,\delta}$. Hence there exists a finite cover $N_1, \dots, N_n$ of $\tilde{\mathcal{C}}$, where $N_i:=N^{y_i,\delta}$ for appropriate $y_i \in \tilde{\mathcal{C}}$; we define $z_i:=z(y_i,\delta)$.

We use the continuity of $L(x,\cdot): \mathcal{C}\rightarrow \mathds{R}_+$ (cf.~Lemma~\ref{Lemma5.22}~2.) and the fact that $\tilde{\mathcal{C}}$ is compact, we obtain for $\delta$ small enough that for all $y \in \tilde{\mathcal{C}}$, $v \in N^{y,\delta}$,
\begin{equation} \label{EqLemma5.23.1}
|L(x,v)-L(x,z(y,\delta))| < \frac{\epsilon}{4}.
\end{equation}
%(where $C_6$ is chosen according to Lemma~\ref{Lemma5.21} for $B=C_8$)
%and the corresponding cover $(N_i)_i$ (with corresponding $y_i$'s and $z_i$'s) for this $\delta$. 
We let $\theta_i$ be almost optimal for $z_i$ in the sense that
\begin{equation} \label{EqLemma5.23.2}
\tilde{\ell}(x,z_i,\theta_i) \geq L(x,z_i)-\frac{\epsilon}{4}.
\end{equation}
We now set $B^{x}:=\max_i |\theta_i|$ and let $y \in \tilde{\mathcal{C}}$, say $y \in N_i$. Then, making use successively of
\eqref{EqLemma5.23.1} and \eqref{EqLemma5.23.2}, we obtain
\begin{align}
L(x,y) &\leq L(x,z_i) + \frac{\epsilon}{4} \notag \\
&\leq \tilde\ell(x,z_i,\theta_i) + \frac{\epsilon}{2} \notag \\
&=\tilde\ell(x,y,\theta_i)+  \frac{\epsilon}{2} + \langle \theta_i, z_i-y \rangle. \label{EqLemma5.23.4}
\end{align}
We have $z_i-y=\sum_{j} \mu_j h_j$ for appropriate $\mu_j=\alpha_j+\delta \in (0,2\delta)$ and by Lemma~\ref{Lemma5.21} (cf.~also Inequality~\eqref{EqLemma5.23.1}), $\langle \theta_i, h_j \rangle \leq C_6$. Hence 
\begin{equation} \label{EqLemma5.23.5}
\langle z_i-y, \theta_i\rangle =\sum_{j}  \mu_j \langle h_j,\theta_i \rangle \leq 2 k C_6 \delta \leq \frac{\epsilon}{4},
\end{equation}
provided we choose $\delta$ such that $8 k C_6 \delta \leq \epsilon$.
Therefore by Inequalities~\eqref{EqLemma5.23.4} and~\eqref{EqLemma5.23.5} for all $y \in \tilde{\mathcal{C}}$,
\begin{equation}
\label{EqLemma5.23.6}
L(x,y) \leq \tilde\ell(x,y,\theta_i) + \epsilon \quad \text{(recall that $|\theta_i| \leq B^{x}$)}.
\end{equation}
%\\
%{\color{red} REVOIR CETTE PARTIE}

Let now for all $x \in A$, $B^{x}$ be the bound obtained above belonging to $\frac{\epsilon}{4}$.\footnote{Note that $B^x$ depends on $x$ only through $\beta(x)$.\label{Foot5.23}} Let furthermore $x,\tilde x \in A$ with $|\beta(x)-\beta(\tilde x)| < \delta$ for some $\delta>0$, $y \in \tilde{\mathcal{C}}$, $|y| \leq C_8$ and $\theta\in \mathds{R}^d$ such that $\tilde\ell(x,y,\theta)\geq -1$ (which implies $\langle \theta, h_j \rangle \leq C_6$ by Lemma~\ref{Lemma5.21}). This implies
\begin{equation}\label{EqLemma5.23.7}
|\tilde\ell(x,y,\theta)-\tilde\ell(\tilde x,y,\theta) |  \leq \sum_{j} | \beta_j(x) - \beta_j(\tilde x)| \e^{C_6} < \frac{\epsilon}{4} 
\end{equation}
for $\delta$ small enough (and independent of $x, \tilde x,y, \theta$). Let now be $\tilde \theta$ be almost optimal for $\tilde x,y$.
Using twice \eqref{EqLemma5.23.7} and once \eqref{EqLemma5.23.6}, we obtain
\begin{align}
L(\tilde x,y) & \leq \tilde\ell (\tilde x,y,\tilde \theta) + \frac{\epsilon}{4} \notag \\
&\leq \tilde\ell(x,y,\tilde \theta) + \frac{\epsilon}{2} \notag \\
&\leq \sup_{|\theta| \leq B^{x}} \tilde\ell( x,y,\theta) + \frac{3 \epsilon}{4} \notag \\
&\leq \sup_{|\theta| \leq B^{x}} \tilde\ell(  \tilde x,y,\theta)  + \epsilon \notag\end{align}
We can cover the compact interval $[\underline \beta,\bar \beta]$ by finitely many $\tilde\delta$-neighborhoods of $\beta^i$. The assertion follows by taking the maximum of the corresponding $B^{i}$ (cf.~Footnote~\ref{Foot5.23}).  
\end{proof}

\begin{lem} \label{Lemma5.32}
Assume that $\log \beta_j$ ($j=1,\dots,k$) is bounded.
There exist constants $B_4$ and $C_9$ such that for all, $x \in A$ and $y \in \mathcal{C}$,
\[
L(x,y) \leq 
\begin{cases}
C_9  &\text{if } |y| \leq B_4 \\
C_9 |y|  \log |y| & \text{if } |y| > B_4.
\end{cases}
\]
\end{lem}

\begin{proof}
From the formula for $\overline{L}(x,y)$ and Lemma~\ref{Lemma5.20}, we have
\begin{align*}
L(x,y)
&\leq  \sum_j \bar \beta + C |y| \log|y| + C|y| |\log \underline \beta| \\
&\leq k \cdot \big( \bar \beta + C |y| \log|y| + C|y| |\log \underline \beta|\big).
\end{align*}
\end{proof}

We also obtain the continuity of $L$ in $x$ 

\begin{lem} \label{Lemma5.33}
Assume that $\log \beta_j$ ($j=1,\dots,k$) is bounded and continuous.
For all $y \in \mathcal C$,
\[
L(\cdot,y) : A \rightarrow \mathds{R}_+
\]
is continuous. The continuity is uniform over bounded $y$. 
%More precisely, we can quantify the continuity as follows. For $\alpha>0$, $\epsilon>0$, there exist constant $\delta=\delta(\alpha,\epsilon)$ (independent of $x$ and $y$) such that for all $\rho$, $x, \tilde x \in C_\rho$ with
%\[
%|x-\tilde x| < \delta(\alpha, \epsilon) |y|^{-\alpha} \rho^{2+\alpha}
%\]
%and all $y \in \mathds{R}^d$,
%\[
%|L(x,y)-L(\tilde x,y)| < \epsilon.
%\]
\end{lem}

\begin{proof}

We let $y \in \mathcal{C}$ with $|y|\leq B$, $0<\epsilon<1$. and $x, \tilde x \in A$. Let $\theta$ such that 
\[
L(x,y) \leq \tilde\ell(x,y,\theta) + \frac{\epsilon}{2}.
\]
We have by the continuity of the $\beta_j$ and Lemma~\ref{Lemma5.21},
\[
|\tilde\ell(x,y,\theta)-\tilde\ell(\tilde x,y,\theta)| \leq \sum_j |\beta_j(x)- \beta_j(\tilde x)| e^{C_6}<\frac{\epsilon}{2}
\]
if $|x-\tilde x|< \delta$ for appropriate $\delta>0$ (independent of $x,\tilde x\in A$ and $y$ with $|y|\leq B$). Thus,
\[
L(x,y) \leq \tilde\ell(\tilde x,y,\theta) + \epsilon \leq L(\tilde x ,y)+\epsilon.
\]
Reversing the roles of $x$ and $\tilde x$ proves the assertion.
\end{proof}
Combinig Lemma \ref{Lemma5.33} and Lemma \ref{Lemma5.22}, we deduce the 
\begin{cor}
Assume that $\log \beta_j$ ($j=1,\dots,k$) is bounded and continuous.
Then $L : A\times\R^d\to\R_+$ is continuous.
\end{cor}

\subsection{The rate function}
Recall that
for $\phi:[0,T]\rightarrow A$, we let
\begin{align*}
I_T(\phi)&:= \begin{cases}
\int_0^T L(\phi(t),\phi'(t))  dt& \text{ if } \phi \text{ is absolutely continuous} \\
\infty & \text{ otherwise.}
\end{cases}
\end{align*}
For $x\in A$ and $\phi:[0,T]\rightarrow A$, let
\begin{equation*}
I_{T,x}(\phi):=
\begin{cases}
I_{T}(\phi)& \text{ if } \phi(0)=x \\
\infty & \text{ otherwise.}
\end{cases}
\end{equation*}

We  first have the following statement, which follows readily from
point 2 in Lemma \ref{Lemma5.12-5.14}.
\begin{lem}
Assume that $x \in A$. If $\phi$ solves the ODE~\eqref{ODE}, then $I_{T,x}(\phi)=0$. Conversely, if the ODE~\eqref{ODE} admits a unique solution $Y^x$ and
$I_{T,x}(\phi)=0$, then $\phi(t)=Y^x(t)$ for all $t\in [0,T]$.
\end{lem}

In the next statement, $B_1$ refers to the constant appearing in Lemma \ref{Lemma5.17}.
\begin{lem}\label{Lemma5.18}
Assume that $\beta_j$ ($j=1,\dots,k$) is bounded.
\begin{enumerate}
\item
Let $K, \epsilon>0$. There exits $\delta>0$ such that for all $\phi$ with $I_{T,x}(\phi) \leq K$ and for all finite collections of non-overlapping subintervals of $[0,T]$, $[s_1,t_1], \dots, [s_J,t_J]$, with $\sum_i (t_i-s_i)=\delta$,
\[
\sum_i |\phi(t_i)-\phi(s_i)| < \epsilon.
\]
\item
Let $K>0$. Then, for all constants $B\geq B_1$  and for all $\phi$ with $I_{T,x}(\phi) \leq K$,
\[
\int_0^T \mathds{1}_{\{|\phi'(t)| \geq B\}} dt <  \frac{K}{C_1 B \log B}.
\]
\end{enumerate}
\end{lem}

\begin{proof}
\begin{enumerate}
\item
Note first that
\[
f(\alpha):=\inf_{x,y} \Big\{\frac{L(x,y)}{|y|} \Big| |y|\geq \alpha\Big\} \rightarrow \infty
\]
as $\alpha \rightarrow \infty$ by Lemma~\ref{Lemma5.17}. For $g(t):=\mathds{1}_{\cup_j [s_j,t_j]}$ and $\alpha = 1/\sqrt{\delta}$, we obtain by the fact that $\phi$ is absolutely continuous
\begin{align*}
\sum_j |\phi(t_j) - \phi(s_j)| &\leq \int_0^T |\phi'(t)| g(t) dt \\
& \leq \int_0^T \alpha \mathds{1}_{\{|\phi'(t)| \leq \alpha\}}g(t) dt + \int_0^T  \mathds{1}_{\{|\phi'(t)| > \alpha\}} \frac{L(\phi(t),\phi'(t))}{L(\phi(t),\phi'(t))/|\phi'(t)|} g(t) dt \\
&\leq \alpha \delta + \frac{I_{T,x}(\phi)}{f(\alpha)}\\ &\leq  \sqrt{\delta} + \frac{K}{f(1/\sqrt{\delta})} \rightarrow 0,
\end{align*}
as $\delta \rightarrow 0$.
\item
For $B>0$, we define the function
\[
f(B):=\inf_{x, y} \Big\{ \frac{L(x,y)}{|y|} \, \Big| \, |y| \geq B \Big\}.
\]
By Lemma~\ref{Lemma5.17}, $f(B) \geq C_1 \log B$ for $B \geq B_1$.
\begin{align*}
\int_0^T \mathds{1}_{\{|\phi'(t)| \geq B\}} dt &\leq \frac{1}{B} \int_0^T |\phi'(t)| \mathds{1}_{\{|\phi'(t)| \geq B\}} dt \\
&=\frac{1}{B}\int_0^T \frac{L(\phi(t),\phi'(t)) |\phi'(t)|}{L(\phi(t),\phi'(t))} \mathds{1}_{\{|\phi'(t)| \geq B\}} dt \\
& \leq \frac{K}{B f(B)} \leq \frac{K}{C_1 B \log B}. 
\end{align*}
\end{enumerate}
\end{proof}

\begin{theorem} \label{Theorem5.35}
Assume that $\log \beta_j$ ($j=1,\dots,k$) is bounded and continuous. Let $\phi \in D([0,T];A)$ with $I_{T,x}(\phi)< \infty$. For all $\epsilon >0$, there exists a $\delta>0$ such that for 
\[
\tilde \phi:[0,T] \rightarrow A \quad \text{with} \quad
\sup_{0 \leq t \leq T} |\tilde \phi(t) - \phi(t)| < \delta,
\]
\[
\Big |\int_0^T \big(L(\tilde \phi(t), \phi'(t)) -  L(\phi(t), \phi'(t))\big) dt\Big| < \epsilon.
\]
\end{theorem}

\begin{proof}
We  choose $B\geq B_1\vee B_4$ large enough such that for $x \in A$, $y \in \mathcal{C}_x=\mathcal C$ (independent of $x$) with $|y|\geq B$ (cf.~Lemmas~\ref{Lemma5.17} and~\ref{Lemma5.32}),
\[
C_1 |y| \log |y| \leq L(x,y) \leq C_9 |y| \log |y|.
\]
As $I_{T,x}(\phi)<\infty$, the set $\{t|\phi'(t) \not\in \mathcal{C}\}$ is a Lebesgue null-set and we assume w.l.o.g. that for all $t$, $\phi'(t) \in \mathcal{C}$. We hence obtain that
\begin{align*}
 \int_0^T L(\tilde \phi(t),\phi'(t)) \mathds{1}_{\{|\phi'(t)| \geq B\}}  dt &\leq \int_0^T C_9 |\phi'(t)| \log|\phi'(t)| \mathds{1}_{\{|\phi'(t)| \geq B\}}  dt \\
 &\leq \frac{C_9}{C_1} \int_0^T L(\phi(t),\phi'(t)) \mathds{1}_{\{|\phi'(t)| \geq B\}}  dt.
\end{align*}
From this and Lemma~\ref{Lemma5.18}, we can choose $B$ large enough such that 
\begin{align*}
\sup\left( \int_0^T L(\tilde \phi(t),\phi'(t)) \mathds{1}_{\{|\phi'(t)| \geq B\}}dt,  \int_0^T L( \phi(t),\phi'(t)) \mathds{1}_{\{|\phi'(t)| \geq B\}}  dt\right)
<\frac{\epsilon}{4}.
\end{align*} 
By Lemma~\ref{Lemma5.33}, there exists an $\delta>0$ such that for all $x,\tilde x\in A$ with $|x - \tilde x|<\delta$ and  $y \in \mathcal{C}$ with $|y|\leq B$,
\[
|L(x,y)-L(\tilde x,y)| < \frac{\epsilon}{2 T}.
\]
We obtain for $\sup_{0 \leq t \leq T} |\tilde \phi(t) - \phi(t)| < \delta$,
\begin{align*}
& \Big |\int_0^T \big(L(\tilde \phi(t), \phi'(t)) -  L(\phi(t), \phi'(t))\big) dt\Big| \\*
& \qquad \leq  \Big |\int_0^T L(\tilde \phi(t),\phi'(t)) \mathds{1}_{\{|\phi'(t)| \geq B\}}  dt \Big |
	+ \Big |\int_0^T L(\phi(t),\phi'(t)) \mathds{1}_{\{|\phi'(t)| \geq B\}}  dt \Big | \\*
& \qquad \qquad +  \int_0^T \big|L(\tilde \phi(t), \phi'(t)) -  L(\phi(t), \phi'(t))\big|
	 \mathds{1}_{\{|\phi'(t)| < B\}} dt\\
&\qquad  <\epsilon.
 \end{align*}
\end{proof}

\color{black}

\subsection{$I$ is a good rate function}

We first have
\begin{lemma} \label{Lemma5.40}
For $\delta>0$, $x\in A$ and $y\in \mathds{R}^d$, we define
\[
L_\delta(x,y):=\sup_{\theta \in \mathds{R}^d} \tilde\ell_\delta (x,y,\theta), 
\]
where
\[
\tilde\ell_\delta(x,y,\theta):= \langle \theta,y \rangle - \sup_{z\in A; |z-x|\leq \delta}\sum_{j} \beta_j(z) \big( \e^{\langle \theta, h_j \rangle} -1 \big).
\]

Since the $\beta_j$ are bounded and continuous, then
\[
L_\delta(x,y) \uparrow L_0(x,y)=L(x,y)
\] 
and $L_\delta(x,y)$ is lower semicontinuous in $(\delta,x,y)$.
\end{lemma}

\begin{proof}
It is easy to see that $\tilde\ell_\delta(x,y,\theta)$ is continuous is $(x,y,\delta)$, hence the first assertion follows. The second assertion follows from the fact that the supremum of a family of lower semicontinuous functions is lower semicontinuous.
\end{proof}

We next establish (recall the metric $d_D$ introduced in subsection \ref{SubSecDefRateFct})

\begin{lem} \label{Lemma5.42}
Let the $\beta_j$ be bounded and continuous. Then, $I_T$ is lower semicontinuous with respect to the metric $d_D$ on $D([0,T];A)$.
\end{lem}

\begin{proof}
As $I_T(\phi)=\infty$ if $\phi$ is not absolutely continuous, we can restrict our attention to sequences of absolutely continuous functions. As the Skorohod topology relativized to $C([0,T];A)$ coincides with the uniform topology (see, e.g., \cite{Billingsley1999}, Section 12, p.124), we can consider a sequence of functions $\phi_n \in C([0,T];A)$ converging to a function $\phi$ under the uniform topology. We can furthermore assume that $I_T(\phi_n) \leq K$ for some $K$ and all $n\ge1$.
 By Lemma~\ref{Lemma5.18}, the functions $\phi_n$ are hence uniformly absolutely continuous and therefore the limit $\phi$ is absolutely continuous.

Therefore, for any given $\delta>0$, there exists a $\Delta>0$ such that
\[
|s-t| \leq \Delta \Rightarrow |\phi_n(s)- \phi_n(t) | \leq \delta \text{ for all } n.
\]
We choose $\Delta$ smaller if necessary such that $T/\Delta=:J \in \mathds{N}$ and divide $[0,T]$ into subintervals $[t_j,t_{j+1}]$, $j=1,\dots,J$ of length $\le\Delta$. We note that for $|x'-x|\leq \delta$, we have $L_\delta(x',y) \leq L(x,y)$. Furthermore, we observe that $L_\delta(x,\cdot)$ is convex as a supremum of linear functions and hence by Jensen's Inequality
\begin{align}
\int_0^T L(\phi_n(t),\phi_n'(t)) dt &\geq \sum_{j=1}^J \int_{t_j}^{t_{j+1}} 
L_\delta (\phi_n(t_j), \phi_n'(t)) dt \notag \\
&\geq \sum_{j=1}^J \Delta L_\delta \Big(\phi_n(t_j), \frac{\phi_n(t_{j+1})- \phi_n(t_j)}{\Delta} \Big). \label{EqLem5.42.1}
\end{align}
We now further divide the interval $[0,T]$ into subintervals of length $\Delta_k:=\Delta 2^{-k}$, $k \in \mathds{N}$, $[t_j^k,t_{j+1}^k]$, $j=1,\dots,J_k:=2^k J$ and define the functions
\[
\underline{\phi}^k(t):=\phi(t_j^k) \quad \text{if } t \in [t_j^k,t_{j+1}^k], \quad \overline{\phi}^k(t) :=\underline{\phi}^k(t+\Delta_k). 
\]
Note that there exits a sequence $\delta_k \downarrow 0$ such that
\[
|s-t| < \Delta_k \Rightarrow |\phi_n(s) - \phi_n(t)|<\delta_k \text{ for all } n.
\]
Hence by Inequality~\eqref{EqLem5.42.1} and Lemma~\ref{Lemma5.40} for all $k \in \mathds{N}$,
\begin{align}
\liminf_{n\rightarrow \infty}
\int_0^T L(\phi_n(t),\phi_n'(t)) dt &\geq \sum_{j=1}^{J_k} \Delta_k \liminf_{n\rightarrow \infty}L_{\delta_k} \Big(\phi_n(t_j^k), \frac{\phi_n(t_{j+1}^k)- \phi_n(t_j^k)}{\Delta_k} \Big) \notag \\
&\geq \int_0^{T- \Delta_k} L_{\delta_k} \Big(\underline{\phi}^k(t), \frac{\overline{\phi}^k(t)- \underline{\phi}^k(t)}{\Delta_k} \Big) dt. \label{EqLemma5.40.2}
\end{align}
As $\phi$ is absolutely continuous, we have that for almost all $t \in [0,T]$,
\[
\frac{\overline{\phi}^k(t)- \underline{\phi}^k(t)}{\Delta_k} 
\rightarrow \phi'(t) \quad \text{as } k \rightarrow \infty.
\]
We conclude by using Inequality~\eqref{EqLemma5.40.2}, Fatou's Lemma  and Lemma~\ref{Lemma5.40} again:
\begin{align*}
\liminf_{n\rightarrow \infty}
\int_0^T L(\phi_n(t),\phi_n'(t)) dt 
&\geq \liminf_{k \rightarrow \infty} \int_0^{T- \Delta_k} L_{\delta_k} \Big(\underline{\phi}^k(t), \frac{\overline{\phi}^k(t)- \underline{\phi}^k(t)}{\Delta_k} \Big) dt \\
&\geq \int_0^{T} \liminf_{k \rightarrow \infty} \Big( \mathds{1}_{[0,T - \Delta_k]}(t) L_{\delta_k} \Big(\underline{\phi}^k(t), \frac{\overline{\phi}^k(t)- \underline{\phi}^k(t)}{\Delta_k} \Big) \Big) dt \\
&\geq 
\int_0^{T}  L (\phi(t), \phi'(t))  dt
\end{align*}
as required.
\end{proof}

We define for $K>0$, $x \in A$,
\begin{align*}
\Phi(K)&=\big\{ \phi \in D([0,T];A) | I_{T}(\phi)\leq K \big \}, \\
\Phi_x(K)&=\big\{ \phi \in D([0,T];A) | I_{T,x}(\phi)\leq K \big \}.
\end{align*}
We have moreover
\begin{prop} \label{Prop5.46}
Assume that $\beta_j$ ($j=1,\dots,k$) are bounded and continuous. Let furthermore $K>0$ and $\tilde A \subset A$ be compact.% and $\delta>0$. 
Then, the sets
\[
\bigcup_{x \in \tilde A} \Phi_x(K) %\quad \text{and} \quad \bigcup_{x \in A} \Phi^\delta_x(K)
\]
are compact in $C([0,T];A)$.
\end{prop}

\begin{proof}
By Lemma~\ref{Lemma5.18}, %respectively Lemma~\ref{Lemma5.39} 
the functions in $\bigcup_{x \in \tilde A} \Phi_x(K)$ %respectively in $\bigcup_{x \in A} \Phi^\delta_x(K)$ 
are equicontinuous. As $\tilde A$ is compact, the Theorem of Arzel\`{a}-Ascoli hence implies that $\bigcup_{x \in \tilde A} \Phi_x(K)$ has compact closure. Now, the semicontinuity of $I$ (cf.~Lemma~\ref{Lemma5.42}) implies that $\bigcup_{x \in \tilde A} \Phi_x(K)$ is closed which finishes the proof.
\end{proof}

 We define for $S \subset D([0,T];A)$,
\[
I_x(S):=\inf_{\phi \in S} I_{T,x}(\phi).
\]

\begin{lem} \label{Lemma5.47}
Assume that $\beta_j$ ($j=1,\dots,k$) are bounded and continuous.
Let $F\subset C([0,T];A)$ be closed. Then $I_x(F)$ is lower semicontinuous in $x$.
\end{lem}

\begin{proof}
We let $x_n \rightarrow x$ with $\liminf_{n\rightarrow \infty} I_{x_n}(F)=:K< \infty$. For simplicity, we assume that $I_{x_n}(F) \leq K+ \epsilon$ for some fixed $\epsilon>0$ and for all $n$. By Proposition~\ref{Prop5.46}, we have that for all $\epsilon,\delta>0$,
\[
F\cap \Phi_{x_n}(K+\epsilon) \quad \text{and} \quad 
F \cap \bigcup_{|x-y|\leq \delta} \Phi_y(K+\epsilon)
\]
are compact. By the semicontinuity of $I_T(\cdot)$ (cf.~Lemma~\ref{Lemma5.42}) and the fact that a l.s.c. function attains its minimum on a compact set, there exist $\phi_n\in F$ such that $I_{x_n}(F)=I_{T,x_n}(\phi_n)$ (for $n$ large enough). As the $\phi_n$ are in a compact set, there exists a convergent subsequence with limit $\phi$, in particular $\phi(0)=x$. As $F$ is closed, we have $\phi \in F$. We use Lemma~\ref{Lemma5.42} again and obtain
\[
I_x(F) \leq I_{T}(\phi) \leq \liminf_{n\rightarrow \infty} I_T(\phi_n) = \liminf_{n\rightarrow \infty} I_{x_n}(F)=K
\]
as required.
\end{proof}

The following result is  a direct consequence of Lemma~\ref{Lemma5.47}. 

\begin{lem} \label{Lemma5.63}
Assume that $\beta_j$ ($j=1,\dots,k$) is bounded and continuous.
For $F\subset D([0,T];A)$ closed and $x \in A$, we have
\[
\lim_{\epsilon \downarrow 0} \inf_{y \in A, \, |x-y|<\epsilon} I_y(F) =I_x(F)
\]
\end{lem}

We can now establish the main result of this subsection.
\begin{prop}\label{Prop4.49}
Let the $\beta_j$ be bounded and continuous. For all $x$, $I_x$ is a good rate function on $C([0,T];A)\cap\{\phi | \phi(0)=x\}$. 
\end{prop}

\begin{proof}
It is clear that $I_{T}$ is non-negative as $L$ is non-negative. Furthermore, it is lower semicontinuous by Lemma~\ref{Lemma5.42}. By Proposition~\ref{Prop5.46} its level sets are compact.
\end{proof}

We have moreover

\begin{cor} \label{Cor5.50}
Let the $\beta_j$ be bounded and continuous. For all $x\in A$, $I_{T,x}$ is a good rate function on $D([0,T];A)\cap\{\phi | \phi(0)=x\}$ under both metrics
$d_C$ and $d_D$. 
\end{cor}

\begin{proof}
Since $I_{t,x}$ is finite only for absolutely continuous functions, it suffices to consider sequences in $C([0,T];A)\cap\{\phi | \phi(0)=x\}$. Limits of such sequences (under either metric) are continuous and convergence is equivalent for both metrics (see, e.g., \cite{Billingsley1999}). Lower semicontinuity follows. Compactness of the level sets follows by~Proposition~\ref{Prop4.49} and the fact that the identity maps from $(C([0,T];A),d_C)$ into $(D([0,T];A),d_C)$ and $(D([0,T];A),d_D)$ are continuous .
\end{proof}

\subsection{A property of non--exponential equivalence}
It is worth wondering whether or not $\{Z^{N,x_N}_t,\ 0\le t\le T\}$ and $\{Z^{N,x}_t,\ 0\le t\le T\}$
are exponentially equivalent,  whenever $x_N\to x$ as $N\to\infty$. Indeed, \cite{Dembo2009} prove that
property for diffusions with small noise and Lipschitz coefficients, and use it to establish certain results,
of which we shall prove analogs below, but without that exponential equivalence, which fails to hold
in our Poissonian case. 

 Let $x,y\in A$, and consider the processes
\begin{align*}
Z^{N,x}_t&=x+\sum_{j=1}^k\frac{h_j}{N}P_j\left(N\int_0^t\beta_j(Z^{N,x}_s)ds\right),\\
Z^{N,y}_t&=y+\sum_{j=1}^k\frac{h_j}{N}P_j\left(N\int_0^t\beta_j(Z^{N,y}_s)ds\right).
\end{align*}
For any $\delta>0$, as $|x-y|\to 0$, we ask what is the limit, as $|x-y|\to0$, of
\begin{equation}\label{exp-equiv}
  \limsup_{N\to\infty}\frac{1}{N}\log\P\left(\sup_{0\le t\le T}|Z^{N,x}_t-Z^{N,y}_t|>\delta\right)\ \text{?}
  \end{equation} 
  If that limit is $-\infty$, then we would have the above exponential equivalence.
  We now show on a particularly simple example that this is not the case. It is easy to infer that 
  it in fact fails in the above generality, assuming that the $\beta_j$'s are Lipschitz continuous and bounded.
We consider the case $d=1$, $A=\R_+$, $k=1$, $\beta(x)=x$, $h=1$. We could truncate $\beta(x)$ to make it bounded, in order to comply with our standing assumptions. The modifications below would be minor, but
we prefer to keep the simplest possible notations. Assume $0<x<y$ and consider
the two processes
\begin{align*}
Z^{N,x}_t&=x+\frac{1}{N}P\left(N\int_0^t Z^{N,x}_s ds\right),\\
Z^{N,y}_t&=y+\frac{1}{N}P\left(N\int_0^t Z^{N,y}_s ds\right).
\end{align*}
It is plain that $0<Z^{N,x}_t<Z^{N,y}_t$ for all $N\ge1$ and $t>0$. Let $\Delta^{N,x,y}_t= Z^{N,y}_t-Z^{N,x}_t$.
The law of $\{\Delta^{N,x,y}_t,\ 0\le t\le T\}$ is the same as that of the solution of
\[ \Delta^{N,x,y}_t=y-x+\frac{1}{N}P\left(N\int_0^t \Delta^{N,x,y}_s ds\right).\]
We deduce from Theorem \ref{TheoremLDPLower} below (which is established in case of a bounded coefficient $\beta(x)$, but it  makes no difference here) that
\begin{align*}
\liminf_{N\to\infty}\frac{1}{N}\log\P[\Delta^{N,x,y}_1>1]&\ge-\inf_{\phi(0)=y-x,\ \phi(1)>1}I_{1,y-x}(\phi)\\
&\ge-I_{1,y-x}(\psi),
\end{align*}
with $\psi(t)=y-x+t$, hence
\begin{align*}
I_{1,y-x}(\psi)&=\int_0^1L(y-x+t,1)dt\\
&=\int_0^1[y-x+t-1-\log(y-x+t)]dt\\
&=y-x+1/2-(y-x+1)\log(y-x+1)+(y-x)\log(y-x)\\
&\to1/2,
\end{align*}
as $y-x\to0$. This clearly contradicts the exponential equivalence. 

We note that the above process $Z^{N,x}_t$ can be shown to be ``close'' (in a sense which is made very precise in 
\cite{Kurtz1978}) to its diffusion approximation
\[ X^{N,x}_t=x+\int_0^tX^{N,x}_sds+\frac{1}{\sqrt{N}}\int_0^t\sqrt{X^{N,x}_s}dB_s,\]
where $\{B_t,\ t\ge0\}$ is  standard Brownian motion. One can study large deviations  of this diffusion process
from its Law of Large Numbers limit (which is the same as that of $Z^{N,x}_t$). The rate function on the time interval $[0,1]$  is now
\[ I(\phi)=\int_0^1\frac{(\phi'(t)-\phi(t))^2}{\phi(t)}dt.\]
Let again $\psi(t)=y-x+t$, now with $0=x<y$.
$I(\psi)=\log(1+y)-\log(y)-3/2+y\to+\infty$, as $y\to0$. We see here that the large deviations behaviour of the solution of the Poissonian SDE and of its diffusion approximation differ dramatically, as was already noted by \cite{Pakdaman2010} (see also the references in this paper).

\section{Lower bound} \label{SecLower}

We first establish the LDP lower bound under the assumption that the rates are bounded away from zero, or in other words the
$\log\beta_j$'s are bounded. From this, we will derive later the general result.

\subsection{LDP lower bound if the rates are bounded away from zero} \label{SubSecBoundedRatesLower}

We first note that if the $\beta_j$ are bounded away from zero, then the convex cone $\mathcal C_x$ is dependent of $x$, $\mathcal{C}_x=\mathcal C$ for all $x$. Note that this implies that the ``domain'' $A$ of the process cannot be bounded.

 We require a LDP for linear functions. This follows from the LLN (Theorem~\ref{ThLLN}).

\begin{prop} \label{Prop5.53}
Assume that $\log \beta_j$ ($j=1,\dots,k$) is bounded and continuous. 
For any $\epsilon>0$, $\delta>0$ there exists an $\tilde \epsilon>0$ such that
for $x\in A$, $y \in \mathcal C$ and $\mu \in V_{x,y}=\tilde V_y$,
\[
\liminf_{N\rightarrow \infty} \frac 1 N \log \Big(\inf_{z\in A; |z-x|<\tilde \epsilon} \mathbb{P} \Big[\sup_{t \in [0,T]} |Z^{N,z}(t) - \phi^x(t)|< \epsilon\Big] \Big)\geq - \int_0^T  \ell(\phi^x(t),\mu) dt - \delta,
\]
where
\[
\phi^x(t):=x+t y= x + t\sum_j \mu_jh_j.
\]
 
\end{prop}

\begin{proof}
We define
\[
F^{N,z}:=\Big\{ \sup_{t \in [0,T]} |Z^{N,z}(t)-\phi^z(t)| < \frac \epsilon 2 \Big\}
\]
and let $\tilde \epsilon<\epsilon_1= \epsilon/2$.
Let now $\xi_T=\xi^{N,z}_T=\frac{d \tilde{\mathbb{P}}}{d \mathbb{P}}|_{\mathcal{F}_T}$ be given as in Theorem~\ref{TheoremB6} for initial value $z$ and constant rates $\tilde \beta_j=\mu_j$. Then, with the notation $\tilde{\mathbb{E}}_{F^{N,z}}[X]:=\tilde{\mathbb{E}}[X | F^{N,z}]$ and (recall that $\xi_T\not=0$ $\tilde{\mathbb{P}}$-almost surely)
\[
X_T^{N,z}:=X_T:=\log \xi_T^{-1}=\sum_{\tau\leq T} \Big[  \log\beta_{j(\tau)}(Z^{N,z}(\tau-))-  \log\mu_{j(\tau)}  \Big]   -N \sum_j\int_0^{T} \big( \beta_j(Z^{N,z}(t)) - \mu_j\big) dt,
\]

\begin{align}
&\liminf_{N\rightarrow \infty} \frac{1}{N} \log \Big(\inf_{z\in A, \,|x-z|<\tilde \epsilon} \mathbb{P}\big[\sup_{t \in [0,T]} | Z^{N,z}(t) - \phi^x(t)| < \epsilon\big] \Big)\notag \\
&\qquad \geq \liminf_{N\rightarrow \infty} \frac{1}{N} \log \inf_{z\in A, \,|x-z|<\tilde \epsilon} \mathbb{P}[F^{N,z} ]\notag \\
&\qquad = \liminf_{N\rightarrow \infty} \frac{1}{N}  \inf_{z\in A, \,|x-z|<\tilde \epsilon} \log\mathbb{P}[F^{N,z}] \notag\\
&\qquad \geq \liminf_{N\rightarrow \infty} \frac{1}{N}  \inf_{z\in A, \,|x-z|<\tilde \epsilon} \log\tilde{\mathbb{E}}\big[ \xi_T^{-1} \mathds{1}_{F^{N,z}}\big] \notag \\
&\qquad = \liminf_{N\rightarrow \infty} \frac{1}{N}  \inf_{z\in A, \,|x-z|<\tilde \epsilon} \log \big( \tilde{\mathbb{P}}[F^{N,z}] \tilde{\mathbb{E}}_{F^{N,z}}[ \exp(X_T)] \big) \notag  \\
&\qquad \geq \liminf_{N\rightarrow \infty} \frac{1}{N}  \inf_{z\in A, \,|x-z|<\tilde \epsilon}  \log \tilde{\mathbb{P}}[F^{N,z}]+  \liminf_{N\rightarrow \infty} \frac{1}{N}    \inf_{z\in A, \,|x-z|<\tilde \epsilon} \log  \tilde{\mathbb{E}}_{F^{N,z}}[ \exp(X_T)]  \notag \\
&\qquad \geq   \liminf_{N\rightarrow \infty} \inf_{z\in A, \,|x-z|<\tilde \epsilon}   \tilde{\mathbb{E}}_{F^{N,z}}\Big[\frac{X_T}{N} \Big],\label{EqProp5.53Exit.2}
\end{align}
where we have used Corollary~\ref{CorB.6} for the second inequality, Theorem~\ref{ThLLN} and Jensen's inequality on the last line. 
Note the independence of the constants $\tilde C_1, \tilde C_2$  of $z$  in Theorem~\ref{ThLLN} and hence
\[
\tilde{\mathbb{P}}[F^{N,z}] \rightarrow 1 \quad \text{as } N\rightarrow \infty \text{ independently of } z.
\]
We have
\begin{equation} 
\frac{1}{\tilde{\mathbb{P}}[F^{N,z}]} \tilde{\mathbb{E}} \big[\mathds{1}_{F^{N,z}} T \sum_j  \mu_j\big] =T\sum_j \mu_j. \label{EqProp5.53Exit.3}
\end{equation}

By the fact that the $\beta_j$'s are bounded and continuous and by Theorem~\ref{ThLLN}, we have for $j=1,\dots,k$,
\[
\sup_{t \in [0,T]} | \beta_j(Z^{N,z}(t)) - \beta_j(\phi^z(t)) | \rightarrow 0 \quad \tilde{\mathbb{P}}-\text{a.s.}
\]
as $N\rightarrow \infty$ uniformly in $z$. This implies 
\begin{equation}
\frac{1}{\tilde{\mathbb{P}}[F^{N,z}]} \tilde{\mathbb{E}} \Big[\mathds{1}_{F^{N,z}} \int_0^{T} \sum_j \beta_j(Z^{N,z}(t)) dt \Big]
\longrightarrow \sum_j \int_0^{T} \beta_j(\phi^z(t)) dt  \label{EqProp5.53Exit.4}
\end{equation}
as $N\rightarrow \infty$ uniformly in $z$.

Let us now define the following processes. For $z \in A$, $j=1,\dots,k$  and $0\leq t_1<t_2 \leq T$ let
\begin{equation*}
Y_j^{N,z,t_1,t_2}:=\frac{1}{N} \cdot  \#\text{jumps of } Z^{N,z} \text{ in direction } h_j \text{ in }[t_1,t_2].\footnotemark
\end{equation*}

Let furthermore $\tau_j \in [0,T]$ denote the jump times of $Z^N$ in direction $h_j$; we obtain 
\begin{align}
&\frac{1}{\tilde{\mathbb{P}}[F^{N,z}]} 
	\sum_{j; \mu_j>0} 
		\tilde{\mathbb{E}} \Big[\frac{1}{N} \mathds{1}_{F^{N,z}} \sum_{\tau_j} 
			\log  \mu_j\Big] \notag \\
&\qquad =\frac{1}{\tilde{\mathbb{P}}[F^{N,z}]} 
	\sum_{j; \mu_j>0} \log \mu_j 
		\Big\{\tilde{\mathbb{E}} \big[Y_j^{N,z,0,T}\big] \tilde{\mathbb{P}} \big[F^{N,z}\big]+\widetilde{\Cov}(\mathds{1}_{F^{N,z}},  Y_j^{N,z,0,T}) \Big\} \notag 
\\
 &\qquad \rightarrow T\sum_{j} \mu_j\log \mu_j, \label{EqProp5.53Exit.5} 
\end{align}
since, for a given set $F$,
\begin{align*}
\tilde{\mathbb{E}}[Y_j^{N,z,t_1,t_2}] &= (t_2-t_1) \mu_j\\
\widetilde{\Var}[Y_j^{N,z,t_1,t_2}] &= (t_2-t_1)  \mu_j\notag\\
 |\widetilde{\Cov}(\mathds{1}_{F},  Y_j^{N,z,t_1,t_2})|  &\leq  \sqrt{\widetilde{\Var}[\mathds{1}_{F}]} \sqrt{\widetilde{\Var}[ Y_j^{N,z,t_1,t_2}]} = \sqrt{\tilde{\mathbb{P}}[F]-\tilde{\mathbb{P}}[F]^2} \sqrt{(t_2-t_1) \mu_j}.
\end{align*}

We now define the set
\[
\tilde F^{N,z}:=\Big\{ \sup_{t \in [0,T]} |Z^{N,z}(t)-\phi^z(t)| < \epsilon_N \Big\} \quad \text{for } \epsilon_N:=\epsilon \wedge \frac{1}{N^{1/3}};
\]
we have (for $N$ large enough)
\begin{equation}\notag
\tilde{\mathbb{P}} [\tilde F^{N,z}] \geq 1-  \tilde C_1 \exp\big(-N\tilde C_2(\epsilon_N)\big)  \longrightarrow 1
\end{equation}
as $N \rightarrow \infty$ uniformly in $z$ by Theorem~\ref{ThLLN}.
We furthermore let $\bar A\subset A$ be compact such that for all $z$, $|z-x|<\tilde \epsilon$ and all $t\in [0,T]$, $\phi^z(t)\in \bar A$ and $Z^{N,z}(t)\in \bar A$ on $F^{N,z}$.
As the $\log \beta_j$ are bounded and uniformly continuous, there exit constants $\tilde{\delta}_N>0$ with $\tilde \delta_N \downarrow 0$ such that
\begin{equation}  \notag
\tilde x, \bar x \in \bar A, \, |\tilde x-\bar x|<  \frac{2}{N^{1/3}}  \Rightarrow |\log\beta_j( \tilde x) - \log \beta_j(\bar x)|< \tilde \delta_N.
\end{equation}
We define $\bar \mu=\max_j \mu_j$,
\[
M = M(N) := \lfloor T N^{{1}/{3}} k \bar h \bar \mu+1 \rfloor
\]
and divide the interval $[0, T]$ into $M$ subintervals $[ t_r,t_{r+1}]$ ($r=0,\dots,M-1$, $t_r=t_r(N)$) of length $\Delta = \Delta (N)$, i.e. for $N\geq N_0$ independent of $z$ large enough,
\[
\Delta < \frac{1}{  N^{1/3}k \bar \mu \bar h}.
\]
For $j$, $r=0,\dots,M-1$ and $\tau_j,t \in [ t_r,t_{r+1}]$, since for $|t-s|<\frac{1}{  N^{1/3}k \bar \mu \bar h}$, $|\phi^z(t)-\phi^z(s)|<\frac{1}{N^{1/3}}$, 
we have on $\tilde F^{N,z}$
\[
|Z^{N,z}(\tau_j-) - \phi^z(t)| \leq |Z^{N,z}(\tau_j-) - \phi^z(\tau_j)|
+ |\phi^z(\tau_j) - \phi^z(t)| \leq \frac{2}{N^{1/3}},
\]
and hence
\begin{equation} \notag
\inf_{ t \in [t_r,t_{r+1}]} \log \beta_j(\phi^z( t))  - \tilde \delta_N
\leq   \log \beta_j(Z^{N,z}(\tau_j-)) 
\leq \sup_{t \in  [t_r,t_{r+1}]} \log \beta_j(\phi^z( t)) + \tilde \delta_N.
\end{equation}
We compute
\begin{align}
&\frac{1}{\tilde{\mathbb{P}}[F^{N,z}]} 
	\tilde{\mathbb{E}} \Big[\frac{1}{N}\mathds{1}_{F^{N,z}} \sum_\tau \log \beta_{j(\tau)}(Z^{N,z}(\tau-)) \Big] \notag \\*
	&\qquad =\frac{1}{\tilde{\mathbb{P}}[F^{N,z}]}  \sum_{j,\, \mu_j>0} \sum_{r=0}^{M-1} \tilde{\mathbb{E}} \Big[\frac{1}{N}\mathds{1}_{\tilde F^{N,z}} \sum_{\tau_j\in [t_r,t_{r+1})} \log \beta_{j}(Z^{N,z}(\tau_j-)) \Big] \notag \\
	&\qquad \qquad +\frac{1}{\tilde{\mathbb{P}}[F^{N,z}]}  \sum_{j,\, \mu_j>0} \sum_{r=0}^{M-1} \tilde{\mathbb{E}} \Big[\frac{1}{N}\mathds{1}_{F^{N,z}\setminus\tilde F^{N,z}} \sum_{\tau_j\in [t_r,t_{r+1})} \log \beta_{j}(Z^{N,z}(\tau_j-)) \Big]\notag \\
	&\qquad 	\leq \frac{1}{\tilde{\mathbb{P}}[F^{N,z}]} \sum_{j,\, \mu_j>0} \sum_{r=0}^{M-1} \Big( \sup_{t \in [t_r,t_{r+1})} \log\beta_j(\phi^z(t)) +\tilde \delta_N \Big) \tilde{\mathbb{E}} \Big[\mathds{1}_{\tilde F^{N,z}} Y_j^{N,z,t_r,t_{r+1}}\Big] \notag \\
	&\qquad \qquad +\frac{1}{\tilde{\mathbb{P}}[F^{N,z}]}  \sum_{j,\, \mu_j>0} \sum_{r=0}^{M-1}  \log\bar \beta \tilde{\mathbb{E}} \Big[\mathds{1}_{ F^{N,z}\setminus\tilde F^{N,z}} Y_j^{N,z,t_r,t_{r+1}}\Big] \notag \\
	&\qquad 	\leq \frac{1}{\tilde{\mathbb{P}}[F^{N,z}]} \sum_{j,\, \mu_j>0} \sum_{r=0}^{M-1}  \sup_{t \in [t_r,t_{r+1})} \log\beta_j(\phi^z(t)) \Big\{ \tilde{\mathbb{E}} \big[\mathds{1}_{\tilde F^{N,z}}\big]\tilde{\mathbb{E}} \big[ Y_j^{N,z,t_r,t_{r+1}} \big]
+ \widetilde{\Cov}(\mathds{1}_{\tilde F^{N,z}},Y_j^{N,z,t_r,t_{r+1}})\Big\} \notag \\
	&\qquad \qquad +\frac{1}{\tilde{\mathbb{P}}[F^{N,z}]}  \sum_{j,\, \mu_j>0} \sum_{r=0}^{M-1}  \log\bar \beta\Big\{ \tilde{\mathbb{E}} \big[\mathds{1}_{F^{N,z}\setminus\tilde F^{N,z}}\big]\tilde{\mathbb{E}} \big[ Y_j^{N,z,t_r,t_{r+1}} \big]
+ \widetilde{\Cov}(\mathds{1}_{F^{N,z}\setminus\tilde F^{N,z}},Y_j^{N,z,t_r,t_{r+1}})\Big\}  \notag \\
	&\qquad \qquad  + \frac{1}{\tilde{\mathbb{P}}[F^{N,z}]}\tilde \delta_N k \bar \mu T  \notag \\
	&\qquad 	\leq \frac{\tilde{\mathbb{P}} [\tilde F^{N,z}]}{\tilde{\mathbb{P}}[F^{N,z}]} \sum_{j,\, \mu_j>0} \mu_j\sum_{r=0}^{M-1} \Delta  \sup_{t \in [t_r,t_{r+1})}\log \beta_j(\phi^z(t))  \notag \\
	&\qquad \qquad  + \frac{1}{\tilde{\mathbb{P}}[F^{N,z}]} \Big\{ \log \bar \beta
\sum_{j,\, \mu_j>0}  \sum_{r=0}^{M-1}   
\Big(|\widetilde{\Cov}(\mathds{1}_{\tilde F^{N,z}},Y_j^{N,z,t_r,t_{r+1}})|
+|\widetilde{\Cov}(\mathds{1}_{F^{N,z}\setminus \tilde F^{N,z}},Y_j^{N,z,t_r,t_{r+1}})|\Big)\notag \\
	&\qquad \qquad
+ \tilde \delta_N k \bar \mu T +\tilde{\mathbb{P}}[F^{N,z}\setminus \tilde F^{N,z}]  k \bar \mu T\log \bar \beta\Big\} \label{EqProp5.53Exit.22} \\
	&\qquad =: \overline{S}^{N,z} + \overline{U}^{N,z}, \notag
\end{align}
where $\overline{S}^{N,z}$ and $\overline{U}^{N,z}$ are the first respectively the second term in Inequality~\eqref{EqProp5.53Exit.22}. In a similar fashion we obtain
\begin{align}
	&\frac{1}{\tilde{\mathbb{P}}[F^{N,z}]} 
\tilde{\mathbb{E}} \Big[\frac{1}{N}\mathds{1}_{F^{N,z}} \sum_\tau \log \beta_{j(\tau)}(Z^{N,z}(\tau-)) \Big] \notag \\
	&\qquad \geq \frac{\tilde{\mathbb{P}} [\tilde F^{N,z}]}{\tilde{\mathbb{P}}[F^{N,z}]} \sum_{j,\, \mu_j>0} \mu_j\sum_{r=0}^{M-1} \Delta  \inf_{t \in [t_r,t_{r+1})} \log \beta_j(\phi^z(t))  \notag \\
	&\qquad \qquad  + \frac{1}{\tilde{\mathbb{P}}[F^{N,z}]} \Big\{ \log \underline \beta
\sum_{j,\, \mu_j>0}  \sum_{r=0}^{M-1}   
\Big(|\widetilde{\Cov}(\mathds{1}_{\tilde F^{N,z}},Y_j^{N,z,t_r,t_{r+1}})| 
 +|\widetilde{\Cov}(\mathds{1}_{F^{N,z}\setminus\tilde F^{N,z}},Y_j^{N,z,t_r,t_{r+1}})|\Big)\notag \\
&\qquad \qquad - \tilde \delta_N k \bar \mu T+\tilde{\mathbb{P}}[F^{N,z}\setminus \tilde F^{N,z}]  k \bar \mu T\log \underline \beta\Big\}\notag \\
&\qquad =: \underline{S}^{N,z} + \underline{U}^{N,z}; \notag
\end{align}
we first note that $\overline{U}^{N,z}, \underline U^{N,z} \rightarrow 0$ as $N\rightarrow \infty$ uniformly in $z$, since $\tilde \delta_N\rightarrow 0$ and 
 $\tilde{\mathbb{P}}[F^{N,z} \setminus \tilde F^{N,z}] \rightarrow 0$ as $N \rightarrow 0$ uniformly in $z$.  Furthermore, as (up to a factor which converges to $1$ uniformly in $z$) $\overline S^{N,z}$ and $\underline S^{N,z}$ are upper respectively lower Riemann sums, we obtain
\begin{equation} \label{EqProp5.53Exit.6}
\overline S^{N,z}, \underline S^{N,z} \rightarrow \sum_{j} \mu_j\int_0^{T} \log \beta_j(\phi^z(t))dt
\end{equation}
as $N \rightarrow \infty$; since
\[
|\overline S^{N,z}- \underline S^{N,z}| \leq 2 \frac{\tilde{\mathbb{P}} [\tilde F^{N,z}]}{\tilde{\mathbb{P}}[F^{N,z}]}  k \bar \mu \tilde \delta_N 
\rightarrow 0
\]
uniformly in $z$, the convergence in~\eqref{EqProp5.53Exit.6} is likewise uniform in $z$. 

The uniform convergence implies (cf.~\eqref{EqProp5.53Exit.2}~-~\eqref{EqProp5.53Exit.5} and the preceding discussion) that
\[
\liminf_{N\rightarrow \infty} \frac{1}{N} \log \Big(\inf_{z\in A, \,|x-z|<\tilde \epsilon} \mathbb{P}\big[\sup_{t \in [0,T]} | Z^{N,z}(t) - \phi^x(t)| < \epsilon\big] \Big)
\geq - \sup_{z\in A, \,|x-z|<\tilde \epsilon} \int_0^T  \ell (\phi^z(t),\mu) dt
\]
In combination with the uniform continuity of $\ell(\cdot,\mu)$ (recall the boundedness of the $\log \beta_j$) this proves the assertion.
\end{proof}

The main building block for the lower bound is the following result.

\begin{theorem} \label{Theorem5.51a}
Assume that $\log \beta_j$ ($j=1,\dots,k$) is bounded and continuous. Let $\phi \in D([0,T];A)$ with $\phi(0)=x$ and $\epsilon >0$. Then,
\[
\liminf_{N\rightarrow \infty} \frac{1}{N} \log \mathbb{P}\Big[\sup_{t\in [0,T]}|Z^{N,x}(t)-\phi(t)|<\epsilon\Big] \geq - I_{T,x}(\phi).
\]
The convergence is uniform in $x\in A$.
\end{theorem}

\begin{proof}
We can w.l.o.g. assume that $I_{T,x}(\phi)<\infty$ (and hence $\phi$ is absolutely continuous) as else the assertion is trivial. We approximate the function $\phi$ by a continuous piecewise linear function and then apply the LDP for linear functions to each of these linear functions. To this end, we let $\delta>0$ and divide the interval $[0,T]$ into $J$ subintervals of length $\Delta=T/J$, $[t_{r-1},t_r]$ ($r=1,\dots,J$) such that the resulting piecewise linear approximation
\[
\tilde \phi(t) = \phi(t_{r-1}) + \frac{t-t_{r-1}}{\Delta} (\phi(t_{r})-\phi(t_{r-1}))
\] 
satisfies
\[
\sup_{t\in[0,T]} |\phi(t)-\tilde \phi(t)|<\frac \epsilon 2
\]
(recall that $\phi$ is continuous). 

We now apply Theorem~\ref{Theorem5.35} twice (in Inequalities~\eqref{EqTh5.51.1} and~\eqref{EqTh5.51.3}) and choose $J$ large enough in order to assure
\begin{align}
\int_0^T L(\phi(t),\phi'(t)) dt & = \sum_{r=1}^J \int_{t_{r-1}}^{t_r} L(\phi(t),\phi'(t))dt \notag \\
& \geq \sum_{r=1}^J \int_{t_{r-1}}^{t_r} L(\phi(t_{r-1}),\phi'(t))dt - \frac \delta 4\label{EqTh5.51.1} \\
 &\geq \Delta \sum_{r=1}^J  L\Big (\phi(t_{r-1}),\frac{\Delta\phi(t_r)}{\Delta}\Big) - \frac \delta 4 \label{EqTh5.51.2} \\
 &\geq \sum_{r=1}^J \int_{t_{r-1}}^{t_r} L\big (\tilde\phi(t),\tilde \phi'(t)\big)dt - \frac \delta 2, \label{EqTh5.51.3}
\end{align}
where 
\[
\Delta\phi(t_r):=\phi(t_r)-\phi(t_{r-1}).
\] 
Note that for Inequality~\eqref{EqTh5.51.2}, we have applied Jensen's inequality and the fact that $L$ is convex in its second argument (cf.~Corollary~\ref{Corollary5.27}). 
As $I_{T,x}(\tilde\phi)< \infty$, this implies
\[
\Delta\phi(t_r) \in \mathcal C \quad \text{for all } r.
\] 
We note that by the continuity of $L(\cdot,y)$, $\mu^*(x,y)$ (the minimizing $\mu \in V_{x,y}=\tilde V_y$ for $ \ell(x,\cdot)$) is ``almost optimal'' for all $\tilde x$ sufficiently close to $x$ (in the sense that $\ell(\tilde x,\mu^*(x,y))$ is close to $L(\tilde x,y)$). By dividing each interval $[t_{r-1},t_r]$ into further subintervals $[s_{j-1},s_j]$ if necessary, we can hence represent the directions $\Delta\phi(t_k)/\Delta$ by
\[
\mu^j \in V_{\tilde \phi(t),\Delta\phi(t_r)/\Delta} =\tilde V_{\Delta\phi(t_r)/\Delta}
\]
in such a way that
\[
L\Big(\tilde\phi(t),\frac{\Delta\phi(t_r)}{\Delta}\Big) \geq \ell (\tilde\phi(t),\mu^j)- \frac{\delta}{4 T} \quad \text{for all } t \in[s_{j-1},s_j].
\]
For simplicity of exposition, we assume that this further subdivision of the intervals $[t_{r-1},t_r]$ is not required and denote the ``almost optimal'' $\mu$'s by $\mu^r$ ($r=1,\dots,J$). 
Hence
\begin{equation} \label{EqTh5.51.4}
\int_0^T L(\phi(t),\phi'(t)) dt \geq \sum_{r=1}^J \int_{t_{r-1}}^{t_r} \ell (\tilde\phi(t),\mu^r)-  \frac{3\delta}{4}.
\end{equation}

Choose now $\tilde \epsilon=\tilde \epsilon_{J-1}$ according to Proposition~\ref{Prop5.53} corresponding to $\epsilon/2$, $\delta/(4J)$, initial value $\tilde\phi(t_{J-1})$ and time-horizon $\Delta$. Using the Markov property of $Z^N$, we compute
\begin{align}
\mathbb{P}\Big[\sup_{t \in [0,T]} | Z^{N,x}(t)-\phi(t)|< \epsilon\Big] 
	& \geq \mathbb{P}\Big[\sup_{t \in [0,t_{J-1}]} | Z^{N,x}(t)-\tilde\phi(t)|< \tilde \epsilon\Big] \notag \\
	& \quad \cdot  \!\!\!\!\!\!\!\! \!\!\!\!\!\!\!\! \inf_{z \in A; |z-\tilde \phi(t_{J-1})|< \tilde\epsilon}  \mathbb{P}\Big[\sup_{t \in [t_{J-1},T]} | Z^{N,z}(t)-\tilde\phi(t)|< \frac \epsilon 2\Big]; \notag
\end{align}
here, we denote (by a slight abuse of notation) the process starting at $z$ at time $t_{J-1}$ by $Z^{N,z}$. 
%Now, for $z \in A$ with 
%\[
%|z-\tilde \phi(t_{J-1})|<\frac{J-1}{J} \frac \epsilon 2 \quad  \text{and}\quad 
%|Z^{N,z}(t)-\underbrace{(\tilde \phi(t)+z-\tilde\phi(t_{J-1}))}_{=:\phi^z(t)}|< \frac{\epsilon}{2 J},
%\]
%we have
%\[
%|Z^{N,z}(t)-\tilde \phi(t)|< \frac{\epsilon}{2}
%\]
%and therefore
%\[
%\mathbb{P}\Big[\sup_{t \in [t_{J-1},T]} | Z^{N,z}(t)-\phi^z(t)|< \frac{\epsilon}{2J}\Big]\leq\mathbb{P}\Big[\sup_{t \in [t_{J-1},T]} | Z^{N,z}(t)-\tilde\phi(t)|< \frac \epsilon 2\Big].
%\]
Proposition~\ref{Prop5.53} implies
\begin{align}
& \liminf_{N \rightarrow \infty} \frac 1 N \log \mathbb{P}\Big[\sup_{t \in [0,T]} | Z^{N,x}(t)-\phi(t)|< \epsilon\Big] \notag \\
	& \qquad \geq \liminf_{N \rightarrow \infty} \frac 1 N \log \mathbb{P}\Big[\sup_{t \in [0,t_{J-1}]} | Z^{N,x}(t)-\tilde\phi(t)|< \tilde \epsilon_{J-1}\Big]   \notag \\
	& \qquad \qquad  +\liminf_{N \rightarrow \infty} \frac 1 N \log \Big( \inf_{z \in A; |z-\tilde \phi(t_{J-1})|< \tilde\epsilon}  \mathbb{P}\Big[\sup_{t \in [t_{J-1},T]} | Z^{N,z}(t)-\tilde\phi(t)|< \frac \epsilon 2\Big] \Big)\notag \\
	& \qquad \geq \liminf_{N \rightarrow \infty} \frac 1 N \log \mathbb{P}\Big[\sup_{t \in [0,t_{J-1}]} | Z^{N,x}(t)-\tilde\phi(t)|< \tilde \epsilon_{J-1}\Big]    - \int_{t_{J-1}}^T \ell (\tilde \phi(t),\mu^{J})dt -  \frac{\delta}{4 J}.\notag
\end{align}
Iterating this procedure, we obtain
\[ \liminf_{N \rightarrow \infty} \frac 1 N \log \mathbb{P}\Big[\sup_{t \in [0,T]} | Z^{N,x}(t)-\phi(t)|< \epsilon\Big]  \geq   - \sum_{r=1}^J \int_{t_{r-1}}^{t_r} \ell (\tilde \phi(t),\mu^{r})dt -  \frac{\delta}{4}\notag
\]
and the assertion follows from Inequality~\eqref{EqTh5.51.4} if we let $\delta \rightarrow 0$.

We note that the convergence is uniform in $x$ by the uniformity in Proposition~\ref{Prop5.53}.
\end{proof}

\begin{theorem}\label{Theorem5.51b}
Assume that $\log \beta_j$ ($j=1,\dots,k$) is bounded and continuous.
Let $G\subset D([0,T];A)$ be open and $x \in A$. Then,
\[
\liminf_{N\rightarrow \infty} \frac{1}{N}\log \mathbb{P}[Z^{N,x}\in G] \geq - \inf_{\phi \in G,} I_{T,x}(\phi).
\]
The convergence is uniform in $x \in A$.
\end{theorem}

\begin{proof}
Let $\inf_{\phi \in G} I_{T,x}(\phi)=:I^* < \infty$; hence, for $\delta>0$, there exists a $\phi^\delta \in G$ ($\phi(0)=x$) with
$I_{T,x}(\phi^\delta) \leq I^* +\delta$. For small enough $\epsilon=\epsilon(\phi^\delta)>0$, we have 
\[
\Big\{\phi \in D([0,T];A)|\sup_{t\in [0,T]}; |\phi^\delta(t)-\phi(t)|<\epsilon\Big\}\subset G
\] 
and therefore
\[
\mathbb{P}\Big[\sup_{t\in [0,T]}|Z^{N,x}(t)-\phi^\delta(t)|<\epsilon\Big] \leq
\mathbb{P}[Z^{N,x} \in G].
\]
This implies by Theorem~\ref{Theorem5.51a} that for all $\delta>0$,
\begin{align*}
\liminf_{N\rightarrow \infty} \frac{1}{N}\log \mathbb{P}[Z^{N,x}\in G] &\geq \liminf_{N\rightarrow \infty} \frac{1}{N}\log \mathbb{P}\Big[\sup_{t\in [0,T]}|Z^{N,x}(t)-\phi^\delta(t)|<\epsilon\Big] \\
&\geq -I_{T,x}(\phi^\delta) \\
&\geq - I^* - \delta.
\end{align*}
This implies
\[
\liminf_{N\rightarrow \infty} \frac{1}{N}\log \mathbb{P}[Z^{N,x}\in G] 
\geq -I^*
\]
as desired.
\end{proof}

We obtain the following result.

\begin{cor} \label{CorLDPLowerBoundedUniform2}
Assume that $\log \beta_j$ ($j=1,\dots,k$) is bounded and continuous. Then for all $\phi\in D([0,T];A)$ with $\phi(0)=x$ and $\epsilon,\delta>0$, there exists an $\tilde \epsilon>0$ such that
\begin{equation*}
\liminf_{N\rightarrow \infty} \frac 1 N \log \Big(\inf_{z \in A; |x-z|< \tilde \epsilon}\mathbb{P} \Big[ \sup_{t\in [0,T]}|Z^{N,z} - \phi(t)|<\epsilon\Big] \Big)\geq - I_{T,x}(\phi) - \delta.
\end{equation*}
\end{cor}

\begin{proof}
We assume w.l.o.g. that $I_{T,x}(\phi)<\infty$. By Theorem~\ref{Theorem5.51b}, there exists an $N_0$ and $\tilde \epsilon$ such that for $N\geq N_0$ and $z$ with $|z-x|<\tilde \epsilon$,
\[
\frac 1 N \log \mathbb{P} \Big[ \sup_{t\in [0,T]}|Z^{N,z} - \phi(t)|<\epsilon\Big] \geq - \inf_{\tilde \phi: \|\phi-\tilde \phi\|<\epsilon} I_{T,x}(\tilde \phi) - \delta  \geq -I_{T,x}(\phi) -  \delta.
\]
The assertion follows.
\end{proof}

\subsection{LDP lower bound with vanishing rates} \label{SubSecLDPLowerVanish}

In the following, we drop the assumption that the log-rates are bounded. Instead, we rather consider situations, where Assumption~\ref{MainAss} is satisfied. 

We start by some preliminary considerations and assume that Assumption~\ref{MainAss}~(A1) and~(A2) are satisfied. We note that there exists a constant $\alpha>0$ such that for all $x\in A$ there exists a $i\leq I$ such that $B(x,\alpha)\subset B_i$. Indeed, assume that this is incorrect and consider a sequence of points $x_n\in A$ such that $B(x_n,1/n)$ is not contained in any $B_i$. W.l.o.g., we can assume that $x_n \rightarrow x \in A$ (recall that $A$ is compact). As $x\in B_{i_0}$ for some $i_0$, we have $B(x_n,1/n) \subset B_{i_0}$ for $n$ large enough, a contradiction.

\begin{lemma} \label{Lemma2005.3.5}
Assume that $\beta_j$ ($j=1,\dots,k$) is bounded and that Assumption~\ref{MainAss}~(A1)~and (A2) are satisfied. Then, for $T>0$, $K>0$, there exists a $J=J(T,K)\in \mathds{N}$ such that for all $\phi \in D([0,T];A)$ with $I_T(\phi)  \leq K$, there exist
\[
0=t_0 < t_1 < \cdots < t_J=T \text{ and } i_1,\dots,i_J  \text{ such that } \phi(t) \in B_{i_r} \text{ for } t \in [t_{r-1},t_r].
\]
Furthermore, for $r=1,\dots,J$,
\[
\dist(\phi(t_{r-1}),\partial B_{i_r})\geq\alpha \quad \text{and} \quad \dist(\phi(t),\partial B_{i_r})\geq\alpha/2 \quad \text{for } t\in [t_{r-1},t_r]
\]
for $\alpha$ as before.
\end{lemma}

\begin{proof}
By the considerations above, we have $B(x,\alpha)\subset B_{i_1}$ for an appropriate $i_1$. We define 
\[
\tilde t_1:=\inf\{ t\geq 0|B(\phi(t),\alpha/2) \not\subset B_{i_1}\}\wedge T>0.
\]
Now, there exists an $i_2$ such that $B(\phi(t_1),\alpha)\subset B_{i_2}$. If $\tilde t_1<T$, we define
\[
\tilde t_2:=\inf\{ t\geq t_1|B(\phi(t),\alpha/2) \not\subset B_{i_2}\}\wedge T>\tilde t_1.
\]
In the same way, we proceed. By the uniform absolute continuity (Lemma~\ref{Lemma5.18}) of all $\phi$ with $I_T(\phi)\leq K$, we have
\[
\tilde t_r-\tilde t_{r-1} \geq \delta \quad \text{for a constant } \delta>0 \text{ independent of } \phi.
\] 
The assertion hence follows for $J:=\lfloor \frac T \delta \rfloor +1$ and $t_r:=r\delta$ ($r=1,\dots,J-1$), $t_J:=T$.
\end{proof}

We now  define  a function $\phi^\eta$ which is close to a given function $\phi$ with $I_T(\phi)<\infty$. We assume that Assumption~\ref{MainAss}~(A) holds. Hence, for $x \in B_i\cap A$ and $t\in(0,\lambda_2)$,
\[
\di(x+t v_i,\partial A) > \lambda_1 t.
\]
Note that $\lambda_1\leq 1$. Let $\eta>0$ be small. We define for $r=1,\dots,J$, with the notation $\sum_{j=1}^0 ...=0$,
\[
\eta_r:=\eta \sum_{j=1}^{r} \big(\frac{3}{\lambda_1}\big)^{j-1}.
\]
\begin{itemize}
\item
For $r=1,\dots,J$, $t \in \big[t_{r-1}+\eta_{r-1},t_{r-1}+\eta_{r} \big]$,
\[
\phi^\eta(t):=\phi(t_{r-1})+\eta \sum_{j=1}^{r-1} \Big(\frac{3}{\lambda_1}\Big)^{j-1} v_{i_j}+\big(t-t_{r-1}- \eta_{r-1}  \big)v_{i_r}.
\]
\item
For $r=1,\dots,J$, $t \in \big[t_{r-1}+\eta_{r},t_{r}+\eta_{r}\big]$,
\[
\phi^\eta(t):=\phi\big(t-\eta_{r} \big)+\eta \sum_{j=1}^{r} \Big(\frac{3}{\lambda_1}\Big)^{j-1} v_{i_j}.
\]
\end{itemize}
We make the following assumptions on $\eta$:
\[
\eta_J=\eta \sum_{r=1}^J \Big(\frac{3}{\lambda_1}\Big)^{r-1} \leq \frac \alpha 4 \wedge \min_{r=1,\dots,J}|t_{r}-t_{r-1}| \wedge \lambda_2.
\]
Therefore, we have the following properties for $\phi^\eta$:
\begin{itemize}
\item
For $r=1,\dots,J$, $t \in \big[t_{r-1}+\eta_{r-1},t_{r-1}+\eta_{r} \big]$,
\begin{equation}\label{tendto0}
|\phi(t)- \phi^\eta(t)| \leq |\phi(t)-\phi(t_{r-1})|+\eta_{r-1}+(\eta_r-\eta_{r-1})=V_{t_r-t_{r-1}}(\phi) + \eta_r \rightarrow 0 
\end{equation}
as $\eta \rightarrow 0$, where $V_{\cdot}(\phi)$ is the modulus of continuity of $\phi$.
Similarly, for $r=1,\dots,J$, $t \in \big[t_{r-1}+\eta_{r},t_{r}+\eta_{r} \big]$,
\[
|\phi(t)-  \phi^\eta(t)| \leq |\phi(t)-\phi(t-\eta_r)|+\eta_r=V_{\eta_r}(\phi) + \eta_r \rightarrow 0 
\]
as $\eta \rightarrow 0$.
\item
For $r=1,\dots,J$, $t \in \big[t_{r-1}+\eta_{r-1},t_{r-1}+\eta_{r} \big]$,
\[
\dist (\phi^\eta(t), \partial B_{i_r}) \geq \dist (\phi(t_{r-1}), \partial B_{i_r}) -\eta_r \geq \alpha - \frac \alpha 4.
\] 
Similarly, for $r=1,\dots,J$, $t \in \big[t_{r-1}+\eta_{r},t_{r}+\eta_{r} \big]$, hence $t-\eta_r \in[t_{r-1},t_r]$,
\[
\dist (\phi^\eta(t), \partial B_{i_r}) \geq \dist (\phi(t-\eta_r), \partial B_{i_r}) -\eta_r \geq \frac \alpha 2 - \frac \alpha 4.
\] 
Hence, for $r=1,\dots,J$, $t \in \big[t_{r-1}+\eta_{r-1},t_{r}+\eta_{r} \big]$,
\[
\dist (\phi^\eta(t), \partial B_{i_r}) \geq \frac \alpha 4.
\]
\item
For $t\in [0,\eta]$,
\begin{equation} \label{EqTildePhi1}
\dist(\phi^\eta(t), \partial A) \geq t \lambda_1.
\end{equation}
For $t \in \big[\eta,T+\eta_{J}\big]$,
\begin{equation} \label{EqPhiEta1}
\dist (\phi^\eta(t), \partial A) \geq \lambda_1 \eta.
\end{equation}
This can be seen by induction on $r=1,\dots,J$ (the induction hypothesis is clear, cf.~Inequality~\eqref{EqTildePhi1}). For $r=1,\dots,J$, we have (by induction hypothesis and the assumptions on $\eta$)
\[
\phi^\eta \Big(t_{r-1}+\eta_{r-1}\Big) \in  B_{i_r}, \text{ and for }r\ge2,\ \dist \Big(\phi^\eta\Big(t_{r-1}+\eta_{r-1}\Big), \partial A\Big) \geq \eta \lambda_1.
\]
From Assumption~\ref{MainAss}~(A3), the distance of $\phi^\eta(t)$ to the boundary is increasing for $t \in \big[t_{r-1}+\eta \sum_{j=1}^{r-1} \big(\frac{3}{\lambda_1}\big)^{j-1},t_{r-1}+\eta\sum_{j=1}^{r} \big(\frac{3}{\lambda_1}\big)^{j-1}\big]$,  and is at least
\[
\Big(t- \Big( t_{r-1}+\eta \sum_{j=1}^{r-1} \Big(\frac{3}{\lambda_1}\Big)^{j-1}\Big)\Big)\cdot \lambda_1 \vee \eta \lambda_1.
\]
In particular,
\[
\dist \Big(\phi^\eta\Big(t_{r-1}+\eta \sum_{j=1}^{r} \Big(\frac{3}{\lambda_1}\Big)^{j-1}\Big), \partial A\Big) \geq \eta \lambda_1 \Big(\frac{3}{\lambda_1}\Big)^{r-1}.
\]
For $t \in \big[t_{r-1}+\eta \sum_{j=1}^{r} \big(\frac{3}{\lambda_1}\big)^{j-1},t_{r}+\eta\sum_{j=1}^{r} \big(\frac{3}{\lambda_1}\big)^{j-1}\big]$,
we have
\begin{align*}
\phi^\eta(t)&=\phi\Big(t-\eta \sum_{j=1}^{r} \Big(\frac{3}{\lambda_1}\Big)^{j-1}\Big)+ \eta\Big(\frac{3}{\lambda_1}\Big)^{r-1} v_{i_{r}} + \eta \sum_{j=1}^{r-1} \Big(\frac{3}{\lambda_1}\Big)^{j-1} v_{i_j}\\
&=\bar \phi(t)+\eta \sum_{j=1}^{r-1} \Big(\frac{3}{\lambda_1}\Big)^{j-1} v_{i_j}
\end{align*}
and therefore (by elementary calculus and the fact that $|v_i|\leq 1$)
\begin{align*}
\dist(\phi^\eta(t), \partial A) &\geq 
\dist(\bar \phi(t), \partial A)  - 
\Big|\eta \sum_{j=1}^{r-1} \Big(\frac{3}{\lambda_1}\Big)^{j-1} v_{i_j}\Big| \\
&\geq \eta	\left(\frac{3}{\lambda_1}\right)^{r-1}\lambda_1\left(1-\frac{{\bf1}_{r\ge2}}{2}\right)\\
&\geq\eta \lambda_1.
\end{align*}
\end{itemize}

We now have

\begin{lemma} \label{Lemma5.1}
Assume that Assumption~\ref{MainAss} holds. Let $K>0$ and $\epsilon>0$. Then there exists an $\eta_0=\eta_0(T,K,\epsilon)>0$ such that for all $\phi\in D([0,T];A)$ with $I_{T}(\phi)\leq K$ and all $\eta<\eta_0$,
\[
I_{T}(\phi^\eta) \leq I_{T} (\phi) + \epsilon,
\]
where $\phi^\eta(t)$ is defined as above.
\end{lemma}

\begin{proof}
We first use Lemma \ref{lemmaC1} and chose $\eta< \eta_1$ small enough (independent of $i$, $\phi$) such that 
\[
\sum_{r=1}^J \int_{t_{r-1}+\eta_{r-1}}^{t_{r-1}+\eta_{r}} L(\phi^\eta(t),(\phi^\eta)'(t))dt =\sum_{r=1}^J I_{\eta_{r}-\eta_{r-1}} (\phi) < \frac{\epsilon}{2}.
\]
We now denote by $\mu^{*}(t)$ the optimal $\mu$ corresponding to $(\phi(t),\phi'(t))$ (cf.~Lemma~\ref{Lemma3.3}). We let $r=1,\dots,J$ and $t\in[t_{r-1}+\eta_r,t_r+\eta_r]$ and note that $(\phi^\eta)'(t)=\phi'(t-\eta_r)$. By Theorem~\ref{Theorem5.26}, we have
\begin{equation} \label{EqLemma5.1.1}
L(\phi^\eta(t),\phi'(t-\eta_r)) \leq  \ell (\phi^\eta(t),\mu^{*}(t-\eta_r)).
\end{equation}
By the Lipschitz continuity of the $\beta_j$, we have
\begin{equation} \label{EqLemma5.1.2}
|\beta_j(\phi^\eta(t))-\beta_j(\phi(t-\eta_r))| \leq \delta_K(\eta)
\end{equation}
where $\delta_K(\eta)$ is independent of $\phi$ and $\delta_K(\eta) \rightarrow 0$ as $\eta \rightarrow 0$. We deduce 
from \eqref{EqLemma5.1.1} and \eqref{EqLemma5.1.2}
\begin{align}
L(\phi^\eta(t),\phi'(t-\eta_r)) - L(\phi(t-\eta_r),\phi'(t-\eta_r)) &\leq k \delta_K(\eta)+
\sum_j  \mu_j^{*}(t-\eta_r) \log \frac{\beta_j(\phi(t-\eta_r))}{\beta_j(\phi^\eta(t))}.  \label{EqLemma5.1.3}
\end{align}
Let
\[
\tilde v_{i_r}= \Big(\frac{3}{\lambda_1}\Big)^{r-1} v_{i_r} + \sum_{j=1}^{r-1}\Big(\frac{3}{\lambda_1}\Big)^{j-1} v_{i_j},\
\text{ and }\hat v_{i_r}=\frac{\tilde v_{i_r}}{|\tilde v_{i_r}|} \in \mathcal C_{1,i_r}. 
\]

By Assumption~\ref{MainAss}~(B4), there exists a constant $\lambda_4>0$ such that for $z\in B_{i_r}$ and $\eta<\eta_2\leq\eta_1$ small enough (note that $\eta_2$ depends on $\lambda_1$ and $\lambda_2$ but not directly on $\phi$, except  through $K$),
\begin{equation} \notag
\beta_j(z) <\lambda_4\Rightarrow \beta_j(z+\eta \tilde v_{i_r})\geq \beta_j(z),
\end{equation}
hence
\begin{equation} \label{EqLemma5.1.4}
\log \frac{\beta_j(\phi(t-\eta_r))}{\beta_j(\phi^\eta(t))} < 0 \quad \text{if } \beta_j(\phi(t-\eta_r)) <  \lambda_4. 
\end{equation}
If $\beta_j(\phi(t-\eta_r))\geq  \lambda_4$  (recall the definition of $\delta_K(\eta)$ and choose $\eta<\eta_3<\eta_2$ small enough such that 
$\delta_K(\eta)<\lambda_4/2$)
\begin{align}
\log \frac{\beta_j(\phi(t-\eta_r))}{\beta_j(\phi^\eta(t))} &\leq \log \frac{\beta_j(\phi(t-\eta_r))}{\beta_j(\phi(t-\eta_r))-\delta_K(\eta)} \notag \\*
&\leq \log \frac{\lambda_4}{\lambda_4-\delta_K(\eta)} \notag \\
 &= \log\frac{1}{1-\delta_K(\eta)/\lambda_4} \notag \\
&\leq \frac{2 \delta_K(\eta)}{\lambda_4}, \label{EqLemma5.1.5}
\end{align}
since $\log(1/(1-x)) < 2x$ for $0<x\leq 1/2$.

From Lemma~\ref{Lemma5.17} and Lemma~\ref{Lemma3.3}, there exist (universal, i.e., independent of $x$) constants $B\geq B_1\vee B_2$, $B>1$, $C_1$, $C_3$ such that for $|y| \geq B$, and $x \in A$,
\begin{equation} \label{EqLemma5.1.6}
L(x,y) \geq C_1 |y| \log |y|,
\end{equation}
\begin{equation} \label{EqLemma5.1.7}
|\mu^{*}|=|\mu^*(x,y)|\leq C_3 |y|.
\end{equation}
Hence if $|\phi'(t-\eta_r)|\geq B$, using \eqref{EqLemma5.1.3}, \eqref{EqLemma5.1.4}, \eqref{EqLemma5.1.5} and \eqref{EqLemma5.1.7} for the first inequality and
\eqref{EqLemma5.1.6} for the second, we get
\begin{align}
L(\phi^\eta(t),\phi'(t-\eta_r)) - L(\phi(t-\eta_r),\phi'(t-\eta_r)) &\leq
	k\delta_K(\eta)+ k C_3|\phi'(t-\eta_r)| \frac{2 \delta_K(\eta)}{\lambda_4} \notag \\
&\leq k\delta_K(\eta)+ k C_3 \frac{2 \delta_K(\eta)L(\phi(t-\eta_r),\phi'(t-\eta_r))}{C_1\lambda_4 \log|\phi'(t-\eta_r)|}. \label{EqLemma5.1.8}
\end{align}
If however $|\phi'(t-\eta_r)|<B$, Lemma~\ref{Lemma3.3} implies similarly as before  that
$|\mu^{*}(t-\eta_r)| \leq C_3 B$. From \eqref{EqLemma5.1.3}, we deduce
\begin{equation}
L(\phi^\eta(t),\phi'(t-\eta_r)) - L(\phi(t-\eta_r),\phi'(t-\eta_r)) \leq
	k \delta_K(\eta)+ k C_3 B  \frac{2 \delta_K(\eta)}{\lambda_4}.  \label{EqLemma5.1.9}
\end{equation}
Inequalities~\eqref{EqLemma5.1.8} and~\eqref{EqLemma5.1.9} imply
\begin{align}
L(\phi^\eta(t),\phi'(t-\eta_r)) - L(\phi(t-\eta_r),\phi'(t-\eta_r))\leq \delta_{1,K}(\eta) + \delta_{2,K}(\eta)L(\phi(t-\eta_r),\phi'(t-\eta_r)) \notag
\end{align}
with constants $\delta_{i,K}(\eta) \rightarrow 0$ as $\eta \rightarrow 0$. We can hence choose $\eta<\eta_4<\eta_3$ small enough such that
\[
\sum_{r=1}^J \int_{t_{r-1}+\eta_r}^{t_r+\eta_r} L(\phi^\eta(t),\phi'(t-\eta_r))dt - \sum_{r=1}^J \int_{t_{r-1}+\eta_r}^{t_r+\eta_r}L(\phi(t-\eta_r),\phi'(t-\eta_r)) dt < \frac{\epsilon}{2}.
\]
This yields the result.
\end{proof}

\color{black}

The following lemma is the main difference to the corresponding result of~\cite{Shwartz2005}. We transform the LLN from Assumption~\ref{MainAss}~(C) to a LDP lower bound for linear functions following the vector $v_i$ near the boundary.
%\color{red}
%I have made the following changes: $x^N$ instead of $x$ for initial value. This does not change a lot but some things have to be adapted. 
%These are in red if I think you should note them. 
%\color{black}

In the next statement, $\alpha$ is the exponent which appears in the Assumption \ref{MainAss} (C).
\begin{lemma} \label{LemmaLPDLowerLinear}
Assume that Assumption~\ref{MainAss} holds. Let $i\leq I_1$, $x\in A \cap B_i$ 
and $x^N\in A^N\cap B_i$ such that 
\[
\limsup_{N\rightarrow \infty}|x^N-x| N^\alpha <1.
\]
Let furthermore $\epsilon>0$ and define $\mu^i$, $\phi^x$ and $\eta_0$ as in Assumption~\ref{MainAss}~(C). Then for all $\eta$ small enough, 
all $\epsilon$,
\[
\liminf_{N\rightarrow \infty} \frac{1}{N} \log \mathbb{P}\Big[ \sup_{t\in [0,\eta]}|Z^{N,x^N}(t)-\phi^x(t)|<\epsilon\Big] 
\geq - \int_0^\eta   \ell (\phi^x(t),\mu^i)dt,
\]
and the above convergence is uniform in $x\in A$.
\end{lemma}

\begin{proof}
The proof follows the same line of reasoning as the proof of Proposition~\ref{Prop5.53} but is technically more involved.

For simplicity, let $N$ be large enough and $\eta <\eta_0$ (for $\eta_0$ as in~Assumption~\ref{MainAss}~(C)) be small enough such that $\phi^{x^N}(t)\in B_i$ for all $t\leq \eta$. We furthermore let 
\[
\tilde \epsilon < \epsilon_1:=\epsilon \wedge \lambda_1 \eta
\]
Define the set
\[
F^{N}:=\Big\{ \sup_{t \in [0,\eta]} |Z^{N,x^N}(t)-\phi^x(t)| < \tilde \epsilon  \Big\}.
\]
Let $\xi_\eta=\xi^{N}_\eta=\frac{d \tilde{\mathbb{P}}}{d \mathbb{P}}|_{\mathcal{F}_\eta}$ be given as in Theorem~\ref{TheoremB6} for the rates $\tilde \beta_j=\tilde \mu^i_j$.%\footnote{If we take the rates $\mu^i_j$ instead of $\tilde \mu^i_j$, we do not have $\tilde{\mathbb{P}}|_{\mathcal F_\eta} \ll \mathbb{P}|_{\mathcal F_\eta}$.} 
  We note that due to Assumption~\ref{MainAss}~(C),
\begin{equation} \label{EqLemmaLPDLowerLinear}
\tilde{\mathbb{P}}[F^N] \geq 1- \delta(N,\tilde{\epsilon}) \rightarrow 1 \quad \text{as } N\rightarrow \infty.
\end{equation}
From Corollary~\ref{CorB.6}, \eqref{EqLemmaLPDLowerLinear} and Jensen's inequality, we deduce that
\begin{align}
&\liminf_{N\rightarrow \infty} \frac{1}{N} \log  \mathbb{P}\Big[\sup_{t \in [0,\eta]} | Z^{N,x^N}(t) - \phi^x(t)| < \epsilon\Big] \notag\\*
&\qquad\geq
\liminf_{N\rightarrow \infty} \frac{1}{N} \log  \mathbb{P}\big[F^N\big] \notag \\
&\qquad\geq \liminf_{N\rightarrow\infty} \frac 1 N \log \tilde{\mathbb{E}}\big[\xi_\eta^{-1} \mathds{1}_{F^N}\big] \notag \\
&\qquad=\liminf_{N\rightarrow\infty} \frac 1 N \log \Big\{\tilde{\mathbb{P}}\big[F^N\big] \tilde{\mathbb{E}}_{F^N}\big[\exp(X_\eta)\big] \Big\}\notag \\
&\qquad\geq\liminf_{N\rightarrow\infty} \frac 1 N \log \tilde{\mathbb{P}}\big[F^N\big] +\liminf_{N\rightarrow\infty} \frac 1 N \log \tilde{\mathbb{E}}_{F^N}\big[\exp(X_\eta)\big] \notag \\
&\qquad \geq \liminf_{N\rightarrow \infty}    \tilde{\mathbb{E}}_{F^{N}}\Big[\frac{X_\eta}{N} \Big], \label{EqLemmaLPDLowerLinear.1} 
\end{align}
where $\tilde{\mathbb{E}}_{F^{N}}[X]:=\tilde{\mathbb{E}}[X | F^{N}]$
and
\begin{align*}
X_\eta^{N}:=X_\eta:=\log \xi_\eta^{-1}&=\sum_{\tau\leq \eta} \Big[  \log\beta_{j(\tau)}(Z^{N,x^N}(\tau-) - \log\tilde \mu^i_{j(\tau)}(Z^{N,x^N}(\tau-) \Big] \\*
&\qquad +N \sum_j\int_0^{\eta} \big( \tilde\mu^i_j(Z^{N,x^N}(t) - \beta_j(Z^{N,x^N}(t))\big) dt.
\end{align*}

We have $\dist(\phi^{x^N}(t),\partial A)\geq \lambda_1 t$ (cf.~Assumption~\ref{MainAss}~(A3)) and therefore on $F^N$,
\[
\dist(Z^{N,x^N}(t),\partial A) > \lambda_1 t - \tilde\epsilon \quad \text{for } t \in \Big[\frac{\tilde \epsilon}{\lambda_1}, \eta\Big] .
\]
Consequently
\[
\tilde \mu^i_j(Z^{N,x^N}(t))=\mu^i_j \quad \text{for all } j \text{ and for all } t \in \Big[\frac{\tilde \epsilon}{\lambda_1}, \eta\Big].
\]
We obtain
\begin{align} 
&\frac{1}{\tilde{\mathbb{P}}[F^N]} \tilde{\mathbb{E}} \Big[\mathds{1}_{F^N} \int_0^\eta \sum_j \tilde \mu^i_j(Z^{N,x^N}(t)) dt \Big] \notag \\*
&\qquad  =  \frac{1}{\tilde{\mathbb{P}}[F^N]} \Big(\tilde{\mathbb{E}} \Big[\mathds{1}_{F^N} \int_{\tilde \epsilon/\lambda_1}^{\eta} \sum_{j=1}^k  
\tilde \mu^i_j(Z^{N,x^N}(t)) dt \Big] + \tilde{\mathbb{E}} \Big[\mathds{1}_{F^N} \int_{0}^{\tilde \epsilon/\lambda_1} \sum_{j=1}^k  \tilde \mu^i_j(Z^{N,x^N}(t))dt \Big]
\Big) \notag \\
&\qquad  =: \sum_{j=1}^k \int_{\tilde \epsilon/\lambda_1}^{\eta} \mu^i_j dt + X_1^{N}(\tilde \epsilon),  \label{EqLemmaLPDLowerLinear.2}
\end{align}
since $\mu^i_j(Z^{N,x^N}(t))=\mu^i_j$ on $F^N$.
We  note that for all $N$,
\begin{equation} \label{EqLemmaLPDLowerLinear.3}
|X_1^{N}(\tilde \epsilon)| \leq \frac{\tilde \epsilon}{\lambda_1} k \bar \mu \quad \text{where } \bar \mu:=\max_{j=1,\dots,k} \mu^i_j.
\end{equation}

By the fact that the $\beta_j$'s are bounded and continuous and by Theorem~\ref{ThLLN}, we have for $j=1,\dots,k$,
\[
\sup_{t \in [0,\eta]} | \beta_j(Z^{N,x^N}(t)) - \beta_j(\phi^x(t)) | \rightarrow 0 \quad \text{a.s.~as } N\rightarrow \infty.
\] 
 Combined with \eqref{EqLemmaLPDLowerLinear}, this implies
\begin{equation}
\frac{1}{\tilde{\mathbb{P}}[F^N]} \tilde{\mathbb{E}} \Big[\mathds{1}_{F^N} \int_0^{\eta} \sum_{j=1}^k \beta_j(Z^{N,x^N}(t)) dt \Big]
\longrightarrow \sum_{j=1}^k \int_0^{\eta} \beta_j(\phi^x(t)) dt  \label{EqLemmaLPDLowerLinear.4}
\end{equation}
as $N\rightarrow \infty$.

Let us now define the following processes. For $z\in A$, $j=1,\dots,k$  and $0\leq s<t \leq \eta$ let
$\bar Z^{N,z}$ solves Equation~\eqref{EqPoisson} with constant rates $\mu^i_j$ under $\tilde{\mathbb{P}}$, and
\begin{align*}
Y_j^{N,z,s,t}&:=\frac{1}{N} \cdot  \#\text{jumps of } Z^{N,z} \text{ in direction } h_j \text{ in }[s,t], \\
\bar Y_j^{N,z,s,t}&:=\frac{1}{N} \cdot  \#\text{jumps of } \bar Z^{N,z} \text{ in direction } h_j \text{ in }[s,t],
\end{align*}
We have for any event $F$, noting that 
$\tilde{\mathbb{P}}[F]-\tilde{\mathbb{P}}[F]^2=\tilde{\mathbb{P}}[F^c]-\tilde{\mathbb{P}}[F^c]^2$,
\begin{align} \label{EqLemmaLPDLowerLinear.5a}
\tilde{\mathbb{E}}[Y_j^{N,z,s,t}] &\leq (t-s) \mu^i_j =\tilde{\mathbb{E}}[\bar Y_j^{N,z,s,t}], \\
\label{EqLemmaLPDLowerLinear.5b}
\widetilde{\Var}[Y_j^{N,z,s,t}] &\leq (t-s)  \mu^i_j
=\widetilde{\Var}[\bar Y_j^{N,z,s,t}], 
\\
|\widetilde{\Cov}(\mathds{1}_{F}, Y_j^{N,z,s,t})|, 
\,  |\widetilde{\Cov}(\mathds{1}_{F}, \bar Y_j^{N,z,s,t})| 
&\leq  \sqrt{\widetilde{\Var}[\mathds{1}_{F}]} \sqrt{\widetilde{\Var}[ \bar Y_j^{N,z,s,t}]} \notag \\*
&\leq \sqrt{\tilde{\mathbb{P}}[F]-\tilde{\mathbb{P}}[F]^2} \sqrt{(t-s) \mu^i_j}. \label{EqLemmaLPDLowerLinear.5c}
\end{align}

We define the sets
\[
F_1^{N}:=\Big\{ \Big|Z^{N,x^N}\Big(\frac{2\tilde \epsilon}{\lambda_1}\Big)-\phi^x\Big(\frac{2\tilde \epsilon}{\lambda_1}\Big)\Big|<\frac{\tilde \epsilon}{2} \Big\} \in \mathcal{F}_{2 \tilde \epsilon/\lambda_1}
\]
and for $z\in A$ with $|z-\phi^x(2\tilde \epsilon/\lambda_1)|<\tilde \epsilon/2$,
\[
F_2^{N,z}:=\Big\{ \sup_{t\in[0, \eta -2 \tilde \epsilon/\lambda_1]} |Z^{N,z}(t)-\phi^z(t)|<\frac{\tilde \epsilon}{2} \Big\}.
\]
Note that
\[
\dist(\phi^x(t),\partial A) \geq 2 \tilde \epsilon \quad \text{for } t\in \Big[\frac{2 \tilde \epsilon}{\lambda_1},\eta\Big]
\]
and whenever  $|z-\phi^x(2\tilde \epsilon/\lambda_1)|<\tilde \epsilon/2$,
\[
\Big|\phi^z(t)-\phi^x\Big(t+ \frac{2\tilde \epsilon}{\lambda_1} \Big)\Big| < \frac{\tilde \epsilon}{2} \quad \text{for } t\in \Big[0,\eta -\frac{2 \tilde \epsilon}{\lambda_1}\Big].
\]
Hence 
\[
\dist(Z^{N,z}(t),\partial A) \geq  \tilde \epsilon \quad \text{for } t\in \Big[0,\eta-\frac{2 \tilde \epsilon}{\lambda_1}\Big]
\]
and therefore $Z^{N,z}(t)=\bar Z^{N,z}(t)$ on $F_2^{N,z}$. This implies
\begin{equation} \label{EqLemmaLPDLowerLinear.6.1}
F^{N,z}_2=\Big\{ \sup_{t\in[0, \eta -2 \tilde \epsilon/\lambda_1]} |\bar Z^{N,z}(t)-\phi^z(t)|<\frac{\tilde \epsilon}{2} \Big\}.
\end{equation}
We now let
\[
\frac{2 \tilde \epsilon}{\lambda_1} \leq s < t \leq \eta
\]
and compute (by using the Markov property of $Z^N$ and the fact that $Y_j^{N,z,s-2\tilde \epsilon/\lambda_1,t-2 \tilde \epsilon/\lambda_1}
=\bar Y_j^{N,z,s-2\tilde \epsilon/\lambda_1,t-2 \tilde \epsilon/\lambda_1}$ on the event $F_2^{N,z}$)
\begin{align}
\tilde{\mathbb{E}}[Y_j^{N,x^N,s,t}]
	& = \tilde{\mathbb{E}}[\mathds{1}_{F^N_1}Y_j^{N,x^N,s,t}]
			+\tilde{\mathbb{E}}[\mathds{1}_{(F^N_1)^c}Y_j^{N,x^N,s,t}] \notag \\
	&\geq\tilde{\mathbb{P}}[F^N_1]\cdot \inf_{z; |z - \phi^x(\tilde \epsilon/\lambda_1)|<\tilde \epsilon/2} \tilde{\mathbb{E}}[Y_j^{N,z,s-2\tilde \epsilon/\lambda_1,t-2 \tilde \epsilon/\lambda_1}]
		+\tilde{\mathbb{E}}[\mathds{1}_{(F^N_1)^c}Y_j^{N,x^N,s,t}] \notag \\
	&\geq\tilde{\mathbb{P}}[F^N_1]\cdot \inf_{z; |z - \phi^x(\tilde \epsilon/\lambda_1)|<\tilde \epsilon/2} \tilde{\mathbb{E}}[\mathds{1}_{F^{N,z}_2} 
	Y_j^{N,z,s-2\tilde \epsilon/\lambda_1,t-2 \tilde \epsilon/\lambda_1}] \notag \\*
		&\qquad +\tilde{\mathbb{P}}[F^N_1]\cdot \inf_{z; |z - \phi^x(\tilde \epsilon/\lambda_1)|<\tilde \epsilon/2} \tilde{\mathbb{E}}[\mathds{1}_{(F^{N,z}_2)^c}Y_j^{N,z,s-2\tilde \epsilon/\lambda_1,t-2 \tilde \epsilon/\lambda_1}] \notag \\*
				&\qquad +\tilde{\mathbb{E}}[\mathds{1}_{(F^N_1)^c}Y_j^{N,x^N,s,t}] \notag \\
	& = \mu_j^i (t-s) \tilde{\mathbb{P}}[F^N_1] \cdot \inf_{z; |z - \phi^x(\tilde \epsilon/\lambda_1)|<\tilde \epsilon/2} \tilde{\mathbb{P}}[F^{N,z}_2] \notag \\*
			&\qquad +\tilde{\mathbb{P}}[F^N_1]\cdot \inf_{z; |z - \phi^x(\tilde \epsilon/\lambda_1)|<\tilde \epsilon/2} \widetilde{\Cov}(\mathds{1}_{F^{N,z}_2},Y_j^{N,z,s-2\tilde \epsilon/\lambda_1,t-2 \tilde \epsilon/\lambda_1}) \notag \\*
			&\qquad +\tilde{\mathbb{P}}[F^N_1]\cdot \inf_{z; |z - \phi^x(\tilde \epsilon/\lambda_1)|<\tilde \epsilon/2} \tilde{\mathbb{E}}[\mathds{1}_{(F^{N,z}_2)^c}Y_j^{N,z,s-2\tilde \epsilon/\lambda_1,t-2 \tilde \epsilon/\lambda_1}] \notag \\*
			&\qquad +\tilde{\mathbb{E}}[\mathds{1}_{(F^N_1)^c}Y_j^{N,x^N,s,t}]\label{EqLemmaLPDLowerLinear.6}	\\
	& \geq \mu_j^i (t-s) \tilde{\mathbb{P}}[F^N_1] \cdot \inf_{z; |z - \phi^x(\tilde \epsilon/\lambda_1)|<\tilde \epsilon/2} \tilde{\mathbb{P}}[F^{N,z}_2] 
		\notag \\*
			&\qquad -\tilde{\mathbb{P}}[F^N_1]\cdot \inf_{z; |z - \phi^x(\tilde \epsilon/\lambda_1)|<\tilde \epsilon/2}\big| \widetilde{\Cov}(\mathds{1}_{F^{N,z}_2},Y_j^{N,z,s-2\tilde \epsilon/\lambda_1,t-2 \tilde \epsilon/\lambda_1})\big|
			\label{EqLemmaLPDLowerLinear.7}
\end{align}
as the third and the fourth term in \eqref{EqLemmaLPDLowerLinear.6} are non-negative. 
%We also note that these terms tend to zero as $N\rightarrow \infty$ by Theorem~\ref{ThLLN} (for the third term; recall the uniformity 
%in the convergence with respect to initial values) respectively Assumption~\ref{MainAss}~(C) (for the fourth term). 
As furthermore
\[
\tilde{\mathbb{P}}[F^N_1] \cdot \inf_{z; |z - \phi^x(\tilde \epsilon/\lambda_1)|<\tilde \epsilon/2} \tilde{\mathbb{P}}[F^{N,z}_2] \rightarrow 1
\]
and
\[
\tilde{\mathbb{P}}[F^N_1]\cdot \inf_{z; |z - \phi^x(\tilde \epsilon/\lambda_1)|<\tilde \epsilon/2} \widetilde{\Cov}(\mathds{1}_{F^{N,z}_2},Y_j^{N,z,s-2\tilde \epsilon/\lambda_1,t-2 \tilde \epsilon/\lambda_1}) \rightarrow 0
\]
as $N\rightarrow \infty$ by Assumption~\ref{MainAss}~(C), Theorem~\ref{ThLLN}  and \eqref{EqLemmaLPDLowerLinear.5c}. Combinig the resulting
inequality with \eqref{EqLemmaLPDLowerLinear.5a} for all $\tilde \epsilon<\epsilon_1$ and $2 \tilde \epsilon/\lambda_1 \leq s  < t$, we deduce that
\begin{equation}
\notag
\lim_{N\rightarrow \infty} \tilde{\mathbb{E}} [Y_j^{N,x^N,s,t}]= \mu^i_j (t-s).
\end{equation}
Note that for $0\leq s < t \leq \eta$, and all $\tilde \epsilon<\epsilon_1$
\[
Y_j^{N,x^N,s,t} = Y_j^{N,x^N,s,(s\vee2\tilde \epsilon/\lambda_1)\wedge t}+ Y_j^{N,x^N,(s\vee2\tilde \epsilon/\lambda_1)\wedge t,t}
\]
and hence also for $0\leq s < t \leq \eta$, since when $s<2\tilde \epsilon/\lambda_1$, $Y_j^{N,x^N,s,2\tilde \epsilon/\lambda_1\wedge t}$ is of the order of 
$\tilde \epsilon$,
\begin{equation}
\label{EqLemmaLPDLowerLinear.8}
\lim_{N\rightarrow \infty} \tilde{\mathbb{E}} [Y_j^{N,x^N,s,t}]= \mu^i_j (t-s).
\end{equation}

Let now $\tau_j \in [0,\eta]$ denote the jump times of $Z^{N,x^N}$ in direction $h_j$. Since
$\tilde \mu^i_j(Z^{N,x^N}(\tau_j-))= \log \mu^i_j$  $\tilde{\mathbb{P}}$ a.s.,
\begin{align}
&\frac{1}{\tilde{\mathbb{P}}[F^N]} 
	\sum_{j; \mu^i_j>0} 
		\tilde{\mathbb{E}} \Big[\frac{1}{N} \mathds{1}_{F^N} \sum_{\tau_j\leq \eta} 
			\log \tilde \mu^i_j(Z^{N,x^N}(\tau_j-))\Big] \notag \\*
&\qquad =\frac{1}{\tilde{\mathbb{P}}[F^N]} 
	\sum_{j; \mu^i_j>0} 
		\log \mu^i_j \,\, \tilde{\mathbb{E}} \big[ \mathds{1}_{F^N} Y^{N,x^N,0,\eta}_j\big]  \notag \\
&\qquad =\frac{1}{\tilde{\mathbb{P}}[F^N]} 
	\sum_{j; \mu^i_j>0} \log \mu^i_j 
		\Big\{ \tilde{\mathbb{P}} [F^N]\cdot 
			\tilde{\mathbb{E}}\big[  Y^{N,x^N,0,\eta}_j\big]
			+ \widetilde{\Cov}(\mathds{1}_{F^N}, Y^{N,x^N,0,\eta}_j) \Big\} \notag \\*		&\qquad \longrightarrow \sum_{j; \mu^i_j>0}\eta \mu^i_j \log \mu^i_j  =\sum_{j=1}^k \int_0^{\eta} \mu^i_j \log \mu^i_j dt \label{EqLemmaLPDLowerLinear.9}
\end{align}
as $N\rightarrow \infty$ by \eqref{EqLemmaLPDLowerLinear}, \eqref{EqLemmaLPDLowerLinear.5c} and \eqref{EqLemmaLPDLowerLinear.8}.

For the last and most extensive step of the proof, we define for $\tilde \epsilon < \epsilon_1$ and (cf.~Assumption~\ref{MainAss}~(C) and \eqref{EqEpsilonNAlpha})
\[
 \epsilon_N = \frac{1}{N^\alpha} 
\]
and the set
\[
\tilde F^N:=\Big\{\sup_{t\in [0,\eta]} | Z^{N,x^N}(t) - \phi^x(t)|  < \epsilon_N \Big\}.
\]
We assume w.l.o.g.~that from now on $N$ is large enough (cf.~Assumption~\ref{MainAss}~(C)) such that
\[
\tilde{\mathbb{P}}\big[\tilde F^N\big] \geq 
\tilde{\mathbb{P}}\Big[\sup_{t\in [0,\eta]}|Z^{N,x^N}(t) - \phi^{x^N}(t) | < \epsilon_N\Big] \geq 1-\delta(N,\epsilon_N),
\]
where $\delta(N,\epsilon_N)\to0$
as $N\rightarrow \infty$. We note that we have for all $\tilde \epsilon \leq \epsilon_1$, $j=1,\dots,k$ with $\mu^i_j>0$, $N\in \mathds{N}$ and $t\in [2\tilde \epsilon/\lambda_1,\eta]$,
\begin{equation} \notag
\log \beta_j(\phi^x(t)), \log\beta_j(Z^{N,x^N}(t)) \geq \log \underline{\beta}(\tilde \epsilon)>0
\quad
\text{on } F^N \text{ and } \tilde F^N.
\end{equation}
We compute
\begin{align}
&\frac{1}{\tilde{\mathbb{P}}[F^N]} 
	\tilde{\mathbb{E}} \Big[\frac{1}{N}\mathds{1}_{F^N} \sum_{\tau \leq \eta} \log \beta_{j(\tau)}(Z^{N,x^N}(\tau-)) \Big] \notag \\
&\qquad 
	= \frac{1}{\tilde{\mathbb{P}}[F^N]}  \sum_{j,\, \mu^i_j>0} \tilde{\mathbb{E}} \Big[\frac{1}{N}\mathds{1}_{\tilde F^N} \sum_{\tau_j\in [2 \tilde \epsilon/\lambda_1,\eta]} \log \beta_{j}(Z^{N,x^N}(\tau_j-)) \Big] \notag \\
&\qquad \qquad  +\frac{1}{\tilde{\mathbb{P}}[F^N]} \sum_{j,\, \mu^i_j>0} \tilde{\mathbb{E}} \Big[\frac{1}{N}\mathds{1}_{\tilde F^N} \sum_{\tau_j\in [0,2 \tilde \epsilon/\lambda_1]} \log \beta_{j}(Z^{N,x^N}(\tau_j-)) \Big] \notag \\
&\qquad \qquad  +\frac{1}{\tilde{\mathbb{P}}[F^N]} \sum_{j,\, \mu^i_j>0} \tilde{\mathbb{E}} \Big[\frac{1}{N}\mathds{1}_{F^N\setminus \tilde F^N} \sum_{\tau_j\in [0,\eta]} \log \beta_{j}(Z^{N,x^N}(\tau_j-)) \Big]. \label{EqLemmaLPDLowerLinear.10}
\end{align}

Let us first consider the first term in Equation~\eqref{EqLemmaLPDLowerLinear.10}. As for all $j$, 
\[
\log \beta_j(\cdot):\tilde A(\tilde \epsilon):= \{z \in A|\dist(z, \partial A)\geq \tilde \epsilon\} \rightarrow \mathds{R}
\] 
is uniformly continuous, there exit constants $\tilde{\delta}_N>0$ with $\tilde \delta_N \downarrow 0$ such that
\begin{equation}  \label{EqLemmaLPDLowerLinear.11}
z, \tilde z \in \tilde A(\tilde \epsilon), \, |\tilde z-z|<  3 \epsilon_N  \Rightarrow |\log\beta_j( \tilde z) - \log \beta_j(z)|< \tilde \delta_N.
\end{equation}
We define
\[
M = M(N) := \lfloor (\eta- 2 \tilde \epsilon /\lambda_1) \epsilon_N^{-1}+1 \rfloor
\]
and divide the interval $[2 \tilde \epsilon/\lambda_1,\eta]$ into $M$ equidistant subintervals $[t_r,t_{r+1}]$ ($r=0,\dots,M-1$, $t_r=t_r(N)$) of length $\Delta = \Delta (N)$, i.e. (for $N$ large enough),
\[
\frac{\epsilon_N}{2}\leq \Delta < \epsilon_N.
\]
For $j=1,\dots,k$, $r=0,\dots,M-1$ and $\tau_j,t \in [t_r,t_{r+1}]$ we have,
\[
|Z^{N,x^N}(\tau_j-) - \phi^x(t)| \leq 2 \epsilon_N \quad \text{on } \tilde F^N,
\]
since $|\phi^x(\tau_j)-\phi^x(t)|\leq |\tau_j-t|$ as $|v_i| \leq 1$, and hence (cf.~\eqref{EqLemmaLPDLowerLinear.11})
\begin{equation*} \label{EqLemmaLPDLowerLinear.12}
\inf_{ t \in [t_r,t_{r+1}]} \log \beta_j(\phi^x( t))  - \tilde \delta_N
\leq   \log \beta_j(Z^{N,x^N}(\tau_j-)).
\end{equation*}
From this inequality, we deduce
\begin{align}
&\frac{1}{\tilde{\mathbb{P}}[F^N]} 
	\tilde{\mathbb{E}} \Big[\frac{1}{N}\mathds{1}_{\tilde F^N} \sum_{\tau\in [2 \tilde \epsilon/\lambda_1,\eta]} \log \beta_{j(\tau)}(Z^{N,x^N}(\tau-)) \Big] \notag \\*
&\qquad 
	= \frac{1}{\tilde{\mathbb{P}}[F^N]}  \sum_{j,\, \mu^i_j>0} \sum_{r=0}^{M-1} \tilde{\mathbb{E}} \Big[\frac{1}{N}\mathds{1}_{\tilde F^N} \sum_{\tau_j\in [t_r,t_{r+1})} \log \beta_{j}(Z^{N,x^N}(\tau_j-)) \Big] \notag \\
&\qquad \geq \frac{1}{\tilde{\mathbb{P}}[F^N]} \sum_{j,\, \mu^i_j>0} \sum_{r=0}^{M-1} \Big( \inf_{t \in [t_r,t_{r+1})} \log\beta_j(\phi^x(t)) -\tilde \delta_N \Big) \tilde{\mathbb{E}} \Big[\mathds{1}_{\tilde F^N} Y^{N,x^N,t_r,t_{r+1}}_j \Big] \notag \\
&\qquad \geq\frac{1}{\tilde{\mathbb{P}}[F^N]} \sum_{j,\, \mu^i_j>0} \sum_{r=0}^{M-1}  \inf_{t \in [t_r,t_{r+1})} \log\beta_j(\phi^x(t)) \tilde{\mathbb{P}} [\tilde F^N]
	\tilde{\mathbb{E}} \big[ Y^{N,x^N,t_r,t_{r+1}}_j \big] \notag \\*
&\qquad \qquad -	\frac{1}{\tilde{\mathbb{P}}[F^N]} \sum_{j,\, \mu^i_j>0} \sum_{r=0}^{M-1}  \big| \log\underline \beta(\tilde \epsilon)\big| \big| 	 \widetilde{\Cov}(\mathds{1}_{\tilde F^N}, Y^{N,x^N,t_r,t_{r+1}}_j) \big| \notag  \\* 
	 &\qquad \qquad  - \frac{1}{\tilde{\mathbb{P}}[F^N]}\tilde \delta_N \sum_{j,\, \mu^i_j>0}  \tilde{\mathbb{E}} \big[\mathds{1}_{\tilde F^N}  Y^{N,x^N,2 \tilde \epsilon/\lambda_1,\eta}]. \label{EqLemmaLPDLowerLinear.13}
\end{align}	
The second term in Inequality~\eqref{EqLemmaLPDLowerLinear.13} satisfies (cf.~Inequality~\eqref{EqLemmaLPDLowerLinear.5c} and Assumption~\ref{MainAss}~(C); we assume that $N$ is sufficiently large such that $M\leq 2\epsilon_N^{-1}\eta$),
\begin{align}
&\frac{1}{\tilde{\mathbb{P}}[F^N]} \sum_{j,\, \mu^i_j>0} \sum_{r=0}^{M-1}  \big| \log\underline \beta(\tilde \epsilon)\big| \big| 	 \widetilde{\Cov}(\mathds{1}_{\tilde F^N}, Y^{N,x^N,t_r,t_{r+1}}_j) \big|
\notag \\*
&\qquad \leq \frac{1}{\tilde{\mathbb{P}}[F^N]} 2 k \eta  \big| \log\underline \beta(\tilde \epsilon)\big| \epsilon_N^{-1} \sqrt{\bar \mu\; \epsilon_N}  \sqrt{\delta(N,\epsilon_N)} \notag \\*
& \qquad \rightarrow 0 \label{EqLemmaLPDLowerLinear.14}
\end{align}
as $N \rightarrow \infty$. The third term in Equation~\eqref{EqLemmaLPDLowerLinear.13} satisfies
\begin{equation}
\frac{1}{\tilde{\mathbb{P}}[F^N]}\tilde \delta_N \sum_{j,\, \mu^i_j>0}  \tilde{\mathbb{E}} \big[\mathds{1}_{\tilde F^N}  Y^{N,x^N,2 \tilde \epsilon/\lambda_1,\eta}] 
 \leq \frac{1}{\tilde{\mathbb{P}}[F^N]}\tilde \delta_N k \bar \mu \eta \rightarrow 0
\label{EqLemmaLPDLowerLinear.15}
\end{equation}
as $N \rightarrow \infty$. Finally, let us consider the first term in Equation~\eqref{EqLemmaLPDLowerLinear.13}. Recall that by \eqref{EqLemmaLPDLowerLinear.5a} and~\eqref{EqLemmaLPDLowerLinear.7}, we have
\begin{align*}
\mu_j^i(t_{r+1}-t_r) & \geq \tilde{\mathbb{E}}[Y_j^{N,x^N,t_r,t_{r+1}}] \\
& \geq \mu_j^i (t_{r+1}-t_r) \tilde{\mathbb{P}}[F^N_1] \cdot \inf_{z; |z - \phi^x(\tilde \epsilon/\lambda_1)|<\tilde \epsilon/2} \tilde{\mathbb{P}}[F^{N,z}_2] 
		\notag \\*
			& \qquad \qquad -\tilde{\mathbb{P}}[F^N_1]\cdot \inf_{z; |z - \phi^x(\tilde \epsilon/\lambda_1)|<\tilde \epsilon/2}\big| \widetilde{\Cov}(\mathds{1}_{F^{N,z}_2},Y_j^{N,z,t_r-2\tilde \epsilon/\lambda_1,t_{r+1}-2 \tilde \epsilon/\lambda_1})\big|.
\end{align*}
We define
\begin{align}\notag
\alpha_1^N&:= \frac{\tilde{\mathbb{P}} [\tilde F^N]}{\tilde{\mathbb{P}}[F^N]},  \notag \\
\alpha^N_2&:= \tilde{\mathbb{P}}[F^N_1] \cdot \inf_{z; |z - \phi^x(\tilde \epsilon/\lambda_1)|<\tilde \epsilon/2} \tilde{\mathbb{P}}[F^{N,z}_2] < 1, \notag \\*
\alpha_3^N&:=	\tilde{\mathbb{P}}[F^N_1] \cdot\inf_{z; |z - \phi^x(\tilde \epsilon/\lambda_1)|<\tilde \epsilon/2}\big| \widetilde{\Cov}(\mathds{1}_{F^{N,z}_2},Y_j^{N,z,t_r-2\tilde \epsilon/\lambda_1,t_{r+1}-2 \tilde \epsilon/\lambda_1})\big|, \notag \\
\phi^r_j&:=\inf_{t \in [t_r,t_{r+1})} \log\beta_j(\phi^x(t)), \notag \\
 S^N&:= \sum_{j,\, \mu^i_j>0} \mu^i_j \sum_{r=0}^{M-1} (t_{r+1}-t_r)  \inf_{t \in [t_r,t_{r+1})}\log \beta_j(\phi^x(t)). \notag 
\end{align}
We compute (for $N$ large enough as before)
\begin{align} 
& \frac{\tilde{\mathbb{P}} [\tilde F^N]}{\tilde{\mathbb{P}}[F^N]} \sum_{j,\, \mu^i_j>0} \sum_{r=0}^{M-1}  \inf_{t \in [t_r,t_{r+1})} \log\beta_j(\phi^x(t)) \tilde{\mathbb{E}} \big[ Y^{N,x^N,t_r,t_{r+1}}_j \big] \notag \\*
& \qquad \geq \alpha_1^N \sum_{j,\, \mu^i_j>0} \sum_{r=0}^{M-1}  \inf_{t \in [t_r,t_{r+1})} \log\beta_j(\phi^x(t)) \cdot
\Big\{\mathds{1}_{\{\phi^r_j<0\}} \mu_j(t_{r+1}-t_r) \notag\\*
&\qquad \qquad +\mathds{1}_{\{\phi^r_j>0\}} \big(\alpha_2^N \mu_j(t_{r+1}-t_r) - \alpha_3^N\big) \Big\} \notag \\
& \qquad \geq \alpha_1^N \sum_{j,\, \mu^i_j>0} \mu_j^i\sum_{r=0}^{M-1} (t_{r+1}-t_r)  \inf_{t \in [t_r,t_{r+1})} \log\beta_j(\phi^x(t)) \Big\{\mathds{1}_{\{\phi^r_j <0\}} + \alpha_2^N \mathds{1}_{\{\phi^r_j >0\}} \Big\} \notag \\*
& \qquad \qquad - 2 \eta\alpha_1^N k  |\log \bar \beta|  \epsilon_N^{-1} \alpha_3^N \notag \\*
& \qquad \geq \alpha_1^N S^N - 2\eta\alpha_1^N k  |\log \bar \beta| \epsilon_N^{-1} \alpha_3^N
- \alpha_1^N k  |\log \bar \beta| \bar\mu \eta  (1-\alpha_2^N). \label{EqLemmaLPDLowerLinear.16}
\end{align}
We readily observe that
\begin{equation} \label{EqLemmaLPDLowerLinear.17a}
\alpha_1^N, \alpha_2^N\rightarrow 1 \quad \text{as } N\rightarrow \infty
\end{equation}
by Theorem~\ref{ThLLN} and Assumption~\ref{MainAss}~(C). We furthermore note that by Assumption~\ref{MainAss}~(C) and Theorem~\ref{ThLLN} (cf.~also the comment corresponding to~\eqref{AssSWRemoved} and again the fact that the rate of convergence in Theorem~\ref{ThLLN} is independent of initial values),
\begin{equation} \notag
\frac{\inf_z \sqrt{\tilde{\mathbb{P}} [F_2^{N,z}] -\tilde{\mathbb{P}} [F_2^{N,z}] ^2}}
{\sqrt{\epsilon_N}}
\leq \frac{\inf_z \sqrt{\tilde{\mathbb{P}} [(F_2^{N,z})^c] }}{\sqrt{\epsilon_N}}\rightarrow 0
\quad \text{as } N\rightarrow \infty.
\end{equation}
Therefore (for $N$ sufficiently large as before),
\begin{align} 
2\eta\alpha_1^N k  |\log \bar \beta| \epsilon_N^{-1} \alpha_3^N & \leq 2 \eta \alpha_1^N  k |\log \bar \beta| \tilde{\mathbb{P}}[F^N_1]\sqrt{\bar \mu} \frac{1}{\sqrt{\epsilon_N}}
\inf_{z; |z - \phi^x(\tilde \epsilon/\lambda_1)|<\tilde \epsilon/2} \sqrt{\tilde{\mathbb{P}} [F_2^{N,z}] -\tilde{\mathbb{P}} [F_2^{N,z}] ^2} \notag \\*
&\rightarrow 0 \label{EqLemmaLPDLowerLinear.17b}
\end{align}
as $N\rightarrow 0$. Finally, $S^N$ is a Riemann sum and we have
\begin{equation} \label{EqLemmaLPDLowerLinear.17c}
S^N\rightarrow \sum_{j} \mu^i_j \int_{2 \tilde \epsilon /\lambda_1}^\eta \log \beta_j(\phi^x(t))dt \quad \text{as } N \rightarrow \infty.
\end{equation}
 We observe that \eqref{EqLemmaLPDLowerLinear.13}~-~\eqref{EqLemmaLPDLowerLinear.17c} yield
\begin{align} 
& \liminf_{N\rightarrow\infty} \frac{1}{\tilde{\mathbb{P}}[F^N]}  \sum_{j,\, \mu^i_j>0} \tilde{\mathbb{E}} \Big[\frac{1}{N}\mathds{1}_{\tilde F^N} \sum_{\tau_j\in [2 \tilde \epsilon/\lambda_1,\eta]} \log \beta_{j}(Z^{N,x^N}(\tau_j-)) \Big] \notag \\*
&\qquad \geq \sum_{j=1}^k \mu^i_j \int_{2 \tilde \epsilon /\lambda_1}^\eta \log \beta_j(\phi^x(t))dt. \label{EqLemmaLPDLowerLinear.18}
\end{align}

We now consider the second term in the right ahnd side of \eqref{EqLemmaLPDLowerLinear.10}. We define
\[
\tilde M = \tilde M(N) := \lfloor 2 \tilde \epsilon \epsilon_N^{-1} +1 \rfloor
\]
and divide the interval $[0, 2 \tilde \epsilon /\lambda_1]$ into $\tilde M$ subintervals $[\tilde t_r,\tilde t_{r+1}]$ ($r=0,\dots,\tilde M-1$, $\tilde t_r=\tilde t_r(N)$) of length $\tilde \Delta = \tilde \Delta (N)$, i.e., for $N$ large enough,
\[
\frac{\epsilon_N}{2 \lambda_1}\leq \tilde \Delta < \frac{\epsilon_N}{\lambda_1}.
\]
For $r=0,\dots,\tilde M-1$ and $\tau_j \in [\tilde t_r,\tilde t_{r+1}]$, we obtain on $\tilde F^N$,
\begin{align}
\dist (Z^{N,x^N}(\tau_j-),\partial A) &> \dist(\phi^x(\tau_j-),\partial A) -2\epsilon_N\notag \\
&\geq \dist (\phi^x(\tilde t_{r}),\partial A) - 2\epsilon_N \notag \\
&\geq \lambda_1  \tilde t_r - 2\epsilon_N \notag \\ &\geq \frac{r-4}{2} \epsilon_N. \label{EqLemmaLPDLowerLinear.18b}
\end{align}
Hence, $\dist(Z^{N,x^N}(\tau_j-),\partial A)> \epsilon_N$ for $r \geq 6$. We compute for $j$ with $\mu_j>0$,
\begin{align}
&\frac{1}{\tilde{\mathbb{P}}[F^N]}
	\tilde{\mathbb{E}} \Big[\frac{1}{N}\mathds{1}_{\tilde F^N} 
	\sum_{\tau_j\in [0,2 \tilde\epsilon /\lambda_1]} \log \beta_{j}(Z^{N,x^N}(\tau_j-)) \Big]\notag \\*
&\qquad = \frac{1}{\tilde{\mathbb{P}}[F^N]} \sum_{r=0}^{\tilde M-1}
	\tilde{\mathbb{E}} \Big[\frac{1}{N}\mathds{1}_{\tilde F^N} 
	\sum_{\tau_j\in [\tilde t_r,\tilde t_{r+1}]} \log \beta_{j}(Z^{N,x^N}(\tau_j-)) \Big] 
	\notag \\
&\qquad = \frac{1}{\tilde{\mathbb{P}}[F^N]} \sum_{r=6}^{\tilde M-1}
	\tilde{\mathbb{E}} \Big[\frac{1}{N}\mathds{1}_{\tilde F^N} 
	\sum_{\tau_j\in [\tilde t_r,\tilde t_{r+1}]} \log \beta_{j}(Z^{N,x^N}(\tau_j-)) \Big]\notag \\
&\qquad \qquad + \frac{1}{\tilde{\mathbb{P}}[F^N]} \sum_{r=0}^{5}
	\tilde{\mathbb{E}} \Big[\frac{1}{N}\mathds{1}_{\tilde F^N} 
	\sum_{\tau_j\in [\tilde t_r,\tilde t_{r+1}]} \log \beta_{j}(Z^{N,x^N}(\tau_j-)) \Big]. \label{EqLemmaLPDLowerLinear.19a}
\end{align}
We note that for all $j$, $\beta_j(Z^{N,x^N}(\tau_j-))\geq \underline \beta(\lambda_0/N)$ $ \tilde{\mathbb{P}}$-a.s.~by Assumption~\ref{MainAss}~(A1).
The second term in the right hand side of \eqref{EqLemmaLPDLowerLinear.19a} can be bounded from below (w.l.o.g. $\underline{\beta}(\lambda_0/N)<1$):
\begin{align}
&\frac{1}{\tilde{\mathbb{P}}[F^N]} \sum_{r=0}^{5}
	\tilde{\mathbb{E}} \Big[\frac{1}{N}\mathds{1}_{\tilde F^N} 
	\sum_{\tau_j\in [\tilde t_r,\tilde t_{r+1}]} \log \beta_{j}(Z^{N,x^N}(\tau_j-)) \Big] \notag \\*
	&\qquad \geq  \frac{4}{\tilde{\mathbb{P}}[F^N]} 
	\log \underline\beta\Big(\frac{\lambda_0}{N} \Big)\sum_{r=0}^5
	\Big\{
	\tilde{\mathbb{P}} [\tilde F^N] \tilde{\mathbb{E}} [ Y_j^{N,x^N,\tilde t_r,\tilde t_{r+1}}] 
	+ \widetilde{\Cov} (\mathds{1}_{\tilde F^N} ,Y_j^{N,z,\tilde t_r,\tilde t_{r+1}}) \Big\}
	\notag \\*
		&\qquad \geq  \frac{6 \tilde{\mathbb{P}} [\tilde F^N]}{\tilde{\mathbb{P}}[F^N]} \log \underline\beta\Big(\frac{\lambda_0}{N}\Big)
	 \bar \mu \frac{\epsilon_N}{\lambda_1}  +  \frac{4}{\tilde{\mathbb{P}}[F^N]} \log \underline\beta\Big(\frac{\lambda_0}{N} \Big)
	 \sqrt{\frac{\bar \mu \; \epsilon_N}{\lambda_1}}  \sqrt{\delta(N,\epsilon_N)} \notag\\
	 &\qquad \rightarrow 0 \label{EqLemmaLPDLowerLinear.19b}
\end{align}
as $N\rightarrow \infty$ by Assumption~\ref{MainAss}~(C) (cf.~also~\eqref{EqAssC3.3}). For the first term in Equation~\eqref{EqLemmaLPDLowerLinear.19a}, we compute for $j$ with $\mu_j>0$ (similarly as before, we assume w.l.o.g.~that $\underline \beta(\tilde \epsilon)<1$ and note that
$\beta_{j}(Z^{N,x^N}(\tau_j-))\geq \underline \beta(\lambda_1 \tilde t_r -\epsilon_N)\geq \underline\beta((r-4)\lambda_1 \tilde \Delta/2)$)
\begin{align}
&\frac{1}{\tilde{\mathbb{P}}[F^N]} \sum_{r=6}^{\tilde M-1}
	\tilde{\mathbb{E}} \Big[\frac{1}{N}\mathds{1}_{\tilde F^N} 
	\sum_{\tau_j\in [\tilde t_r,\tilde t_{r+1}]} \log \beta_{j}(Z^{N,x^N}(\tau_j-)) \Big] \notag \\*
	&\quad \geq \frac{1}{\tilde{\mathbb{P}}[F^N]} \sum_{r=2}^{\tilde M-5}
	\log \underline\beta(r\lambda_1 \tilde \Delta/2) \tilde{\mathbb{E}}  [\mathds{1}_{\tilde F^N}  
	Y^{N,x^N,\tilde t_r,\tilde t_{r+1}} ] \notag \\
	&\quad = \frac{\tilde{\mathbb{P}}[\tilde F^N]}{\tilde{\mathbb{P}}[F^N]} \mu_j^i\sum_{r=2}^{\tilde M-5} \tilde \Delta
	\log \underline\beta(r\lambda_1 \tilde \Delta/2) 	
	+ \frac{1}{\tilde{\mathbb{P}}[F^N]} \sum_{r=2}^{\tilde M-5}
	\log \underline\beta(r\lambda_1 \tilde \Delta/2) \widetilde{\Cov} (\mathds{1}_{\tilde F^N}, 
	Y^{N,x^N,\tilde t_r,\tilde t_{r+1}} ). \label{EqLemmaLPDLowerLinear.19c}
\end{align}
For the first term in Equation~\eqref{EqLemmaLPDLowerLinear.19c}, we have by Assumption~\ref{MainAss}~(C) (in particular by the fact that the integral below converges, cf.~\eqref{EqAssC3.2})
\begin{equation}
\frac{\tilde{\mathbb{P}}[\tilde F^N]}{\tilde{\mathbb{P}}[F^N]} \mu_j^i\sum_{r=2}^{\tilde M-5} \tilde \Delta
	\log \underline\beta(r\lambda_1 \tilde \Delta/2) 	
	\rightarrow \mu_j^i\int_{0}^{2 \tilde \epsilon/\lambda_1}  \log \underline\beta(\lambda_1\rho/2) d\rho 
	\label{EqLemmaLPDLowerLinear.19d}
\end{equation}
as $N \rightarrow \infty$. Similarly, we obtain for the second term in Equation~\eqref{EqLemmaLPDLowerLinear.19c}, 
\begin{align}
&\frac{1}{\tilde{\mathbb{P}}[F^N]} \sum_{r=2}^{\tilde M-5}
	\log \underline\beta(r\lambda_1 \tilde \Delta/2) \widetilde{\Cov} (\mathds{1}_{\tilde F^N}, 
	Y^{N,x^N,\tilde t_r,\tilde t_{r+1}} ) \notag \\*
&\qquad \geq  \frac{1}{\tilde{\mathbb{P}}[F^N]} \sqrt{\mu_j^i\frac{\delta(N,\epsilon_N)}{\tilde \Delta}} \sum_{r=2}^{\tilde M-5} \tilde \Delta
	\log \underline\beta(r\lambda_1 \tilde \Delta/2) \notag \\*
&\qquad \rightarrow 0
\end{align}
as $N\rightarrow 0$ by Assumption~\ref{MainAss}~(C) (cf.~\eqref{EqAssC3.2a} and~\eqref{EqAssC3.2}).
	
Finally, we consider the third term in Equation~\eqref{EqLemmaLPDLowerLinear.10}. We obtain by Assumption~\eqref{MainAss}~(C),
\begin{align} 
&\frac{1}{\tilde{\mathbb{P}}[F^N]} \sum_{j,\, \mu^i_j>0} \tilde{\mathbb{E}} \Big[\frac{1}{N}\mathds{1}_{F^N\setminus \tilde F^N} \sum_{\tau_j\in [0,\eta]} \log \beta_{j}(Z^{N,z}(\tau_j-)) \Big] \notag \\*
&\geq \frac{1}{\tilde{\mathbb{P}}[F^N]} 
\log \underline \beta\Big(\frac{\lambda_0}{N}\Big)  
\sum_{j; \, \mu^i_j>0}\big \{ \tilde{\mathbb{P}} [(\tilde F^N)^c]
\tilde{\mathbb{E}}[Y_j^{N,z,0,\eta}] +  \widetilde{\Cov} (\mathds{1}_{(\tilde F^N)^c} ,Y_j^{N,z,0,\eta} )\big\}\notag\\ 
&\geq \frac{1}{\tilde{\mathbb{P}}[F^N]} \log \underline \beta\Big(\frac{\lambda_0}{N} \Big)  k\, \bar \mu\, \eta\, \delta(N,\epsilon_N)
	+\frac{1}{\tilde{\mathbb{P}}[F^N]} \log \underline \beta\Big(\frac{\lambda_0}{N} \Big)  k \sqrt{\bar \mu\, \eta\, \delta(N,\epsilon_N)}\notag\\* 
&\rightarrow 0 \label{EqLemmaLPDLowerLinear.20}
\end{align}
as $N\rightarrow \infty$ similarly as before (cf.~\eqref{EqAssC3.2a} and~\eqref{EqAssC3.3}). 

We obtain by Equation~\eqref{EqLemmaLPDLowerLinear.10} and~\eqref{EqLemmaLPDLowerLinear.18}, \eqref{EqLemmaLPDLowerLinear.19a}~-~\eqref{EqLemmaLPDLowerLinear.20},
\begin{align}
&\liminf_{N\rightarrow\infty} \frac{1}{\tilde{\mathbb{P}}[F^N]} 
	\tilde{\mathbb{E}} \Big[\frac{1}{N}\mathds{1}_{F^N} \sum_{\tau \leq \eta} \log \beta_{j(\tau)}(Z^{N,x^N}(\tau-)) \Big] \notag \\*
& \qquad \geq\sum_{j=1}^k \mu^i_j \int_{2 \tilde \epsilon /\lambda_1}^\eta |\log \beta_j(\phi^x(t))|dt - k \bar \mu \int^{2 \tilde \epsilon /\lambda_1}_0 |\log \underline \beta(\lambda_1 \rho/2)| d\rho. 
	\label{EqLemmaLPDLowerLinear.21}
\end{align}

We conclude by letting $\delta>0$ and choosing $\tilde \epsilon<\epsilon_1$ small enough such that (cf.~Equations~\eqref{EqLemmaLPDLowerLinear.2}, \eqref{EqLemmaLPDLowerLinear.3} and Inequality~\eqref{EqLemmaLPDLowerLinear.21}; note that we require the convergence of the integral in~\eqref{EqAssC3.2} of Assumption~\ref{MainAss}~(C) here)
\begin{equation} \notag
\frac{\tilde \epsilon}{\lambda_1} k \bar \mu,\, k \bar\mu \int_{0}^{2 \tilde \epsilon /\lambda_1} |\log \beta_j(\phi^x(t))|dt, k \bar \mu \int^{2 \tilde \epsilon /\lambda_1}_0 |\log \underline \beta(\lambda_1 \rho/2)| d\rho < \frac{\delta}{4}.
\end{equation}
The assertion now follows from Inequality~\eqref{EqLemmaLPDLowerLinear.1} and~\eqref{EqLemmaLPDLowerLinear.2}, \eqref{EqLemmaLPDLowerLinear.3}, \eqref{EqLemmaLPDLowerLinear.4}, \eqref{EqLemmaLPDLowerLinear.9} and~\eqref{EqLemmaLPDLowerLinear.21}:
\begin{align} \notag
\liminf_{N\rightarrow \infty}    \tilde{\mathbb{E}}_{F^{N}}\Big[\frac{X_\eta}{N} \Big]
\geq - \int_0^\eta \ell(\phi^x(t),\mu)dt - \delta.
\end{align}
The uniformity of the convergence follows from the fact that we have used only Assumption~\ref{MainAss}~(C) and Theorem \ref{ThLLN}, where the convergences 
are uniform in $x$.
\end{proof}

Again in the following result, the exponent $\alpha$ is the one from Assumption~\ref{MainAss}~(C).
\begin{theorem} \label{TheoremLDPLowerMod} 
Assume that Assumption~\ref{MainAss} holds. Let $x\in A$
and $x^N\in A^N$ such that 
\[
\limsup_{N\rightarrow \infty}|x^N-x| N^\alpha <1.
\]
Then, for $\phi \in D([0,T];A)$ and $\epsilon>0$,
\[
\liminf_{N\rightarrow \infty} \frac 1 N \log \mathbb{P} \Big[ \sup_{t \in [0,T]} |Z^{N,x^N}(t)-\phi(t)|<\epsilon\Big] \geq - I_{T,x}(\phi).
\]
Moreover the above convergence is uniform in $x\in A$
\end{theorem}

\begin{proof}
We can w.l.o.g. assume that $I_{T,x}(\phi)\leq K<\infty$. Let $\delta>0$ and divide the interval $[0,T]$ into $J$ subintervals as before. We define the function $\phi^\eta$ as before and choose $\eta_1$ small enough such that for all $\eta<\eta_1$ (cf.~Lemma~\ref{Lemma5.1}),
\begin{equation}\label{EqLowerMain.0}
  \int_{\eta}^{T} L(\phi^\eta(t),(\phi^\eta)'(t))dt <\int_{0}^{T} L(\phi(t),\phi'(t))dt+\frac{\delta}{3}.
\end{equation}

We furthermore assume that $\eta<\eta_1$ is such that
\[
\sup_{t\in [0,T]} |\phi(t) - \phi^\eta(t) | < \frac \epsilon 4.
\]
Hence,
\[
\liminf_{N \rightarrow \infty} \frac 1 N \log \mathbb{P} \Big[ \sup_{t\in[0,T]}|Z^{N,x^N}(t) - \phi(t)| < \epsilon\Big] \geq \liminf_{N \rightarrow \infty} \frac 1 N \log \mathbb{P} \Big[ \sup_{t\in[0,T]}|Z^{N,x^N}(t) - \tilde\phi^{\eta}(t)| <\frac \epsilon 2\Big].
\]

From \eqref{EqPhiEta1},  for $t\geq \eta$,
$
\dist(\phi^\eta(t),\partial A)\geq \eta \lambda_1$.
We define
\[
\epsilon_1=\epsilon_1(\eta)= \frac \epsilon 2 \wedge \frac{\lambda_1 \eta}{4},
\]
\[
\underline \beta^\eta:=\inf\Big\{
\beta_j(z)\Big|\ 1\le j\le k,\  z \in A, \dist(z,\partial A) \geq \frac{\eta\lambda_1}{2} \Big\}>0
\]
and
\[
\tilde\beta^\eta_j(z):=
\begin{cases}
\beta_j(z)\vee \underline \beta^\eta & \text{ if } z \in A\\
\tilde \beta_j^\eta(\psi_A(z)) & \text{ else,}
\end{cases}
\]
where the function $\psi_A$ has been specified in Assumption (A4).
We denote by $\tilde Z^{N,z,\eta}$ the process starting at $z$ at time $\eta$ with rates $\tilde \beta_j^\eta$. As the $\log\tilde \beta_j^\eta$ are bounded, we have by Theorem~\ref{Theorem5.51a} that there exists an
\[
\epsilon_2=\epsilon_2(\eta)<\epsilon_1(\eta)
\]
such that for all $\tilde \epsilon < \epsilon_2$,
\[
\liminf_{N \rightarrow \infty} \frac 1 N \log\Big( \inf_{|z-\phi^\eta(\eta)|< \tilde \epsilon} \mathbb{P} \Big[ \sup_{t\in[\eta,T]}|\tilde Z^{N,z,\eta}(t) - \tilde\phi^{\eta}(t)| <  \epsilon_1\Big]\Big)\geq - \int_\eta^T \tilde L^\eta(\phi^\eta(t),(\phi^\eta)'(t)) dt - \frac \delta 3,
\]
where $\tilde L^\eta$ denotes the Legendre transform corresponding to the rates $\tilde \beta_j^\eta$. We readily observe that for all $t \in [\eta,T]$,
\[
\tilde L^\eta(\phi^\eta(t),\phi^\eta(t)) = L(\phi^\eta(t),(\phi^\eta)'(t))
\]
and that for $|z-\phi^\eta(\eta)|<\tilde \epsilon$, denoting by an abuse of notation $Z^{N,z}$ the process starting from $z$ at time $\eta$, 
\[
\sup_{t\in [\eta,T]} | Z^{N,z}(t) -\phi^\eta(t)|<\epsilon_1 \Leftrightarrow \sup_{t\in [\eta,T]} | \tilde Z^{N,z,\eta}(t) -\phi^\eta(t)|<\epsilon_1.
\]
and hence
\begin{align*}
\mathbb{P} \Big[ \sup_{t\in[\eta,T]}|Z^{N,z}(t) - \tilde\phi^{\eta}(t)| <\frac \epsilon 2\Big]&
\geq \mathbb{P} \Big[ \sup_{t\in[\eta,T]}|Z^{N,z}(t) - \tilde\phi^{\eta}(t)| <\epsilon_1 \Big] \\
& =\mathbb{P} \Big[ \sup_{t\in[\eta,T]}|\tilde Z^{N,z,\eta}(t) - \tilde\phi^{\eta}(t)| < \epsilon_1\Big]
\end{align*}
consequently for $\tilde \epsilon < \epsilon_2$
\begin{align*}% \label{EqLowerMain.1}
\liminf_{N \rightarrow \infty} \frac 1 N \log\Big( \inf_{|z-\phi^\eta(\eta)|< \tilde \epsilon} \mathbb{P} \Big[ \sup_{t\in[\eta,T]}| Z^{N,z}(t) - \tilde\phi^{\eta}(t)| <  \epsilon_1\Big]\Big)&\geq - \int_\eta^T  L(\phi^\eta(t),(\phi^\eta)'(t)) dt - \frac \delta 3\\
&\ge - \int_0^T  L( \phi(t), \phi'(t)) dt - \frac{2\delta}{3},
\end{align*}
where we have used \eqref{EqLowerMain.0} for the second inequality.
We use the Markov property of $Z^N$ and obtain for $\tilde \epsilon < \epsilon_2$
\begin{align*}
\mathbb{P} \Big[ \sup_{t\in[0,T]}|Z^{N,x^N}(t) - \phi(t)| < \epsilon\Big] 
& \geq \mathbb{P} \Big[ \sup_{t\in[0,\eta]}|Z^{N,x^N}(t) - \tilde\phi^{\eta}(t)| <\tilde \epsilon\Big] \notag \\*
& \qquad \cdot \inf_{|z-\phi^\eta(\eta)|<\tilde \epsilon} \mathbb{P} \Big[ \sup_{t\in[\eta,T]}|Z^{N,z}(t) - \tilde\phi^{\eta}(t)| <\epsilon_1\Big]. 
%\label{EqLowerMain.2}
\end{align*}
Combining the last two inequalities with Lemma \ref{LemmaLPDLowerLinear}, we deduce that ($i$ being the index of the ball $B_i$
to which the starting point $x$ belongs) 
\begin{align*}
\liminf_{N\to\infty}\frac{1}{N}\log\mathbb{P} \Big[ \sup_{t \in [0,T]} |Z^{N,x^N}(t)-\phi(t)|<\epsilon\Big] &\geq
-\int_0^\eta\ell(\phi^x(t),\mu^i)dt-\int_0^T  L( \phi(t), \phi'(t)) dt - \frac{2\delta}{3}\\
&\ge -\int_0^T  L( \phi(t), \phi'(t)) dt-\delta
\end{align*}
thanks to Lemma \ref{le:new} below, provided $\eta$ is small enough. The result follows since $\delta>0$ is arbitrary.
\end{proof}

\begin{lemma}\label{le:new}
Let $x\in B_i$, where $i\le I_1$, and suppose $\phi^x(t)=x+t v_i$. Let moreover $\mu^i$ be such that 
$\sum_{j=1}^k\mu^i_jh_j=v_i$. Then, uniformly in $x$, as $t\to0$,
$$\int_0^t \ell(\phi^x(s),\mu^i)ds\to0.$$
\end{lemma}
\begin{proof}
Since according to Assumption (A3) $d(\phi^x(t),\partial A)\ge \lambda_1 t$, the result follows from 
\eqref{EqAssC3.3} from Assumption (C).
\end{proof}
\begin{theorem} \label{TheoremLDPLower}
Assume that Assumption~\ref{MainAss} as well as the assumptions from Theorem \ref{TheoremLDPLowerMod} hold. Then for 
any open set $G\subset D([0,T];A)$,
\[
\liminf_{N\rightarrow \infty} \frac 1 N \log \mathbb{P} \big[ Z^{N,x^N} \in G\big] \geq - \inf_{\phi \in G} I_{T,x}(\phi).
\]
Moreover the convergence is uniform in $x$.
\end{theorem}

\begin{proof}
The proof follows the same line of reasoning as the proof of Theorem~\ref{Theorem5.51b}.
\end{proof}
We will need the following stronger version. Recall the definition of $A^N$ at the start of section \ref{SectionSetup}. 
\begin{theorem} \label{TheoremLDPLowerUnif}
Assume that Assumption~\ref{MainAss} 
%as well as the assumptions from Theorem \ref{TheoremLDPLowerMod} 
holds. Then for 
any open set $G\subset D([0,T];A)$ and any compact subset $K\subset A$,
\[
\liminf_{N\rightarrow \infty} \frac 1 N \log \inf_{x\in K\cap A^N}\mathbb{P} \big[ Z^{N,x} \in G\big] \geq -\sup_{x\in K} \inf_{\phi \in G} I_{T,x}(\phi).
\]
\end{theorem}
\begin{proof}
This follows readily from the uniformity in $x$ of the convergence in  Theorem \ref{TheoremLDPLower}.  
\end{proof}
\section{LDP upper bound} \label{SecUpper}

We now prove the LDP upper bound. For reasons of readability, we split up the proof into four parts. In the first three parts, we prove the main auxiliary results required (Sections~\ref{SubSecUpper1}~-~\ref{SubSecUpper3}). Finally, we prove the main results of the section  in Section~\ref{SubSecUpperMain}.

In this section, whenever we consider the process $Z^{N,x}$, we will mean that the process $Z^N$ is started from the nearest point to $x$ on the grid $A^N$
(see the beginning of section \ref{SectionSetup} for the definition of $A^N$). 
%{\color{red} Should unify presentation of sections 5 and 6}

\subsection{Piecewise linear approximation} \label{SubSecUpper1}

The goal of this section is to prove that $Z^{N,x}$ is exponentially close to its piecewise linear approximation. 
For $Z^{N,x}$, we define the piecewise linear interpolation $Y^{N,x}$. To this end, we divide $[0,T]$ into $N$ subintervals $[t_{j-1},t_{j}]$ with $t_j=\frac{j T}{N}$, $j=1,\dots,N$. We define $t \in [t_{j-1},t_{j}]$
\begin{equation}\label{defYN}
Y^{N,x}_t= Z^{N,x}(t_{j-1}) + \frac{t-t_{j-1}}{t_j - t_{j-1}} (Z^{N,x}(t_j)- Z^{N,x}(t_{j-1})).
\end{equation}

We prove that $Y^{N,x}$ is exponentially close to $Z^{N,x}$.

\begin{lem} \label{Lemma5.57}
Assume that $\beta_j$ ($j=1,\dots,k$) is bounded. Let $\delta >0$. Then
\[
\limsup_{N \rightarrow \infty} \frac{1}{N} \log \mathbb{P} [ \di(Y^{N,x},Z^{N,x})>\delta]=-\infty
\]
uniformly in $x \in A$.
\end{lem}
\begin{proof}
For any $1\le j\le [N/T]$, we have the inclusion
$$\{\sup_{t\in[t_{j-1},t_j]}|Y^{N,x}_t-Z^{N,x}_t|\ge\delta\}\subset\{\sup_{t\in[t_{j-1},t_j]}|Z^{N,x}_t-Z^{N,x}_{t_{j-1}}|\ge\delta/2\}.$$
It then follows from Lemma \ref{Cor5.55} that for some positive constant $C$ and for each $j$,
$$\mathbb{P}(\sup_{t\in[t_{j-1},t_j]}|Y^{N,x}_t-Z^{N,x}_t|\ge\delta)\le\exp\left(-CN\delta\log(CN\delta)\right)).$$
Consequently
\begin{align*}
\mathbb{P}\left(\sup_{t\in[0,T]}|Y^{N,x}_t-Z^{N,x}_t|\ge\delta\right)&=\mathbb{P}\left(\bigcup_{j=1}^{[N/T]}\left\{\sup_{t\in[t_{j-1},t_j]}|Y^{N,x}_t-Z^{N,x}_t|\ge\delta\right\}\right)\\
&\le N\exp\left(-CN\delta\log(CN\delta)\right)).
\end{align*}
The result clearly follows.
\end{proof}

\subsection{The modified rate function $I^\delta$} \label{SubSecUpper2}

In this section, we define a modified rate function $I^\delta$ and analyse how it relates to $I$. The main result is Corollary~\ref{Cor4.2} below.

We define the following functional (Lemma \ref{Lemma5.40} above). For $\delta >0$, $x\in A$, $y, \theta \in \mathds{R}^d$, let
\begin{align}
\tilde\ell_\delta(x,y,\theta)&:= \langle \theta,y \rangle - \sum_{j=1}^k \sup_{z=z^j \in A; |z-x|<\delta}\beta_j(z)\big(\exp(\langle\theta,h_j \rangle) -1 \big),\notag \\
L_\delta(x,y)&:=\sup_{\theta \in \mathds{R}^d}\tilde\ell_\delta(x,y,\theta). \notag
\end{align}
Obviously, we have
\[
L_\delta(x,y)
\leq L(x,y)
\]
and for the respectively defined functional, $I^\delta$,
\[
I^\delta
\leq I.
\]
We obtain 
\begin{align}
L_\delta(x,y)&=\sup_{\theta \in \mathds{R}^d} \Big\{ \langle \theta,y \rangle - \sum_{j=1}^k \sup_{z^j \in A; |z^j-x|<\delta}\beta_j(z^j)\big(\exp(\langle\theta,h_j \rangle) -1 \big) \Big\} \notag \\
&=\sup_{\theta \in \mathds{R}^d} \inf_{z^1,\dots,z^k \in A; |z^j-x|<\delta}\Big\{ \langle \theta,y \rangle - \sum_{j=1}^k \beta_j(z^j)\big(\exp(\langle\theta,h_j \rangle) -1 \big) \Big\}\notag\\
&=\inf_{z^1,\dots,z^k \in A; |z-x|<\delta} \sup_{\theta \in \mathds{R}^d} \Big\{ \langle \theta,y \rangle - \sum_{j=1}^k \beta_j(z^j)\big(\exp(\langle\theta,h_j \rangle) -1 \big) \Big\} \label{EqMinMax}\\
&=\inf_{z^1,\dots,z^k, |z^j-x|<\delta} \inf_{\mu \in \tilde V_{z_j,y}} \sum_{j=1}^k\big( \beta_j(z^j) - \mu_j + \mu_j \log \mu_j - \mu_j \log \beta_j(z^j)\big) \label{EqL=TildeL} \\
%&=\inf_{z^1,\dots,z^k, |z^j-x|<\delta} \inf_{\mu \in \tilde V_{y}} \sum_{j=1}^k\big( \beta_j(z^j) - \mu_j + \mu_j \log \mu_j - \mu_j \log \beta_j(z^j)\big)\footnotemark \notag\\
&= \ell(z^*,\mu^*),\label{ell}
\end{align}
where we use the slight abuse of notation: for $z=(z^1,\dots,z^k)$,
\[
\ell(z,\mu)=\sum_j \beta_j(z^j) - \mu_j + \mu_j \log \Big(\frac{\mu_j}{\beta_j(z^j)}\Big)
\]
Note that  $\ell(x,\mu)$ depends on $x$ only through the rates $\beta(x)$.

Here, Equation~\eqref{EqMinMax} follows from Sion's min-max  theorem, see e.g. \cite{Komiya},
and Equation~\eqref{EqL=TildeL} follows by Theorem~\ref{Theorem5.26}. Equation~\eqref{ell} follows from Lemma~\ref{Lemma3.3} and the continuity of $\ell$ and $\mu^*$ (as a function in the state). We remark that $|z^*_j-x|=\delta$ is possible.

In a similar fashion as before (cf.~Proposition~\ref{Prop5.46}), we define the sets
\begin{align*}
%\Phi (K) &:= \{ \phi | I (\phi) \leq K \}, \\
\Phi^\delta (K)&:= \{ \phi \in D([0,T];A)| I^\delta(\phi) \leq K\}, \\
%\Phi_x (K)&:= \{ \phi | I_x(\phi) \leq K\}, \\
\Phi^\delta_x (K)&:= \{ \phi \in D([0,T];A)|  I_x^\delta (\phi) \leq K\}. 
\end{align*}
In particular, we have $\Phi(K) \subset \Phi^\delta(K)$ and $\Phi_x(K) \subset \Phi_x^\delta(K)$ and $\Phi_x^\delta(K)$, $\Phi^\delta(K)$ are increasing in $\delta$.

For technical reasons, we define for $m>0$, $z\in A$ the rates
\[
\beta_j^m(z):=\max\{ \beta_j(z), 1/m \}
\]
and the corresponding functionals $L^m$ and $I^m$ by replacing the rates $\beta_j$ by the rates $\beta_j^m$ in the respective definitions. 

We will need the following slightly stronger version of Lemma \ref{Lemma5.1}, where again
 $\phi^\eta$ is defined from $\phi$ as in the lines before Lemma \ref{Lemma5.1}. 
\begin{lemma} \label{Lem4.2}
Assume that Assumption~\ref{MainAss} holds. Let $K>0$ and $\epsilon>0$. Then there exists an $\eta_0=\eta_0(T,K,\epsilon)>0$ such that for all $\eta<\eta_0$ there exists an $m_0>0$ such that for all $m>m_0$ and for all $\phi \in D([0,T];A)$ with $I_T^m(\phi)\leq K$,
\[ 
I_T(\phi^\eta)<K+\epsilon, 
\]
where $\phi^\eta$ is defined before Lemma~\ref{Lemma5.1} and satisfies $\|\phi^\eta-\phi\|\le\epsilon$.
\end{lemma}

\begin{proof}
We follows the first steps of the proof of Lemma {Lemma5.1}, where we replace $\mu^\ast(t)$ by 
 $\mu^{m,*}(t)$ the optimal $\mu$ corresponding to $(\phi(t),\phi'(t))$ and jump rates $\beta^m_j$. Now \eqref{EqLemma5.1.1}
 is replaced by 
 \begin{equation} \label{EqApp1}
L(\phi^\eta(t),\phi'(t-\eta_r)) \leq  \ell(\phi^\eta(t),\mu^{m,*}(t-\eta_r)).
\end{equation}
We now choose $m>1/\eta$ and deduce
\begin{equation} \label{EqApp2}
|\beta_j(\phi^\eta(t))-\beta_j^m(\phi(t-\eta_r))| \leq\frac 1 m +|\beta_j(\phi^\eta (t))-\beta_j(\phi(t-\eta_r))|
\leq  \delta'_K(\eta),
\end{equation}
where $\delta'_K(\eta)=\eta+\delta_K(\eta) \rightarrow 0$ as $\eta \rightarrow 0$ by the (uniform) continuity of the $\beta_j$. We deduce from
\eqref{EqApp1} the following modified version of \eqref{EqLemma5.1.3}
\begin{align}\label{estim_m}
L(\phi^\eta(t),\phi'(t-\eta_r)) - L^m(\phi(t-\eta_r),\phi'(t-\eta_r)) 
 \leq
k\delta'_K(\eta)+\sum_j  \mu^{m,*}(t-\eta_r) \log \frac{\beta^m_j(\phi(t-\eta_r))}{\beta_j(\phi^\eta(t))},
\end{align}
since $\beta^m_j(\phi(t))>0$ and $\beta_j(\phi^\eta(t))>0$ for $t\not=0$.

We recall that $\lambda_1$ and the $v_i$'s have been defined in Assumption~\ref{MainAss}~(A3), and that the $\tilde v_i$'s, $\bar v_i$'s and $\hat v_i$'s  
have been defined in the proof of Lemma \ref{Lemma5.1}.

By Assumption~\ref{MainAss}~(B4), there exists a constant $\lambda_4>0$ such that for $z\in B_{i_r}$ (and $\eta<\eta_2\leq\eta_1$ small enough,
depending upon $\lambda_1$ and $\lambda_2$ but not on $\phi$, except  through $K$),
\begin{equation} \label{EqTilde}
\beta_j(z) <\lambda_4\Rightarrow \beta_j(z+\eta \tilde v_{i_r})\geq \beta_j(z).
\end{equation}

We now want to bound from above the second term in the right hand side of \eqref{estim_m}. If $\beta_j(\phi(t-\eta_r))\geq  \lambda_4$, then
$\beta_j(\phi(t-\eta_r)) \geq  1/m$ and therefore $\beta_j^m(\phi(t-\eta_r))=\beta_j(\phi(t-\eta_r))$, so that the bound \eqref{EqLemma5.1.5} holds.

Now consider the case $\beta_j(\phi(t-\eta_r))<  \lambda_4$.
We define the function $s(\delta):=\inf\{\beta_j(x)|\di(x, \partial A) \geq \delta \}$; hence (recall the continuity of the $\beta_j$ and the compactness of $A$)
$s(\delta)>0 \text{ for }\delta >0$,  and for $x \in A$,  $\beta_j(x) \geq s(\di(x,\partial A))$. 

We furthermore let
\[
m_0=m_0(\eta,\lambda_4)>\max\{1/\lambda_4, 1/s(\lambda_1 \eta)\};
\]
and recall that $\di(\phi^\eta(t),\partial A) \geq \lambda_1 \eta$ for $t\geq \eta$ (cf.~the discussion preceding Lemma~\ref{Lemma5.1}).

We let $m>m_0$.
Since $\beta_j(\phi(t-\eta_r))<  \lambda_4$, by~\eqref{EqTilde},
\[
\beta_j(\phi^\eta(t)) \geq \beta_j(\phi(t-\eta_r)).
\] 
By the definition of $s$, we have furthermore 
\[
\beta_j(\phi^\eta(t)))\geq s(\lambda_1 \eta) \geq 1/m.
\] 
Combining these observations, we obtain
\[
\beta_j(\phi^\eta(t)) \geq \max\{\beta_j(\phi(t-\eta_r)), 1/m \} = \beta_j^m(\phi(t-\eta_r))
\]
and therefore
\begin{equation} \label{EqApp5}
\log \frac{\beta_j^m(\phi(t-\eta_r))}{\beta_j(\phi^\eta(t))} \leq 0. 
\end{equation}

From Lemma~\ref{Lemma5.17} and Lemma~\ref{Lemma3.3}, there exist (universal, i.e., independent of $x$ and $m$) constants 
$B\geq B_1 \vee B_2$, $B>1$, $C_1$, $C_3$ such that
\begin{equation} \label{EqApp7}
|y| > B \Rightarrow \forall x \in A,m,\,\, L(x,y), L^m(x,y) \geq C_1 |y| \log |y|,
\end{equation}
\begin{equation} \label{EqApp8}
|y| > B \Rightarrow \forall x \in A,m, \,\, |\mu^{m,*}|=|\mu^*(x,y,m)|\leq C_3 |y|.
\end{equation}
Hence if $|\phi'(t-\eta_r)|\geq  B$, we get, instead of \eqref{EqLemma5.1.9},
\begin{align}
&L(\phi^\eta(t),\phi'(t-\eta_r)) - L^m(\phi(t-\eta_r),\phi'(t-\eta_r)) \notag \\*
&\qquad \leq
	k \delta'_K(\eta)+ k C_3|\phi'(t-\eta_r)| \frac{2 \delta_K(\eta)}{\lambda_4} \notag \\
&\qquad \leq  k \delta'_K(\eta)+ k C_3 \frac{2 \delta_K(\eta)L^m(\phi(t-\eta_r),\phi'(t-\eta_r))}{C_1\lambda_4 \log|\phi'(t-\eta_r)|}.  \label{EqApp9}
\end{align}
%\footnotetext{Here we use~\eqref{EqApp7}: 
%\[
%|\phi'(t-\eta_r)|\leq \frac{L^m(\phi(t-\eta_r),\phi'(t-\eta_r))}{C_1 \log|\phi'(t-\eta_r)|}.
%\]
%}
If \underline{$|\phi'(t-\eta_r)|<B$}, Lemma~\ref{Lemma3.3} implies that
$|\mu^{m,*}(t-\eta_r)| \leq \tilde C  B$ for a universal constant $\tilde C>0$. Using Equations~\eqref{EqApp2} and~\eqref{estim_m}, we obtain
\begin{equation}
L(\phi^\eta(t),\phi'(t-\eta_r)) - L^m(\phi(t-\eta_r),\phi'(t-\eta_r)) \leq
	k \delta'_K(\eta)+ k \tilde C B  \frac{2 \delta_K(\eta)}{\lambda_4}.  \label{EqApp10}
\end{equation}
Inequalities~\eqref{EqApp9} and~\eqref{EqApp10} imply
\begin{align}
&L(\phi^\eta(t),\phi'(t-\eta_r)) - L^m(\phi(t-\eta_r),\phi'(t-\eta_r)) \notag \\*
&\qquad \leq
	k(\eta+\delta_K(\eta))+ k \tilde C  B  \frac{2 \delta_K(\eta)}{\lambda_4}
	+k C_3 \frac{2 \delta_K(\eta)L^m(\phi(t-\eta_r),\phi'(t-\eta_r))}{C_1\lambda_4 \log|\phi'(t-\eta_r)|}\notag \\
	&\qquad =:\delta_1(\eta) + \delta_2(\eta)L^m(\phi(t-\eta_r),\phi'(t-\eta_r)) \notag
\end{align}
where $\delta_i(\eta) \rightarrow 0$ as $\eta \rightarrow 0$, $i=1,2$. We now choose $\eta>0$ such that 
\[
\delta_2(\eta) K <\frac \epsilon 4 \quad \text{and} \quad T \delta_1(\eta)<\frac \epsilon 4,
\]
and choose $m>m_0(\eta)$; this yields
\[
I_T(\phi^\eta)  < K+\epsilon.
\]
\end{proof}

In the following, we show a relation between $L^m$ and $L_\delta$.

\begin{remark} \label{Remark5.17Delta}
It can easily be seen that Lemma~\ref{Lemma5.17} holds for $L_\delta$ (with exactly the same proof). The same holds true for Lemma~\ref{Lemma5.23}  and Lemma~\ref{Lemma5.18}.
\end{remark}

\begin{lemma} \label{Lem4.4}
Let $\beta_j$ ($j=1,\dots,k$) be bounded and $\epsilon>0$. Then there exists an $m_0>0$ such that for all $m> m_0$, there exists an $\delta_0>0$ such that for all $\delta<\delta_0$ and all $x \in A$ $y \in \mathds{R}^d$,
\[
L^m(x,y) \leq \epsilon + (1+\epsilon) L_\delta (x,y)
\]
\end{lemma}

\begin{proof}
Let $m_0>0$, $m>m_0$ and $\delta>0$. We let $\mu^*=\mu^*(z^*,y)$ be the optimal $\mu$ associated to the optimal $z^*$ according to 
Equation~\eqref{ell}. Then
\begin{equation}
L^m(x,y) - L_\delta(x,y) \leq  \ell^m(x,\mu^*) -  \ell ( z^*,\mu^*). \label{EqLem4.4.1}
\end{equation}
Furthermore, we have by the uniform continuity of the $\beta_j$ (cf.~the proof of Lemma~\ref{Lem4.2}),
\begin{equation}
|\beta^m_j(x) - \beta_j(z^*)| \leq \frac 1 m + K(\delta)=: K_1(m,\delta).  \label{EqLem4.4.2}
\end{equation}
Moreover, we note that if $\beta_j(x) < \frac 1 m - K(\delta)$, then
\begin{equation}
\log \frac{\beta_j(z^*)}{\beta_j^m(x)} < 0. \label{EqLem4.4.3}
\end{equation}
On the other hand, if $\beta_j(x) \geq \frac 1 m - K(\delta)$, then
\begin{equation}
\log \frac{\beta_j(z^*)}{\beta_j^m(x)} \leq\log \frac{\beta^m_j(z^*)}{\beta_j^m(x)} \leq \log \frac{\frac 1 m + K(\delta)}{\frac 1 m} \leq m K(\delta) =: K_2(m,\delta). \label{EqLem4.4.4}
\end{equation}
By Lemma~\ref{Lemma5.17} and Lemma~\ref{Lemma5.20}, there exist constants $B$, $C_1$, $C_3$ and $C_4$ such that for all $x\in A$ and $y \in \mathcal C_x$,
\begin{align}
L(x,y), L_\delta(x,y) > C_1 |y| \log B \quad \text{if } |y| > B,  \label{EqLem4.4.5} \\
|\mu^*(y)| \leq C_3 |y| \quad \text{if } |y| > B,  \label{EqLem4.4.6} \\
|\mu^*(y)| \leq C_4 \quad \text{if } |y| \leq B  \label{EqLem4.4.7} 
\end{align}
(note that the constants in Inequality~\eqref{EqLem4.4.5}  are independent of $\delta$).

For $|y|>B$, we have by Inequalities~\eqref{EqLem4.4.1}, \eqref{EqLem4.4.2}, \eqref{EqLem4.4.3}, \eqref{EqLem4.4.4}, \eqref{EqLem4.4.5} and \eqref{EqLem4.4.6},
\begin{align}
L^m(x,y) - L_\delta(x,y) & \leq k K_1(m,\delta) + k C_3 |y|  K_2(m,\delta) \notag\\
& \leq k K_1(m,\delta) + \frac{k C_3 K_2(m,\delta)}{C_1 \log B} L_\delta (x,y). \label{EqLem4.4.8} 
\end{align}
For $|y|\leq B$, we have by Inequalities~\eqref{EqLem4.4.1}, \eqref{EqLem4.4.2}, \eqref{EqLem4.4.3}, \eqref{EqLem4.4.4} and \eqref{EqLem4.4.7},
\begin{equation}
L^m(x,y) - L_\delta(x,y)  \leq k K_1(m,\delta) + k C_4 K_2(m,\delta) \label{EqLem4.4.9} 
\end{equation}
The assertion now follows from Inequalities~\eqref{EqLem4.4.8} and \eqref{EqLem4.4.9} by choosing $m_0$ large enough, $m>m_0$ and $\delta_0=\delta_0(m)$ small enough such that
\[
k K_1(m,\delta_0), k C_4 K_2(m,\delta_0), \frac{k C_3 K_2(m,\delta_0)}{C_1 \log B} < \frac \epsilon 2.
\]
\end{proof}

We directly deduce the following result 

\begin{corollary} \label{Cor4.1}
Let $\beta_j$ ($j=1,\dots,k$) be bounded and continuous. For all $\epsilon,K,T>0$, there exists an $m_0>0$ such that for all $m>m_0$, there exists a $\delta_0>0$ such that for all $\delta< \delta_0$ and all functions $\phi$ with $I^\delta_T(\phi)\leq K-\epsilon$,
\[
I^m_T(\phi) < K.
\]
\end{corollary}

We now deduce from Lemma~\ref{Lem4.2} and Corollary~\ref{Cor4.1} the analog of Corollary~4.2 from~\cite{Shwartz2005}.

\begin{corollary} \label{Cor4.2}
Assume that Assumption~\ref{MainAss} holds. Then for all $\epsilon,K>0$, there exists a $\delta_0>0$ such that for all $\delta<\delta_0$,
\[
\Phi^\delta_x(K-\epsilon) \subset \big\{ \phi\in D([0,T];A) | \di(\phi,\Phi_x(K)) \leq \epsilon\big\}.
\]
\end{corollary}

\begin{proof}
Let $\epsilon>0$ and choose $m_0$, $m$, $\delta_0$, $\delta$ according to Corollary~\ref{Cor4.1} for $\epsilon/2$. Let $\phi \in \Phi_x^\delta(K-\epsilon)$. Then by Corollary~\ref{Cor4.1}, $\phi \in \Phi_x^m(K-\epsilon/2)$. By  Lemma~\ref{Lem4.2}, there exists a $\tilde \phi$ such that
\[
\| \tilde \phi - \phi \| < \epsilon \quad \text{and} \quad I_{T,x} ( \tilde \phi) \leq K.
\]
\end{proof}
\color{black}

\subsection{Distance of $Y^N$ to $\Phi^\delta$} \label{SubSecUpper3}

In this section, we derive a result about the distance of $Y^N$, defined by \eqref{defYN}, to $\Phi^\delta$ (Lemma~\ref{Lemma4.7} below).

We state the following elementary result (see, e.g.,~\cite{Roydon1968}, Chapter 3, Proposition 22).

\begin{lem} \label{LemmaRoydon}
Let $f:[a,b]\rightarrow \mathds{R}^d$ be measurable with $f \not=\infty$ almost everywhere. For all $\epsilon>0$, there exists a step function $g$ such that $|g-f|<\epsilon$ except on a set with measure less than $\epsilon$. Moreover the range of $g$ is a subset of the convex hull of the range of $f$.
\end{lem}
\color{black}

 We define for $\delta>0$, $\phi:[0,T]\rightarrow A$ and Borel-measurable $\theta:[0,T]\rightarrow \mathds{R}^d$,
\[
I^\delta_T(\phi,\theta):=\int_0^T \tilde\ell_\delta(\phi(t), \phi'(t),\theta(t)) dt.
\]

\begin{lem} \label{Lemma5.43}
Let $\log\beta_j$ ($j=1,\dots,k$) be bounded.
For all absolutely continuous $\phi:[0,T]\rightarrow A$ with $I^\delta_T(\phi) <\infty$ and all $\epsilon>0$ there exists a step function $\theta:[0,T]\rightarrow \mathds{R}^d$ such that
\[
I^\delta_T(\phi,\theta) \geq 
I^\delta_T(\phi)-\epsilon.
\]
\end{lem}

\begin{proof}
As $I^\delta(\phi)<\infty$, there exists a large enough positive number $B$ such that 
%for any function $\theta:[0,T]\rightarrow \mathds{R}^d$
\begin{equation} \label{EqLemma5.43.1}
%\int_0^T\mathds{1}_{\{|\phi'(t)|>B\}} \tilde\ell_\delta (\phi'(t),\phi(t),\theta(t))dt \leq
\int_0^T\mathds{1}_{\{|\phi'(t)|>B\}} L_\delta (\phi'(t),\phi(t))dt \leq \frac{\epsilon}{3}
\end{equation}
(cf.~Lemma~\ref{Lemma5.18} and Remark~\ref{Remark5.17Delta}).
We set 
\[
\theta_1(t):=0 \quad \text{if } |\phi'(t)| >B.
\]

By Lemma~\ref{Lemma5.23}  (which holds true with $L$ replaced by $L_\delta$, see Remark \ref{Remark5.17Delta} --
this is where we need the assumption that the $\log \beta_j$ are bounded),
there exists a constant $\tilde B$ such that for all $x \in A$ and $y\in \mathcal{C}$ with $|y|\leq B$,
\[
\sup_{|\theta|\leq \tilde B} \tilde\ell_\delta(x,y,\theta) > L_\delta(x,y) -\frac{\epsilon}{6 T}.
\]
We set
\[
D:=\{ (x,y,\theta)| \  x\in A, y\in \mathcal{C}, |y| \leq B,|\theta| \leq \tilde B \}.
\]
The function $\tilde\ell_\delta$ is uniformly continuous on $D$. Hence there exists an $\eta>0$ such that for $|x-\tilde x|, |y-\tilde y|, |\theta - \tilde \theta| < \eta$,
\[
|\tilde\ell_\delta(x,y,\theta) -\tilde\ell_\delta(\tilde x, \tilde y,\tilde \theta)| < \frac{\epsilon}{6 T}.
\]
By a compactness argument, we obtain a finite cover $\{\theta_{i,j},x_i,y_j\}$ of $D$ such that
\begin{equation}\label{EqLemma5.43.3}
 \tilde\ell_\delta(x_i,y_j,\theta_{i,j}) \geq L_\delta(x,y) -\frac{\epsilon}{3 T} \quad \text{for } |x-x_i|, |y-y_j|<\eta.
\end{equation}
We set 
\[
\theta_1(t):=\theta_{i,j} \quad \text{if } |\phi(t)-x_i|, |\phi'(t)-y_j|<\eta
\]
(with some kind of tie-breaking rule). Hence $\theta_1$ only takes finitely many values.  However, it is not clear whether $\theta_1$ is piecewise constant.

We now choose $\tilde \eta$ small enough such that $\text{Leb}[E]<\frac{\tilde \eta}{2}$ implies that
\begin{equation} \label{EqLemma5.43.4}
\int_{[0,T] \cap E} L_\delta(\phi(t),\phi'(t))dt\vee\int_{[0,T] \cap E}\sup_{|\theta|\le\tilde B}\left(-\tilde\ell_\delta(\phi(t),\phi'(t),\theta)\right)dt\ < \frac{\epsilon}{3}.
\end{equation}
and
\[
\min |\theta_{i,j} - \theta_{l,m}|> \tilde \eta,\  \   \min |\theta_{i,j}| > \tilde\eta.
\]
By~Lemma~\ref{LemmaRoydon}, there exists a step function $\theta_2$ with $|\theta_1-\theta_2|<\frac{\tilde\eta}{2}$ except on a set $\tilde E $ with 
Lebesgue measure $< \frac{\tilde \eta}{2}$.

 Hence there exists a step function $\theta$ which agrees with $\theta_1$ except on $\tilde E$ (modify $\theta_2$ if necessary such that $|\theta_1-\theta_2|<\frac{\tilde\eta}{2} \Rightarrow \theta_2=\theta_1$ on $\tilde E^c$). Note that $|\theta(t) |\le \tilde B$, for all $t\in[0,T]$.

We conclude by collecting the approximations above:
\begin{align*}
I^\delta_T(\phi) & =\int_{[0,T] } L_\delta(\phi(t),\phi'(t))dt \\*
&\leq  \int_{[0,T] \cap\{|\phi'(t)|> B\}}   L_\delta(\phi(t),\phi'(t))dt + \int_{[0,T] \cap \tilde E} L_\delta(\phi(t),\phi'(t))dt \\*
	&\qquad + \int_{[0,T] \cap(\{|\phi'(t)|\leq B\}\cup \tilde E ^c)} L_\delta(\phi(t),\phi'(t))dt \\
	%\stackrel{\eqref{EqLemma5.43.1},\eqref{EqLemma5.43.3},\eqref{EqLemma5.43.4}}
& {\leq} \frac{2\epsilon}{3} + \int_{[0,T] \cap(\{|\phi'(t)\leq B\}\cup \tilde E ^c)} \tilde\ell_\delta(\phi(t),\phi'(t),\theta(t))dt\\
&= \frac{2\epsilon}{3}+\! \int_{[0,T]}\!\!\!\tilde\ell_\delta(\phi(t),\phi'(t),\theta(t))dt -\!\int_{[0,T]\cap\tilde E}\!\!\!\!\!\!\!\!\!\!\!\tilde\ell_\delta(\phi(t),\phi'(t),\theta(t))dt
-\!\int_{[0,T]\cap\{|\phi'(t)|> B\}\cap\tilde E^c}\!\!\!\!\!\!\!\!\!\!\!\!\!\!\!\!\!\!\!\!\!\!\!\!\!\!\!\!\!\tilde\ell_\delta(\phi(t),\phi'(t),\theta(t))dt\\
&\le\epsilon+\! \int_{[0,T]}\!\!\!\tilde\ell_\delta(\phi(t),\phi'(t),\theta(t))dt.
\end{align*}
Indeed $\theta(t)=0$ on the set $\{|\phi'(t)|> B\}\cap\tilde E^c$, while \eqref{EqLemma5.43.4} implies that the second integral in the next to last line is bounded by $\epsilon/3$.
\end{proof}

We next prove.

\begin{lem}\label{Lemma4.5}
 Let $u:[0,T] \rightarrow \mathds{R}$ be nonnegative and absolutely continuous and $\delta>0$. Then there exists an $\eta>0$, a Borel set $E\subset [0,T]$ with 
 $\text{Leb}(E)<\delta$ and two finite collections $(J_i)_{i\in \mathcal{I}_+}$ and $(H_j)_{j\in\mathcal{I}_0}$ of subintervals of $[0,T]$ such that 
 \[
 [0,T]= E \cup \bigcup_{i\in \mathcal{I}_+} J_i \cup \bigcup_{j\in\mathcal{I}_0} H_j
 \]
 and for all $i\in\mathcal{I}_+$, $j\in\mathcal{I}_0$,
 \[
 \inf_{t\in J_i} u(t) >\eta, \quad  \quad u(t)=0 \text{ on } H_j\cap E^c.
 \]
\end{lem}

\begin{proof}
Given $t\in[0,T]$ such that $u(t)>0$, let $\mathcal{O}_t$ be the largest open interval containing $t$ such that $u(s)>0$ for all $s\in\mathcal{O}_t$.
Let $m_t=\max\{u(t),\ t\in\mathcal{O}_t\}$. Since $u$ is absolutely continuous, there is a finite number of intervals $\mathcal{O}_t$ such that $m_t>1/m$, for each $m\ge1$.
Hence there are at most countably many open intervals $\{\mathcal{O}_i,\ i\ge1\}$ where $u$ is positive. Choose $M$ large enough such that 
$$\text{Leb}\left(\cup_{i=M+1}^\infty \mathcal{O}_i\right)\le\frac{\delta}{2}.$$
For $1\le i\le M$, let $J_i\subset \mathcal{O}_i$ be a closed interval such that 
$$\text{Leb}\left(\mathcal{O}_i\setminus J_i\right)\le\frac{\delta}{2M}.$$
Let 
$$E=\left(\cup_{i=M+1}^\infty\mathcal{O}_i\right)\cup\left(\cup_{i=1}^M\mathcal{O}_i\setminus J_i\right).$$
Clearly  Leb$(E)\le\delta$. Let $\overline M$ be the number of connected components of $[0,T]\setminus\cup_{i=1}^MJ_i$.
For $1\le j\le \overline M$, let $H_j$ denote the closure of the $j$--th connected component of
$[0,T]\setminus\cup_{i=1}^MJ_i$. $H_j$ is an interval. Moreover
\begin{align*}
\inf_{1\le i\le M}&\inf_{t\in J_i} u(t)=\eta>0,\quad\text{and}\\
u(t)&=0,\quad\text{if } t\in H_j\cap E^c.
\end{align*}
\end{proof}

We require this result for the proof of Lemma~4.6 of~\cite{Shwartz2005}. This is a (more general) variant of Lemma~5.43 of~\cite{Shwartz1995} (cf.~also Lemma~\ref{Lemma5.43}).

\begin{lemma}
\label{Lemma4.6}
Assume that $\beta_j$ ($j=1,\dots,k$) is bounded and Lipschitz continuous. Then for all $\phi$ with $I_T(\phi)<\infty$ and $\epsilon>0$, there exists a step function $\theta$ such that
\[
I_T^\delta(\phi,\theta) \geq I_T^\delta(\phi) - \epsilon.
\]
\end{lemma}

\begin{proof}
If none of the $\beta_j(\phi(t))$ vanishes on the interval $[0,T]$, then the proof of Lemma \ref{Lemma5.43} applies.
If that is not the case, we note that since $\phi$ is absolutely continuous and $\beta_j$ is Lipschitz continuous, $t\to\beta_j(\phi(t))$ is absolutely continuous.
Hence we can apply Lemma \ref{Lemma4.5} to the function $u(t):=\beta_j(\phi(t))$, and associate to each $1\le j\le k$ intervals $(J^j_i)_{i\in \mathcal{I}_+}$ 
and $(H^j_i)_{i\in\mathcal{I}_0}$. It is not hard to see that to each $\eta>0$ one can associate 
a real $\eta>0$, an integer $M$, a collection $(I_i)_{1\le i\le M}$  of subintervals of $[0,T]$, with the following properties
 \[
 [0,T]= E \cup \bigcup_{1\le i\le M} I_i, 
 \]
 with Leb$(E)\le\delta$, and moreover to each $1\le i\le N$ we can associate a subset $\mathcal{A}\subset\{1,2,\ldots,k\}$ such that 
 $$\beta_j(\phi(t))>\eta,\ \text{ if }j\in\mathcal{A}, t\in I_i,\quad \text{and }\beta_j(\phi(t))=0,\ \text{ if }j\not\in\mathcal{A}, t\in I_i\cap E^c.$$
 Each interval $I_i$ is an intersection of $J^j_i$'s for $j\in\mathcal{A}$ and of $H^j_i$'s for $j\not\in\mathcal{A}$.
 
 On each subinterval $I_i$, by considering the process with rates and jump directions $\{\beta_j, h_j, j\in\mathcal{A}\}$ only, we can deduce from Lemma  \ref{Lemma5.43}
 that there exists a step function $\theta$ such that
 $$\int_{I_i}\tilde\ell_\delta(\phi(t),\phi'(t),\theta(t))\ge\int_{I_i}L_\delta(\phi(t),\phi'(t))dt-\frac{\epsilon}{2M}.$$
 In fact there exists a unique stepfunction $\theta$ defined on $[0,T]$, such that each of the above inequality holds and
 moreover, by the same argument as in the proof of Lemma \ref{Lemma5.43}, provided $\eta$ is small enough, 
 $$\int_E\tilde\ell_\delta(\phi(t),\phi'(t),\theta(t))\ge-\frac{\epsilon}{2}.$$
 The result follows.
\end{proof}

We now define for $M \in \mathds{N}$
\[
\mathcal{K}(M):=\bigcap_{m\geq M} \{ \phi \in C([0,T];A)| V_{2^{-m}}(\phi) \leq \tfrac{1}{\log m} \},
\]
where $V_\delta$ is the modulus of continuity:
\[
V_\delta(\phi):=\sup_{s,t\in [0,T], \, |s-t|< \delta } |\phi(s)-\phi(t)|.
\]
We readily observe (Arzel\`{a}-Ascoli) that $\mathcal{K}(M)$ is compact in $C([0,T];\mathds{R}^d)$.

We next obtain exponential tightness for the sequence $Y^{N,x}$ defined in \eqref{defYN}.

\begin{lem}\label{Lemma5.58}
Assume that $\beta_j$ ($j=1,\dots,k$) is bounded. 
There exists a positive constant $a$ such that for all $M$ large enough and for all $x \in A$,
%For all $B>0$ there exists a compact set $\mathcal{K}\subset C([0,T];\mathds{R}^d)$ such that for all $x \in A$,
\[
\limsup_{N\rightarrow \infty} \frac{1}{N} \log \mathbb{P}[Y^{N,x}\not\in\mathcal{K}(M)] \leq -a\frac{M}{\log M}.
\]
%More precisely, there exists a (universal) constant $C_4$ such that for all $M\in \mathds{N}$,
%\[
%\limsup_{N\rightarrow \infty} \frac{1}{N} \log \mathbb{P}[Y^{N,x}\not\in\mathcal{K}(M)] \leq -C_4\frac{M}{\log M}.
%\]
\end{lem}

\begin{proof}
Suppose that
\begin{equation}\label{suff-condit}
V_{2^{-m}}(Y^N)\le\frac{1}{\log m},\text{ for }m=M,\ldots,M(N),\text{ where }M(N)=\left\lceil\frac{\log(N/T)}{\log 2}\right\rceil.
\end{equation}
It is plain that $m\ge M(N)$ implies that $2^{-m}<T/N$, hence $V_{2^{-(m+1)}}(Y^N)=\frac{1}{2}V_{2^{-m}}(Y^N)$.
Then, provided $N>4T$, $M(N)\ge2$, hence for any $m\ge M(N)$, $m+1\le m^2$, and also $(2\log m)^{-1}\le(\log(m+1))^{-1}$, 
 and it follows that \eqref{suff-condit} implies that $Y^{N,x}\in\mathcal{K}(M)$.
 
Now if  $M\le m\le M(N)$, 
$$\left\{V_{2^{-m}}(Y^N)>\frac{1}{\log m}\right\}\subset\bigcup_{j=0}^{N-1}\left\{\sup_{t_j\le s\le t_j+2^{1-m}}|Z^{N,x}_s-Z^{N,x}_{t_j}|>\frac{1}{2\log m}\right\}.$$
Consequently, with the help of Lemma \ref{Cor5.55}, for some $C>0$ and provided $M$ is large enough,
\begin{align*}
\mathbb{P}[Y^{N,x}\not\in\mathcal{K}(M)]&\le\sum_{m=M}^{M(N)}\mathbb{P}\left(V_{2^{-m}}(Y^N)>\frac{1}{\log m}\right)\\
&\le N\sum_{m=M}^{M(N)}\exp\left(-\frac{CN}{\log m}\log\left(\frac{C 2^m}{\log m}\right)\right)\\
&\le M(N)N\exp\left(-\frac{CN}{\log M}\log\left(\frac{C 2^M}{\log M}\right)\right),
\end{align*}
where the last inequality follows from the fact that for $x>0$ large enough, the mapping $x\to(\log x)^{-1}\log\left(\frac{C 2^x}{\log x}\right)$ is increasing. 
Consequently
$$\frac{1}{N}\log\mathbb{P}\left[Y^{N,x}\not\in\mathcal{K}(M)\right]\le \frac{\log M(N)}{N}+\frac{\log N}{N}-\frac{c}{\log M}\log\left(\frac{C2^M}{\log M}\right).$$
It remains to take the limsup as $N\to\infty$.
\end{proof}

We now establish the main local estimate for $Y^N$.

\begin{lem} \label{Lemma5.59}
Assume that $\beta_j$ ($j=1,\dots,k$) is bounded. For all $\delta>0$,
we have uniformly in $x \in A$ and $\theta$ in a bounded set,
\[
\limsup_{N \rightarrow \infty} \log\mathbb{E} \big[ \exp\big(N\langle Y^{N,x}(\tfrac{T}{N})-Y^{N,x}(0),\theta \rangle \big) \big]
\leq T \cdot  \sum_{j=1}^k  \sup_{z^j \in A,\, |z^j-x|\leq \delta}\beta_j(z^j)( \e^{\langle \theta, h_j \rangle} -1).
\]
\end{lem}

\begin{proof}
 It is not hard to verify that for any $\theta\in\mathbb{R}^d$, the process
 $$M^\theta_t:=\exp\left(N\langle Z^{N,x}_t-x,\theta\rangle-N\sum_{j=1}^k(e^{\langle\theta,h_j\rangle}-1)\int_0^t\beta_j(Z^{N,x}(s))ds\right)$$
 is a martingale with $M^\theta_0=1$, hence $\mathbb{E}[M^\theta_t]=1$. Let
 $$S_{N,\delta}:=\{\sup_{0\le t\le T/N}|Z^{N,x}_t-x]\le\delta\}.$$
 Since $M^\theta_t>0$, $\mathbb{E}[M^\theta_{T/N}{\bf1}_{S_{N,\delta}}]\le1$. But on the event $S_{N,\delta}$, 
 $$M^\theta_{T/N}\ge\exp\left(N\langle Z^{N,x}(T/N)-x,\theta\rangle-T\sum_{j=1}^k\sup_{z^j\in A, |z^j-x|\le\delta}\beta_j(z^j)(e^{\langle\theta,h_j\rangle}-1)\right),$$ hence
 $$\mathbb{E}\left[\exp\left(N\langle Z^{N,x}(T/N)-x,\theta\rangle\right){\bf1}_{S_{N,\delta}}\right]
 \le\exp\left(T\sum_{j=1}^k\sup_{z^j\in A, |z^j-x|\le\delta}\beta_j(z^j)(e^{\langle\theta,h_j\rangle}-1)\right).$$
 On the other hand, from Lemma \ref{Cor5.55}, for some $C>0$, whenever $|\theta|\le B$,
 \begin{align*}
 \mathbb{E}\left[\exp\left(N\langle Z^{N,x}(T/N)-x,\theta\rangle\right){\bf1}_{S^c_{N,\delta}}\right]
 &\le \sum_{\ell=1}^\infty e^{N(\ell+1)\delta|\theta|}\mathbb{P}\left(\ell\delta\le|Z^{N,x}(T/N)-x|\le(\ell+1)\delta\right)\\
 &\le \sum_{\ell=1}^\infty\exp\left(N\delta\left[(\ell+1) B -C\ell\log(CN\ell\delta)\right]\right)\\
 &\le\sum_{\ell=1}^\infty a(N,\delta)^\ell\\
 &\le 2 a(N,\delta),
 \end{align*}
 provided $N$ is large enough, such that
 $$a(N,\delta):=\exp\left(N\delta\left[2 B-C\log(CN\delta)\right]\right)\le 1/2.$$
 Finally
 $$\mathbb{E}\left[\exp\left(N\langle Z^{N,x}(T/N)-x,\theta\rangle\right)\right]\le \exp\left(T\sum_{j=1}^k\sup_{z^j\in A, |z^j-x|\le\delta}\beta_j(z^j)(e^{\langle\theta,h_j\rangle}-1)\right)+2 a(N,\delta).$$
 The result follows from the fact that $a(N,\delta)\to0$ as $N\to\infty$, for any $\delta>0$. 
\end{proof}

%For a function $\theta:[0,T] \rightarrow \mathds{R}^d$ and absolute continuous $\phi:[0,T]\rightarrow A$ and $\delta>0$, we define
%\[
%I^\delta(\phi,\theta):=
%\begin{cases}
%\int_0^T \ell_\delta(\theta(t),\phi(t),\phi'(t)) dt & \text{ if $\phi$ is abs. cont.} \\
%0 & \text{ else}.
%\end{cases}
%\]

We next establish

\begin{lem} \label{Lemma5.61}
Let $\beta_j$ ($j=1,\dots,k$) be bounded and continuous.
Let $\theta:[0,T]\rightarrow \mathds{R}^d$ be a step function, $\delta>0$ and $\mathcal{K}\subset\mathcal{K}(M)$ be a compact set, such that the subset 
$\mathcal{K}^{ac}$ consisting of those elements of $\mathcal{K}$ which are absolutely continuous is dense in $\mathcal{K}$. Then
\[
\limsup_{N \rightarrow \infty} \frac{1}{N} \log \mathbb{P} [Y^{N,x} \in \mathcal{K} ] \leq - \inf_{\phi \in \mathcal{K},\, \phi(0)=x} I^\delta(\phi,\theta)
\]
uniformly in $x$.
\end{lem}

\begin{proof}
Let $\theta$ be a fixed step function from $[0,T]$ into $\mathds{R}^d$, which we assume w.l.o.g. to be right continuous, and let
$\mathcal{K}$ be a given compact subset of $C([0,T];\mathds{R}^d)$, which has the property that $\mathcal{K}^{ac}$ is dense in $\mathcal{K}$.
We define for each $\delta>0$ the mapping $g_\delta:\mathds{R}^{2d}\to\mathds{R}$ by
$$g_\delta(z,\theta)=\sum_{j=1}^k\sup_{|z_j-x|\le\delta}\beta_j(z_j)(e^{\langle \theta,h_j\rangle}-1).$$
We let $t_\ell:=\ell T/N$, and define for  $z\in\mathcal{K}^{ac}$, the two quantities
\begin{align*}
\tilde{S}_N(z,\theta)&=\sum_{\ell=1}^N\langle z(t_\ell)-z(t_{\ell-1}),\theta(t_{\ell-1})\rangle-\frac{T}{N}\sum_{\ell=1}^N g_\delta (z(t_{\ell-1}),\theta(t_{\ell-1})),\\
S(z,\theta)&=\int_0^T\langle z'(t),\theta(t)\rangle dt -\frac{T}{N}\sum_{\ell=1}^N g_\delta (z(t_{\ell-1}),\theta(t_{\ell-1})).
\end{align*}
Choose any $\eta>0$.
We can assume that $N_0$ has been chosen
large enough, such that 
$$\sup_{z\in\mathcal{K}^{ac}}|\tilde{S}_N(z,\theta)-S(z,\theta)|\le\eta.$$
Indeed, this difference is bounded by twice the number of jumps
of $\theta$ times the sup of $|\theta(t)|$, times the maximal oscillation  of $z$ on intervals of length $1/N$ in $[0,T]$.

It follows from Lemma \ref{Lemma5.59} and the Markov property that, provided $N_0$ has been chosen large enough, for any $N\ge N_0$,
\begin{equation*}
\mathbb{E}\left[\exp\left(N\tilde{S}_N(Y^{N,x},\theta)\right)\right]\le\exp(N\eta).
\end{equation*}
Clearly, on the event $Y^{N,x}\in\mathcal{K}$, 
$$\exp\left[N\left(\tilde{S}_N(Y^N,\theta)-\inf_{z\in\mathcal{K}^{ac}}\tilde{S}_N(z,\theta)\right)\right]\ge1,$$
and combining this fact with the previous inequalities, we deduce that
\begin{align*}
\mathbb{P}(Y^{N,x}\in\mathcal{K})&\le\mathbb{E}\exp\left[N\left(\tilde{S}_N(Y^{N,x},\theta)-\inf_{z\in\mathcal{K}^{ac}}\tilde{S}_N(z,\theta)\right)\right]\\
&\le\exp(N\eta)\exp\left(-N\inf_{z\in\mathcal{K}^{ac}}\tilde{S}_N(z,\theta)\right)\\
&\le\exp(2N\eta)\exp\left(-N\inf_{z\in\mathcal{K}^{ac}}S_N(z,\theta)\right)
\end{align*}
Now, uniformly in $z\in\mathcal{K}^{ac}$, $S_N(z,\theta)\to I^\delta(z,\theta)$, where
$$I^\delta(z,\theta)=\int_0^T\langle z'(t),\theta(t)\rangle dt-\int_0^Tg_\delta(z(t),\theta(t))dt.$$
The result follows from the last two facts, since $\eta>0$ can be chosen arbitrarily small, and
$\mathcal{K}^{ac}$ is dense in $\mathcal{K}$.
\end{proof}

We now have

\begin{lem} \label{Lemma4.7}
Assume that $\beta_j$ ($j=1,\dots,k$) is bounded and Lipschitz continuous.
Then for all $K>0$, $\delta>0$ and $\epsilon>0$,
\[
\limsup_{N \rightarrow \infty} \frac{1}{N} \log \mathbb{P} [ \di(Y^{N,x} , \Phi_x^\delta(K))>\epsilon ] \leq -K+\epsilon
\]
uniformly in $x \in A$.
\end{lem}

\begin{proof}
We fix $\epsilon,\delta,K>0$ and choose $M\in \mathds{N}$ such that $a \frac{M}{\log M} > K-\epsilon$, where $a$ is the constant appearing 
in Lemma  \ref{Lemma5.59}.

For absolute continuous $\phi:[0,T] \rightarrow A$ with $I^\delta(\phi) <\infty$, there exists a step function $\theta^\phi$ such that
\[
I^\delta(\phi,\theta^\phi) \geq I^\delta(\phi)-\tfrac{\epsilon}{2}
\]
(cf.~Lemma~\ref{Lemma5.43}). It can easily be verified by elementary calculus that the function $I^\delta(\cdot,\theta^\phi)$ is continuous 
for the sup norm topology on the set of absolutely continuous functions. Hence there exists a number $0 < \eta^\phi<\frac{\epsilon}{2}$ such 
that for all absolutely continuous $\tilde \phi$ with $\|\phi - \tilde \phi\| < \eta^\phi$,
\begin{equation} \label{EqProp5.62.1}
I^\delta(\tilde \phi,\theta^\phi) \geq I^\delta(\phi)-\epsilon.
\end{equation}
We consider the compact set
\[
\mathcal{K}^x(M):=\{\phi \in\mathcal{K}(M) | \phi(0)=x\}
\]
(cf.~the definition preceding Lemma~\ref{Lemma5.58}).
By a compactness argument, there exist finitely many absolutely continuous functions $\{\phi_i,\ 1\le i\le m\} \subset \mathcal{K}^x(M)$ with $I^\delta (\phi_i) < \infty$ (and corresponding $\theta_i:=\theta^{\phi_i}$ and $\eta_i:=\eta^{\phi_i}$) such that
\[
\mathcal{K}^x(M) \subset \bigcup_{i=1}^m B_{\eta_i}(\phi_i).\]
%\footnotetext{The set of absolutely continuous functions with $I^\delta(\phi)<\infty$ is dense in $\mathcal{K}^x(M)$: every continuous function can be approximated by %piecewise linear functions which can in turn be approximated by piecewise linear functions which do not touch the boundary (thus by absolutely continuous functions with %$I^\delta(\phi)<\infty$).}
For each $1\le i\le m$, we define the compact set
\[
\mathcal{K}_i^x(M):=\overline{B_{\eta_i}(\phi_i) \cap \mathcal{K}^x(M)}.
\]

We now let 
\[
\mathcal{I}:=\{1\le i\le m\ | \di (\phi_i,\Phi_x^\delta(K))\geq \eta_i\}.
\]
Then $\di (Y^{x,N},\Phi_x^\delta(K))\geq \epsilon$ and $Y^{x,N} \in \mathcal{K}^x_i(M)$ imply $i \in \mathcal{I}$, since $\eta_i\le\epsilon/2$.
%\footnote{Else, there exists a $\tilde \phi \in \Phi_x^\delta(K)$ with $\|\tilde \phi - Y^{x,N}\| \leq \| \tilde \phi - \phi_i\| + \|\phi_i - Y^{x,N}\|\leq \eta_i
%+\frac{\epsilon}{2} < \epsilon$.}
Hence
\begin{align*}
&\limsup_{N\rightarrow \infty} \frac{1}{N} \log \mathbb{P}[\di (Y^{x,N},\Phi_x^\delta(K))\geq \epsilon] \notag \\*
	& \qquad \leq \limsup_{N\rightarrow \infty} \frac{1}{N} \log \big\{ \mathbb{P}[ Y^{x,N} \not\in \mathcal{K}^x(M)] +
	\sum_{i \in \mathcal{I}}\mathbb{P}[Y^{x,N} \in \mathcal{K}^x_i(M)] \big\}.% \footnotemark% \label{EqProp5.62.2}
\end{align*}
%\footnotetext{The case $\mathcal{K}_i^x(M) = \emptyset$ for all $i \in \mathcal{I}$ implies via Lemma~\ref{Lemma5.58} that
%\begin{align*}
%\limsup_{N\rightarrow \infty} \frac{1}{N} \log \mathbb{P}[\di (Y^{x,N},\Phi_x^\delta(K))\geq \epsilon] 
%& \leq \limsup_{N\rightarrow \infty} \frac{1}{N} \log \mathbb{P}[ Y^{x,N} \not\in \mathcal{K}^x(M)] 
% \leq - C_4 \frac{M}{\log M} <- K + \epsilon,
%\end{align*}
%finishing the proof. We can hence assume that there exists at least one $i \in \mathcal{I}$ with $\mathcal{K}_i^x(M) \not= \emptyset$.}

Applying first Lemma~\ref{Lemma5.61} and then \eqref{EqProp5.62.1}, we obtain
\begin{align*}
\limsup_{N\rightarrow \infty} \frac{1}{N} \log \mathbb{P}[Y^{x,N} \in \mathcal{K}^x_i(M)] &\leq - \inf_{\phi \in \mathcal{K}_i^x(M)} I^\delta(\phi,\theta_i) \notag \\
&\leq - I^\delta (\phi_i) + \epsilon \notag \\
&< - K + \epsilon \label{EqProp5.62.3}
\end{align*}
as $I^\delta(\phi_i) >K$ (recall that $i \in \mathcal{I}$). The result now follows from the two last inequalities,
Lemma~\ref{Lemma5.58} and the fact that $a \frac{M}{\log M} > K-\epsilon$.
\end{proof}

\subsection{Main results} \label{SubSecUpperMain} 

%The main result of this section is the generalization of Theorem~4.1 of~\cite{Shwartz2005}.

\begin{theorem}\label{th-upper}
Assume that Assumption~\ref{MainAss} is satisfied.
For $F\subset D([0,T];A)$ closed and $x \in A$, we have
\[
\limsup_{y_N\in A^N, y_N \rightarrow x, \,N \rightarrow \infty} \frac{1}{N} \log \mathbb{P} [Z^{N,y_N} \in F] \leq -I_x(F).
\]
\end{theorem}

\begin{proof}
%\color{red}
%Requires:
%Lemma~\ref{Lemma5.63}, Prop~\ref{Prop5.46}, Lemma~\ref{Lemma5.57}, Corollary~\ref{Cor4.2} and Lemma~\ref{Lemma4.7}.
%\color{black}
We first let $I_x(F)=:K<\infty$ and $\epsilon>0$. By Lemma~\ref{Lemma5.63}, there exits a $\delta^\epsilon>0$ such that for all $\delta \leq \delta^\epsilon$,
\begin{equation} \label{EqTheo5.64.1}
y \in A, |x-y| < \delta \Rightarrow I_y(F) \geq I_x(F) - \epsilon =K- \epsilon.
\end{equation}
For $\delta \leq \delta^\epsilon$, we define
\begin{align*}
F^\delta &:=\{\phi \in F ||\phi(0)-x| \leq \delta \}, \\
S^\delta &:= \bigcup_{y \in A, |x-y|\leq \delta} \Phi_y(K-2\epsilon).
\end{align*}
$F^\delta$ is closed in $D([0,T];A;\di_D)$ and $S^\delta$ is compact in $D([0,T];A;\di_D)$ by Proposition~\ref{Prop5.46}. 
%(cf.~also Corollary~A.60 of~\cite{Shwartz1995}). 
Furthermore, the two sets have no common elements.
%\footnote{$|\phi(0)-x| \leq \delta \stackrel{\eqref{EqTheo5.64.1}}{\Rightarrow} I_{\phi(0)} (F) \geq K-\epsilon$. As $\phi\in F^\delta \subset F$, 
%this implies $I_{\phi(0)}(\phi) \geq K- \epsilon$ and hence $\phi \not\in S^\delta$.}
Hence, by the Hahn-Banach Theorem, %(cf.,e.g, \cite{Shwartz1995} Theorem~A.19)
\begin{equation}\label{HB}
\di(F^\delta,S^\delta)=:\eta^\delta>0.
\end{equation}
Note that $\eta^\delta$ is increasing as $\delta$ is decreasing, since the sets $F^\delta$ and $S^\delta$ are decreasing.
We now let $|y-x|\leq \delta$ and $\eta \leq \eta^\delta$. Let $Y^N$ be defined as in the paragraph preceding Lemma~\ref{Lemma5.57}. We have
\begin{align}
\mathbb{P}[Z^{N,y} \in F] &= \mathbb{P}[Z^{N,y} \in F^\delta] \notag \\
& \leq \mathbb{P} [ \di(Y^{N,y},F^\delta) <\tfrac{\eta}{2}] + \mathbb{P} [ \|Y^{N,y} - Z^{N,y}\| \geq \tfrac{\eta}{2}]. \label{EqTheo5.64.2}
\end{align} 
%\footnotetext{\text{If $\di(Y^{N,y},F^\delta) \geq \tfrac{\eta}{2}$, then $Z^{N,y} \in F^\delta$ implies $\|Z^{N,y} \in Y^{N,y}\| \geq \tfrac{\eta}{2}$.}}
Let now $\phi(0)=y$ with $\di(\phi,F^\delta) < \frac{\eta}{2}$, hence from \eqref{HB}
\begin{equation} \label{EqTheo5.64.3}
\di(\phi,\Phi_y(K-2\epsilon)) \geq \frac{\eta}{2}. 
\end{equation} 
%\footnotetext{This follows from the fact that $\Phi_y(K-2\epsilon)\subset S^\delta$ as $|y-x|\leq \delta$ and $\di(S^\delta,F^\delta)=\eta^\delta \geq \eta$.}
Let $\tilde\delta$ be such that Corollary~\ref{Cor4.2} with $K$ replaced by $K- 2 \epsilon$ and $\epsilon$ by $\frac{\eta}{4}$ holds with $\delta$
replaced by $2\tilde\delta$.
Hence \eqref{EqTheo5.64.3} implies
\begin{equation} \label{EqTheo5.64.4}
\di (\phi,\Phi_y^{2 \tilde \delta}(K-2 \epsilon - \tfrac{\eta}{4}) > \frac{\eta}{4}. 
\end{equation}
Indeed, if that is not the case, there exists a $\tilde\phi \in \Phi_y^{2 \tilde \delta}(K-2 \epsilon - \tfrac{\eta}{4})$ with $\|\phi - \tilde \phi\| \leq \frac{\eta}{4}$. 
 Then Corollary~\ref{Cor4.2} %with $K$ replaced by $K- 2 \epsilon$ and $\epsilon$ by $\frac{\eta}{4}$ 
 implies that there exists $\bar \phi \in \Phi_y(K-2 \epsilon)$ with $\| \bar \phi - \tilde \phi \| \leq \frac{\eta}{4}$; consequently $\| \bar \phi - \phi \|  \leq \frac{\eta}{2}$,  which contradicts~\eqref{EqTheo5.64.4}.
We hence obtain by Lemma~\ref{Lemma4.7},
\begin{align}
\limsup_{N \rightarrow \infty} \frac{1}{N} \log \mathbb{P} [\di (Y^{N,y}, F^\delta) < \tfrac{\eta}{2}]
&  \leq
\limsup_{N \rightarrow \infty} \frac{1}{N} \log \mathbb{P} [\di (Y^{N,y},\Phi_y^{2 \tilde \Delta}(K-2 \epsilon - \tfrac{\eta}{4}) > \frac{\eta}{4}] \notag \\
&  \leq -(K - 2 \epsilon - \frac{\eta}{2}) \label{EqTheo5.64.5}
\end{align}
uniformly in $y \in A$ with $|y-x|\leq \delta$.

Furthermore, Lemma~\ref{Lemma5.57} implies
\begin{equation} \label{EqTheo5.64.6}
\limsup_{N\rightarrow \infty}  \frac{1}{N} \log \mathbb{P} [ \|Y^{N,y} - Z^{N,y}\| \geq \tfrac{\eta}{2}] = -\infty
\end{equation}
uniformly in $y \in A$.

Combining Inequalities~\eqref{EqTheo5.64.2}, \eqref{EqTheo5.64.5} and~\eqref{EqTheo5.64.6}, we obtain
\[
\limsup_{N \rightarrow \infty} \frac{1}{N} \log \mathbb{P} [Z^{N,y} \in F]
\leq -(K - 2 \epsilon - \frac{\eta}{2}) 
\]
uniformly in $y \in A$, $|x-y| \leq \delta$. The result now follows as $\epsilon$ and $\eta$ can be chosen arbitrarily small.

The result in case $I_x(F)=\infty$ follows, since this implies that $I_x(F)>K$ for all $K>0$.
\end{proof}

We will need the following stronger version. Recall the definition of $A^N$ at the start of section \ref{SectionSetup}. 

\begin{theorem}\label{ThUBunif}
Assume that Assumption~\ref{MainAss} is satisfied.
For $F\subset D([0,T];A)$ closed and any compact subset $K \subset A$, we have
\[
\limsup_{N \rightarrow \infty} \frac{1}{N} \log\sup_{x\in K\cap A^N} \mathbb{P} [Z^{N,x} \in F] \leq -\inf_{\ x\in K}I_x(F).
\]
\end{theorem}
\begin{proof}
We use the same argument as in the proof of Corollary 5.6.15 in~\cite{Dembo2009}. From Theorem \ref{th-upper},
for any $x\in A$, any $\delta>0$, there exists $\eps_{x,\delta}>0$ and $N_{x,\delta}\ge1$ such that whenever $N\ge N_{x,\delta}$,
$y\in A_N$ with $| y-x|<\eps_{x,\delta}$, 
\[ \frac{1}{N}\log\P[Z^{N,y}\in F]\le - I_x^\delta(F),\]
where $I^\delta_x(F)=\min[I_x(F)-\delta,\delta^{-1}]$.
Consider now a compact set $K\subset A$. There exists a finite set $\{x_i,\ 1\le i\le I\}$ such that 
$K\subset\cup_{i=1}ë B(x_i,\eps_{x_i})$, where $B(x,\eps)=\{y;\ |y-x|<\eps\}$. Consequently, for
$N\ge\sup_{1\le i\le I}N_{x_i,\delta}$, any $y\in A_N\cap K$,
\[ \frac{1}{N}\log\P[Z^{N,y}\in F]\le -\min_{1\le i\le I} I_{x_i}^\delta(F)\le -\inf_{x\in K}I_{x}^\delta(F).\]
It remains to take the $\sup$ over $y\in K\cap A_N$ on the left, take the $\limsup$ as $N\to\infty$, and finally let $\delta$
tend to 0 to deduce the result.
\end{proof}

\section{Time of exit from domain} \label{SectionExitTime}
%{\color{red} ??? }
%As we want to include models with such ``absorbing sets'' as
%\[
%A_0 =\{x \in A| x_1=0\} \quad \text{for the SIRS model},
%\]
%we will slightly relax Assumption~\ref{MainAss} by Assumption~\ref{AssMod} below: we ``freeze'' problematic rates %close to $A_0$ such that for the modified frozen rates, Assumption~\ref{MainAss} is satisfied.

%\begin{assumption} \label{AssMod}
%Let $A_0\subset \partial A$ be closed. Let $\delta>0$. There exist closed sets $A^\delta$ and  rates 
%$\beta^\delta_j:A \rightarrow \mathds{R}_+$ with
%\[
%\dist(z,A_0) \geq \delta \text{ for all }z \in A^\delta \quad \text{and} \quad A^\delta \uparrow A  \text{ as } 
%\delta\rightarrow 0
%\]
%and
%\[
%\beta_j^\delta(x)=\beta_j(x) \quad \text{for all } x \in A^\delta
%\] 
%such that Assumption~\ref{MainAss} is satisfied for $\beta^\delta_j$ instead of $\beta_j$.
%\end{assumption}

%For the SIRS model, this can be achieved by defining
%\color{red}
%This is not what we want later.
%\color{black}
%\[
%A^\delta=\big\{z \in A|\dist(z,A_0)\geq \delta \big\}
%\]
%and
%\begin{equation} \label{EqAssModRatesSIRS}
%\beta_1^\delta(z)=\beta_1(\proj(z,A^\delta)), \quad \beta_3^\delta(z)=\beta_3(z), \quad \beta_2^\delta(x)=\beta_2(z), 
%\end{equation}
%where $\proj(x,A^\delta)$ denotes the projection of $x$ on $A^\delta$. 
%{\color{red}  ??? }

In this section we establish the results for the time of exit of the process from a domain; to this end, we follow the line of reasoning of~\cite{Dembo2009} Section 5.7 and modify the arguments when necessary.

We let $O \subsetneq A$ be relatively open in $A$ (with $O=\tilde O\cap A$ for $\tilde O\subset \mathds{R}^d$ open) and $x^*\in O$ be a stable equilibrium of~\eqref{ODE}. By a slight abuse of notation, we say that
\[
\widetilde{\partial O}:=\partial \tilde O \cap A
\]
is the \emph{boundary} of $O$. For $y,z \in A$, we define the following functionals.
\begin{align}
V(x,z,T)&:= \inf_{\phi\in D([0,T];A), \phi(0)=x, \phi(T)=z } I_{T,x}(\phi) \notag \\
V(x,z)&:= \inf_{T>0} V(x,z)  \notag \\
\bar V&:= \inf_{z \in \widetilde{\partial O}} V(x^*,z). \notag
\end{align}
In other words, $\bar V$ is the minimal energy required to leave the domain $O$ when starting from $x^*$.

\begin{assumption} \label{AssExit}
\begin{enumerate}
\item[(D1)]
$x^*$ is the only stable equilibrium point of~\eqref{ODE} in $O$ and the solution $Y^x$ of~\eqref{ODE} with $x=Y^x(0) \in O$ satisfies 
\[
Y^x(t) \in O \text{ for all } t>0 \text{ and } \lim_{t\rightarrow \infty} Y^x(t)=x^*.
\]
\item[(D2)]
For a solution $Y^x$ of~\eqref{ODE} with $x=Y^x(0)\in \widetilde{\partial O}$, we have 
\[
\lim_{t \rightarrow \infty} Y^x(t)=x^*.
\]
\item[(D3)]
$\bar{V} <\infty$.
\item[(D4)]
For all $\rho>0$ there exist constants $T(\rho)$, $\epsilon(\rho)>0$ with $T(\rho), \epsilon(\rho) \downarrow 0$ as $\rho \downarrow 0$ such that for all $z \in \widetilde{\partial O} \cup \{x^*\}$ and all $x,y \in \overline{B(z,\rho)} \cap A$ there exists an 
\[
\phi=\phi(\rho,x,y):[0,T(\rho)] \rightarrow A\quad \text{ with } \phi(0)=x, \phi(T(\rho))=y \text{ and } I_{T(\rho)}(\phi)<\epsilon(\rho).
\]
\item[(D5)]
For all $z\in \widetilde{\partial O}$ there exists an $\eta_0>0$ such that for all $\eta<\eta_0$ there exists a $\tilde z=\tilde z(\eta)\in A\setminus \bar O$ with $|z-\tilde z|>\eta$.
\end{enumerate}
\end{assumption}

Let us shortly comment on Assumption~\ref{AssExit}. By~(D1), $O$ is a subset of the domain of attraction of $x^*$. (D2) is violated by the applications we have in mind: we are interested in situations where $\widetilde{\partial O}$ is the \emph{characteristic boundary} of $O$, i.e., the boundary separating two regions of attraction of equilibria of~\eqref{ODE}. In order to relax this assumption, we shall add an approximation argument in section \ref{subsec_caractBoundary}. By~(D3), it is possible to reach the boundary with finite energy. This assumption is always satisfied for the epidemiological models we consider. For $z=x^*$, (D4) is also always satisfied in our models as the rates $\beta_j$ are bounded from above and away from zero in small neighborhoods of $x^*$; hence, the function $\phi(x,y,\rho)$ can, e.g., be chosen to be linear with speed one (see, e.g., \cite{Shwartz1995} Lemma 5.22).
(D5) allows us to consider a trajectory which crosses the boundary $\widetilde{\partial O}$, in such a way that all paths in a sufficiently small tube around that trajectory do exit $O$.
%\color{red}
%comment on (D5) and about SIRS, see Appendix.
%Here: Proof that this holds for SIRS, if we take 
%\[
%O=O^\eta:=\{x\in A| x_1>\eta\}.
%\] 
%We recall that $R_0>1$ and hence $\beta >\gamma$ (force of infection $>$ recovery). We observe that for small $\phi_1(t)$, ATTENTION: the argument is incorrect
%\[
%\phi_1'(t)=(\beta-\gamma) \phi(t) - \beta \phi(t)^2>0; 
%\]
%hence Assumption~\ref{AssExit}~(i) is satisfied if we take $\eta$ small enough. For $\eta>0$, (ii) is obviously satisfied; (iii), (v) and~(vi) are clear (recall $\tilde A=\{x \in A| x_1=0\}$). Now, in a $\rho$-neighborhood $\widetilde{\partial O}$ ($\rho$ small enough), the rates are bounded from below and above, so (iv) is satisfied. We have to be careful about (iv) in a neighborhood of $x^*$. Maybe we have to relax this in a similar way as in Assumption~\ref{AssMain}.
%\color{black}

We are interested in the following quantity:
\[
\tau^{N,x}:=\tau^N:=\inf\{t>0|Z^{N,x}(t) \not\in O\},
\] 
i.e., the first time that $Z^{N,x}$ exits $O$. 

\subsection{Auxiliary results}

Assumption~\ref{AssExit} (A4) gives the following analogue of Lemma 5.7.8 of~\cite{Dembo2009}.

\begin{lem} \label{LemVCont}
Assume that Assumption~\ref{AssExit} holds. Then for any $\delta>0$, there exists an $\rho_0>0$ such that for all $\rho<\rho_0$,
\[
\sup_{z\in \widetilde{\partial O}\cup x^*, x,y\in \overline{B(z,\rho)}} \inf_{T\in[0,1]} V(x,y,T) < \delta.
\]
\end{lem}

We can recover the analogue of Lemma 5.7.18 of~\cite{Dembo2009} by using Lemma~\ref{LemVCont}.

\begin{lem} \label{LemAbove1}
Assume that Assumptions~\ref{MainAss} and~\ref{AssExit} hold.
Then, for any $\eta>0$ there exists a $\rho_0$ such that for all $\rho<\rho_0$ there exists a $T_0<\infty$ such that
\[
\liminf_{N \rightarrow \infty}\frac 1 N \log \inf_{x\in \overline{B(x^*,\rho)}} \mathbb{P}[\tau^{N,x} \leq T_0 ] >-(\bar{V}+\eta).
\]
\end{lem}

\begin{proof}
We follow the same line of reasoning as in the proof of Lemma 5.7.18 in~\cite{Dembo2009}. Let $x \in \overline{B(x^*,\rho)}$. We use Lemma~\ref{LemVCont} for $\delta = \eta/4$ (and we let $\rho$ be small enough for Lemma~\ref{LemVCont} to hold). We construct a continuous path $\psi^x$ with $\psi^x(0)=x$, $\psi^x(t_x)=x^*$ ($t_x \leq 1$) and $I_{t_x,x}(\psi^x)\leq \eta/4$. We then use Assumption~\ref{AssExit}~(D3). For $T_1< \infty$, we can construct a path $\phi \in C[0,T_1]$ such that $\phi(0)=x^*$, $\phi(T_1)=z \in \widetilde{\partial O}$ and $I_{T_1,0}(\phi)\leq \bar{V} + \eta/4$. Subsequently, we use Lemma~\ref{LemVCont} and obtain a path $\tilde \psi$ with $\tilde \psi(0)=z$, $\tilde \psi(s_x)\not \in O$ ($s \leq 1$), $I_{s,z}(\tilde \psi)\leq \eta/4$ and $\dist(\bar{z},O)=:\Delta>0$.\footnote{Note that the Assumption~(D5) is required here.}  We finally let $\theta^x$ be the solution of the ODE~\eqref{ODE} with $\theta^x(0)=\bar{z}$ on $[0,2-t_x-s]$, consequently $I_{2-t_x-s,\bar{z}}(\theta^x)=0$, see~Lemma~\ref{Lemma5.12-5.14}.

We concatenate the paths $\psi^x$, $\phi$, $\tilde \psi$ and $\theta^x$ and obtain the path $\phi^x\in C[0,T_0]$ ($T_0=T_1+2$ independent of $x$) with $I_{T_0,x}(\phi^x) \leq \bar{V}+\eta/2$.

Finally, we define
\[
\Psi:=\bigcup_{x \in \overline{B(x^*,\rho)}} \big\{ \psi \in D([0,T_0];A)| \|\psi - \phi^x\| < \Delta/2 \big\};
\]
hence $\Psi \subset D([0,T_0];A)$ is open, $(\phi^x)_{x \in \overline{B(x^*,\rho)}} \subset \Psi$ and $\{Z^{N,x} \in \Psi\}\subset \{\tau^{N,x} \leq T_0\}$. We now use 
Theorem~\ref{TheoremLDPLowerUnif}.

\begin{align*}
\liminf_{N \rightarrow \infty} \frac 1 N \log \inf_{x \in \overline{B(x^*,\rho)}} \mathbb{P}[Z^{N,x}\in \Psi]&\geq
-\sup_{x \in \overline{B(x^*,\rho)}}\inf_{\phi\in \Psi} I_{T_0,x}(\phi) \\
&\geq -\sup_{x \in \overline{B(x^*,\rho)}} I_{T_0,x}(\phi^x) \\
&>-(\bar{V}+\eta).
\end{align*}

\end{proof}

We also require the following result (analogue of Lemma 5.7.19 of~\cite{Dembo2009}).

\begin{lem} \label{LemAbove2}
Assume that Assumption~\ref{AssExit} holds. Let $\rho>0$ such that $\overline{B(x^*,\rho)}\subset O$ and
\[
\sigma^{N,x}_\rho:=\inf\{t>0|Z^{N,x}_t \in \overline{B(x^*,\rho)}\text{ or } Z^{N,x}_t \not\in O\}.
\]
Then
\[
\lim_{t\rightarrow\infty} \limsup_{N \rightarrow \infty}\frac 1 N  \log \sup_{x \in O} \mathbb{P}[\sigma_\rho^{N,x}>t]=-\infty.
\]
\end{lem}

\begin{proof}
We adapt the proof of~\cite{Dembo2009} Lemma 5.7.19 to our case.

Note first that for $x \in \overline{B(x^*,\rho)}$, $\sigma_\rho^{N,x}=0$; we hence assume from now on that $x\notin \overline{B(x^*,\rho)}$. For $t>0$, we define the closed set $\Psi_t\subset D([0,t];A)$,
\[
\Psi_t:=\{\phi \in D([0,t];A) | \phi(s) \in \overline{O \setminus B(x^*,\rho)} \text{ for all } s \in [0,t] \};
\]
hence for all $x,N$,
\begin{equation} \notag
\{\sigma_\rho^{N,x} >t\} \subset \{Z^{N,x} \in \Psi_t\}.
\end{equation}
By Theorem~\ref{ThUBunif}, this implies for all $t>0$,
\begin{align*}
\limsup_{N \rightarrow \infty}\frac 1 N \log \sup_{x \in \overline{O \setminus B(x^*,\rho)}} \mathbb{P}[\sigma_\rho^{N,x}>t] 
&\leq \limsup_{N \rightarrow \infty}\frac 1 N  \log \sup_{x \in \overline{O \setminus B(x^*,\rho)}} \mathbb{P}[Z^{\epsilon,x} \in \Psi_t] \\
&\leq - \inf_{\phi \in \Psi_t} I_{t,\phi(0)} (\phi).
\end{align*}
It hence suffices to show that
\begin{equation} \label{EqLemSigma1}
\lim_{t\rightarrow \infty} \inf_{\phi \in \Psi_t} I_{t,\phi(0)} (\phi) =\infty.
\end{equation}

To this end, consider $x \in \overline{O \setminus B(x^*,\rho)}$ and recall that $Y^x$ is the solution of~\eqref{ODE} (on $[0,t]$ for all $t>0$). By Assumption~\ref{AssExit}~(D2), there exists a $T_x<\infty$ such that $Y^x(T_x) \in \overline{B(x^*,3 \rho)}$. We have (here $B$ denotes the Lipschitz constant of $b$),
\[
|\phi^x(t)-\phi^y(t)| \leq |x-y| + \int_0^t |b(\phi^x(s)) - b(\phi^y(s)) | ds \leq +|x-y| + \int_0^t B |\phi^x(s) - \phi^y(s) | ds
\]
and therefore by Gronwall's inequality
$
|Y^x(T_x)-Y^y(T_x)| \leq |x-y|e^{T_x B};
$
consequently, there exists a neighborhood $W_x$ of $x$ such that for all $y \in W_x$, $Y^y(T_x) \in \overline{B(x^*,3 \rho)}$. By the compactness of $\overline{O \setminus B(x^*,\rho)}$, there exists a finite open subcover $\cup_{i=1}^k W_{x_i} \supset \overline{O \setminus B(x^*,\rho)}$; for $T:=\max_{i=1,\dots,k} T_{x_i}$ and $y \in \overline{O \setminus B(x^*,\rho)}$ this implies that $Y^y(s) \in \overline{B(x^*,2/3 \rho)}$ for some $s \leq T$.

Assume now that~\eqref{EqLemSigma1} is false. Then there exits an $M< \infty$ such that for all $n \in \mathds{N}$ there exists an $\phi_n \in \Psi_{nT}$ with $I_{nT} (\phi_n) \leq M$. The function $\phi_n$ is concatenated by functions $\phi_{n,k} \in \Psi_T$ and we obtain
\[
M \geq I_{nT} (\phi_n) = \sum_{k=1}^n I_{T} (\phi_{n,k}) \geq n \min_{k=1,\dots,n} I_{T} (\phi_{n,k}).
\]
Hence there exists a sequence $(\psi_k)_k \subset \Psi_T$ with $\lim_{k \rightarrow \infty} I_{T} (\psi_k) =0$. Note now that the set
\[
\phi(t):=\{
\phi \in C[0,T]| I_{T,\phi(0)} (\phi) \leq 1, \phi(s) \in \overline{O \setminus B(x^*,\rho)} \text{ for all } s \in [0,T] \} \subset \Psi_T
\]
is compact (as a subset of $(C[0,T],\|\cdot \|_\infty)$); hence there exists a subsequence $(\psi_{k_l})_l$ of $(\psi_k)_k$ such that $\lim_{l\rightarrow \infty} \psi_{k_l}=:\psi^* \in \phi(t)$ in $(C[0,T],\|\cdot \|_\infty)$. By the lower semi-continuity of $I_T$ (cf.~Lemma~\ref{Lemma5.42}) this implies
\[
0=\liminf_{l \rightarrow \infty} I_{T} (\psi_{n_l}) \geq I_{T}(\psi^*),
\]
which in turn implies that $\psi^*$ solves~\eqref{ODE} for $x=\psi^*(0)$. But then, $\psi^*(s) \in \overline{B(x^*,2/3 \rho)}$ for some $s \leq T$, a contradiction to $\psi^* \in \Psi_T$.
\end{proof}

The following lemma is the analogue of~\cite{Dembo2009} Lemma 5.7.21.

\begin{lem} \label{LemLower1}
Assume that Assumptions~\ref{MainAss} and~\ref{AssExit} hold.
Let $C\subset A\setminus O$ be closed. Then
\[
\lim_{\rho \rightarrow 0} \limsup_{N \rightarrow \infty} \frac 1 N \log \sup_{x \in \overline{B(x^*,3 \rho)\setminus B(x^*,2 \rho)}} \mathbb{P}[Z^{N,x}_{\sigma_\rho} \in C ] \leq - \inf_{z \in C} V(x^*,z).
\]
\end{lem}

\begin{proof}
We adapt the proof of~\cite{Dembo2009} Lemma 5.7.21 to our situation.
We can assume without loss of generality that $\inf_{z \in C} V(x^*,z)>0$ (else the assertion is trivial). For $\inf_{z \in C} V(x^*,z)>\delta >0$, we define
\[
V_C^\delta:=(\inf_{z \in C} V(x^*,z)-\delta)\wedge 1/\delta >0.
\]
By Lemma~\ref{LemVCont}, there exists a $\rho_0=\rho_0(\delta)>0$ such that for all $0<\rho<\rho_0$,
\[
\sup_{y \in \overline{B(x^*,3 \rho)\setminus B(x^*,2 \rho)}} V(x^*,y) < \delta;
\]
hence
\begin{equation} \label{Eq1LemLower1}
\inf_{y \in \overline{B(x^*,3 \rho)\setminus B(x^*,2 \rho)}, \,\, z \in C} V(y,z) \geq \inf_{z \in C} V(x^*,z) - \sup_{y \in \overline{B(x^*,3 \rho)\setminus B(x^*,2 \rho)}} V(x^*,y) > V_C^\delta.
\end{equation}

For $T>0$, we define the closed set $\Phi^T\subset D([0,T];A)$ by
\[
\Phi^T:=\Phi:=\{\phi \in D([0,T];A)|\phi(t)\in C \text{ for some } t \in [0,T] \}.
\]
We then have for $y \in \overline{B(x^*,3 \rho)\setminus B(x^*,2 \rho)}$,
\begin{equation} \label{Eq2LemLower1}
\mathbb{P}[Z^{N,y}_{\sigma_\rho} \in C] \leq \mathbb{P} [\sigma_{\rho}^{N,y} > T] + \mathbb{P} [ Z^{N,y} \in \Phi^T].
\end{equation}
In the following, we bound the two parts in Inequality~\eqref{Eq2LemLower1} from above.

For the second part, we note first that (cf.~Inequality~\eqref{Eq1LemLower1})
\[
\inf_{y \in \overline{B(x^*,3 \rho)\setminus B(x^*,2 \rho)},\,\, \phi \in \Phi^T} I_{T,y}(\phi) \geq \inf_{y \in \overline{B(x^*,3 \rho)\setminus B(x^*,2 \rho)},\,\, z \in C} V(y,z) > V_C^\delta;
\]
hence, we obtain by  Theorem~\ref{ThUBunif}
\begin{equation} \label{Eq3LemLower1}
\begin{split}
\limsup_{N \rightarrow \infty} \frac 1 N \log \sup_{y \in \overline{B(x^*,3 \rho)\setminus B(x^*,2 \rho)}} \mathbb{P} [Z^{N,y} \in \Phi^T] 
&\leq - \inf_{y \in \overline{B(x^*,3 \rho)\setminus B(x^*,2 \rho)},\,\,\phi \in \Phi^T} I_{T,y}(\phi) \\
&< -V_C^\delta. 
\end{split}
\end{equation}

For the first part in Inequality~\eqref{Eq2LemLower1}, we use Lemma~\ref{LemAbove2}: There exists a $0<T_0<\infty$ such that for all $T\geq T_0$
\begin{equation} \label{Eq4LemLower1}
\limsup_{N \rightarrow \infty} \frac 1 N \log \sup_{y \in \overline{B(x^*,3 \rho)\setminus B(x^*,2 \rho)}} \mathbb{P}[\sigma^{N,y}>T]< -V_C^\delta.
\end{equation}

We let $T\geq T_0$ and $\rho<\rho_0$ and combine Inequalities~\eqref{Eq2LemLower1}, \eqref{Eq3LemLower1} and~\eqref{Eq4LemLower1}. Hence there exists an $N_0>0$ such that for all $N>N_0$,
\begin{align*}
& \frac 1 N \log \sup_{y \in \overline{B(x^*,3 \rho)\setminus B(x^*,2 \rho)}} \mathbb{P}[Z^{N,y}_{\sigma_\rho} \in C] \\*
& \qquad  \leq \frac 1 N \log \Big( \sup_{y \in \overline{B(x^*,3 \rho)\setminus B(x^*,2 \rho)}} \mathbb{P} [\sigma_{\rho}^{N,y} > T]  + \sup_{y \in \overline{B(x^*,3 \rho)\setminus B(x^*,2 \rho)}} \mathbb{P} [ Z^{N,y} \in \Phi^T] \Big) \\ 
&\qquad < \frac 1 N \log\big(2 e^{-NV_C^\delta} \big) =\frac 1 N \log 2 - V_C^\delta;
\end{align*}
and
\[
\limsup_{N \rightarrow \infty} \frac 1 N \log \sup_{y \in \overline{B(x^*,3 \rho)\setminus B(x^*,2 \rho)}} \mathbb{P}[Z^{N,x}_{\sigma_\rho} \in C ] \leq -V_C^\delta.
\]
Taking the limit $\delta \rightarrow 0$ finishes the proof.
\end{proof}

The next lemma is the analogue of Lemma 5.7.22 of~\cite{Dembo2009}.

\begin{lem} \label{LemLower2}
Assume that Assumption~\ref{AssExit} holds.
Then, for all $\rho>0$ such that $\overline{B(x^*,\rho)}\subset O$ and for all $x  \in O$,
\[
\lim_{N \rightarrow \infty} \mathbb{P}[ Z^{N,x}_{\sigma_\rho} \in \overline{B(x^*,\rho)}] =1.
\]
\end{lem}

\begin{proof}
Let $x\in O \setminus \overline{B(x^*,\rho)}$ (the case $x \in\overline{B(x^*,\rho)}$ is clear). Let furthermore $T:=\inf\{ t\geq 0 |\phi(t) \in B(x^*,\rho/2)\}$. Since $Y^x$ is continuous and never reaches $\widetilde{\partial O}$ (Assumption~\ref{AssExit}~(D1)), we have $\inf_{t\geq 0} \dist(Y^x(t),\widetilde{\partial O}) =:\Delta >0$. Hence we have the following implication:
\[
\sup_{t \in [0,T]} | Z^{N,x}_t - Y^x(t)|\leq \frac{\Delta}{2} \Rightarrow Z^{N,x}_{\sigma_\rho} \in \overline{B(x^*,\rho)}.
\]
In other words,
\begin{equation}
\mathbb{P}[ Z^{N,x}_{\sigma_\rho} \notin \overline{B(x^*,\rho)}]
\leq \mathbb{P}\Big[\sup_{t \in [0,T]} | Z^{N,x}_t - Y^x(t)| > \frac{\Delta}{2}\Big]. \label{ConvergenceA.s.}
\end{equation}
The right hand side of Inequality~\eqref{ConvergenceA.s.} converges to zero as $N \rightarrow \infty$ by~Theorem~\ref{ThLLN}.
\end{proof}

The next lemma is the analogue of~\cite{Dembo2009} Lemma 5.7.23.

\begin{lem} \label{LemLower3}
Assume that Assumption~\ref{AssExit} holds.
Then, for all $\rho,c>0$, there exists a constant $T=T(c,\rho)<\infty$ such that
\[
\limsup_{N \rightarrow \infty} \frac 1 N \log \sup_{x \in O} \mathbb{P}[\sup_{t\in [0,T]} |Z^{N,x}_t - x| \geq \rho] < -c.
\]
\end{lem}

\begin{proof}
Let $\rho,c>0$ be fixed. For $T,N>0$ and $x\in O$ we have
\begin{align}
\mathbb{P}[\sup_{t\in[0,T]}|Z^{N,x}_t - x|\geq \rho] & =\mathbb{P}\Big[ \sup_{t\in[0,T]} \frac 1 N |\sum_j h_j P_j \Big(N\int_0^t \beta_j(Z^{N,x}_s) ds \Big)| \geq \rho \Big] \notag \\
& \leq \mathbb{P} \Big[  \sum_j P_j (N \bar \beta T) \geq  N \rho \bar h^{-1}\Big] \notag \\
& \leq k \mathbb{P} \Big[ P (N \bar \beta T) \geq  N \rho \bar h^{-1} k^{-1} \Big] \label{EqStirling1}
\end{align}
for a standard Poisson process $P$. We now let, with $c_1(T)=\bar{\beta}T$ and $c_2=\rho\bar{h}^{-1}k^{-1}$,
\begin{equation} \label{EqStirling0}
T<T_0:=\frac{e^{-1} c_2}{2 \bar \beta} \wedge \frac{e^{-c/c_2-1} c_2}{\bar \beta} \quad\text{and} \quad  N> N_0:=1/c_2\wedge \frac{\log 2 k}{c_1(T)}. 
\end{equation} 
We then obtain (note that $N c_2>1$ and $\frac{e}{c_2} c_1(T)<1/2$ by~\eqref{EqStirling0})
\begin{align}
k \mathbb{P} \Big[ P (N c_1(T)) \geq N c_2 \Big]&
= k e^{-N c_1(T)} \sum_{m \geq N c_2} \frac{N^{m} c_1(T)^m}{m!} \notag \\
&<k e^{-N c_1(T)} \sum_{m \geq N c_2} \frac{\big(e N\big)^{m} c_1(T)^m}{m^m \sqrt{2 \pi m}} \label{EqStirling2} \\
&\leq \frac{1}{2} \sum_{m\geq N c_2} \frac{\big(e N\big)^{m} c_1(T)^m}{\big(N c_2\big)^{m} } \notag \\
&\leq  \frac{1}{2} \frac{\big(\frac{e}{c_2} c_1(T)\big)^{N c_2}}{1-\tfrac{e}{c_2} c_1(T) } \notag \\
&\leq  \big(\frac{e}{c_2} c_1(T)\big)^{N c_2}; \label{EqStirling3} 
\end{align}
here we applied Stirling's formula, $m!>\sqrt{2 \pi m} (m/e)^m$, in Inequality~\eqref{EqStirling2}. Finally, we have
\begin{equation} \label{EqStirling4}
 \big(\frac{e}{c_2} c_1(T)\big)^{N c_2}
 = \Big(\big(\frac{e}{c_2} c_1(T)\big)^{-c_2} \Big)^{-N}
 <(e^c)^{-N}=e^{-N c}
\end{equation}
by~\eqref{EqStirling0}. The assertion now follows by combining the Inequalities~\eqref{EqStirling1}, \eqref{EqStirling3} and~\eqref{EqStirling4}.
\end{proof}

\subsection{Main results}

We can now deduce the analogue of~\cite{Dembo2009} Theorem 5.7.11 (a). the proof of~\cite{Dembo2009} carries over.

\begin{theorem} \label{TheoExitTime}
Assume that Assumption~\ref{AssExit} holds. Then, for all $x \in O\cap AN$ and $\delta>0$,
\[
\lim_{N\rightarrow\infty} \mathbb{P}\big[e^{(\bar{V} - \delta)N}<\tau^{N,x}<e^{(\bar{V} + \delta)N}\big]=1.
\]
Moreover, for all $x\in O$, as $N\to\infty$,
\[  \frac{1}{N}\log\E(\tau^{N,x})\to\bar{V}.\]
\end{theorem}

\begin{proof}
\underline{Upper bound of exit time:}

We fix $\delta>0$ and apply Lemma~\ref{LemAbove1} with $\eta:=\delta/4$. Hence, for $\rho<\rho_0$ there exists a $T_0<\infty$ and an $N_0>0$ such that for $N>N_0$,
\[
\inf_{x \in \overline{B(x^*,\rho)}} \mathbb{P}[\tau^{N,x} \leq T_0] > e^{-N(\bar{V}+\eta)}.
\]
Furthermore, by Lemma~\ref{LemAbove2} there exists a $T_1<\infty$ and $N_1>0$ such that for all $N>N_1$,
\[
\inf_{x \in O} \mathbb{P}[\sigma_\rho^{N,x} \leq T_1] > 1- e^{-2N\eta}.
\]
For $T:=T_0+T_1$ and $N>N_0 \vee N_1 \vee 1/\eta$, we hence obtain
\begin{align}
q^N:=q &:= \inf_{x \in O}\mathbb{P} [\tau^{N,x} \leq T] \notag \\
&\geq \inf_{x \in O}\mathbb{P}[\sigma_\rho^{N,x} \leq T_1] \inf_{y \in \overline{B(x^*,\rho)}} \mathbb{P}[\tau^{N,y} \leq T_0] \notag \\
&> (1-e^{-2 N\eta}) e^{-N(\bar{V}+\eta)} \notag \\
&\geq e^{-N(\bar{V}+2\eta)}. \label{EqUpperBound1}
\end{align} 
This yields for $k\in \mathds{N}$
\begin{align*}
\mathbb{P}[\tau^{N,x} >(k+1) T] & = \big(1-\mathbb{P}[\tau^{N,x} \leq (k+1) T | \tau^{N,x} > kT]\big)\mathbb{P}[\tau^{N,x} >kT] \\
& \leq (1-q) \mathbb{P}[\tau^{N,x} >kT]
\end{align*}
and hence inductively
\[
\sup_{x\in O} \mathbb{P}[\tau^{N,x} >kT] \leq (1-q)^k. 
\]
This implies, exploiting \eqref{EqUpperBound1} for the last inequality
\begin{equation} \label{EqUpperBound2}
\sup_{x \in O} \mathbb{E}[\tau^{N,x}] \leq T \big( 1+ \sum_{k=1}^{\infty} \sup_{x\in O} \mathbb{P} [\tau^{N,x} > kT] \big) \leq T  \sum_{k=0}^{\infty} (1-q)^k=\frac{T}{q} \leq
T e^{N(\bar{V}+ 2\eta)};
\end{equation}
by Chebychev's Inequality we obtain
\[
\mathbb{P}[\tau^{N,x} \geq e^{N(\bar{V}+ \delta)}] \leq  e^{-N(\bar{V}+ \delta)} \mathbb{E}[\tau^{N,x} ] \leq T e^{-\delta N/2} 
\]
which approaches zero as $N \rightarrow \infty$ as required.

\underline{Lower bound of exit time:}

For $\rho>0$ such that $\overline{B(x^*,3 \rho)} \subset O$, we define recursively $\theta_0:=0$ and for $m\in \mathds{N}_0$,
\begin{align*}
\tau_m^x:=\tau_m&:= \inf \{ t \geq \theta_m^x|Z^{N,x}_t \in \overline{B(x^*,\rho)} \text{ or } Z^{N,x}_t \not\in O \}, \\
\theta_{m+1}^x:=\theta_{m+1}&:= \inf \{ t \geq \tau_m^x|Z^{N,x}_t \in \overline{B(x^*,3 \rho) \setminus B(x^*,2 \rho)} \},
\end{align*}
with the convention $\theta_{m+1}:=\infty$ if $Z^{N}_{\tau_m} \not\in O$. Note that we have $\tau^{N,x}=\tau^x_m$ for some $m \in \mathds{N}_0$.

For fixed $T_0>0$ and $k \in \mathds{N}$ we have the following implication: If for all $m=0,\dots,k$, $\tau_m\not=\tau^N$ and for all $m=1\dots,k$, $\tau_m - \tau_{m-1} > T_0$, then
\[
\tau^N > \tau _k=\sum_{m=1}^k (\tau_m - \tau_{m-1}) > k T_0.
\]
In particular, we have for $k:= \lfloor T_0^{-1} e^{N(\bar{V} - \delta)} \rfloor +1$ (note that $\theta_m - \tau_{m-1} \leq \tau_m - \tau_{m-1}$),
\begin{align}
\mathbb{P}[\tau^{N,x}\leq e^{N(\bar{V} - \delta)}] &\leq \mathbb{P}[\tau^{N,x}\leq k T_0] \notag \\
&\leq \sum_{m=0}^k \mathbb{P}[\tau^{N,x}=\tau_m^x] + \sum_{m=1}^k \mathbb{P}[\theta_m^x - \tau_{m-1}^x \leq T_0] \notag \\
&= \mathbb{P}[\tau^{N,x}=\tau_0^x] 
+ \sum_{m=1}^k \mathbb{P}[\tau^{N,x}=\tau_m^x] + \sum_{m=1}^k \mathbb{P}[\theta_m^x - \tau_{m-1}^x \leq T_0]. \label{Eq1Lower}
\end{align}
In the following, we bound the three parts in~\eqref{Eq1Lower} from above. To this end, we assume $\bar{V}>0$ for now. The simpler case $\bar{V}=0$ is treated below.

For the first part, we have
\begin{equation} \label{EqLowerSumm1}
\mathbb{P}[\tau^{N,x}=\tau_0^x] = \mathbb{P}[Z^{N,x}_{\sigma_\rho} \not\in O]. 
\end{equation}
For the second part, we use the fact that $Z^{N,x}$ is a strong Markov process and that the $\tau_m$'s are stopping times. We obtain for $m\geq 1$ and $x \in O$,
\begin{equation} \label{EqLowerSumm2}
\mathbb{P}[\tau^{N,x}=\tau_m^x] \leq  \sup_{y \in \overline{B(x^*,3 \rho) \setminus B(x^*,2 \rho)}} \mathbb{P}[Z^{N,y}_{\sigma_\rho} \not\in O]. 
\end{equation}
Similarly, we obtain for the third part for $m \geq 1$ and $x \in O$,
\begin{equation}\label{EqLowerSumm3}
\mathbb{P}[\theta_m^x - \tau_{m-1}^x \leq T_0] \leq \sup_{y\in O} \mathbb{P}[\sup_{t \in [0,T_0]} |Z_t^{N,y} - y| \geq \rho].
\end{equation}

The upper bounds in~\eqref{EqLowerSumm2} and~\eqref{EqLowerSumm3} can now be bounded by using the Lemma~\ref{LemLower1} and~\ref{LemLower3}, respectively. We fix $\delta>0$. By~Lemma~\ref{LemLower1} (for $C=A\setminus O$), there exists a $\rho=\rho(\delta)>0$ and an $N_1=N_1(\rho,\delta)>0$ such that for all $N> N_1$,
\begin{equation} \label{EqLowerSumm22}
\sup_{y \in \overline{B(x^*,3 \rho) \setminus B(x^*,2 \rho)}} \mathbb{P}[Z^{N,y}_{\sigma_\rho} \not\in O] \leq \exp\big(-N(\bar{V}-\delta/2)\big).
\end{equation}
By Lemma~\ref{LemLower3} (for $\rho=\rho(\delta)$ as above and $c=\bar{V}$), there exists a constant $T_0=T(\rho,\bar{V})< \infty$ and an $N_2=N_2(\rho,\delta)>0$ such that for all $N> N_2$,
\begin{equation} \label{EqLowerSumm32}
\sup_{y \in O}\mathbb{P}[\sup_{t \in [0,T_0]} |Z_t^{N,y} - y| \geq \rho]
\leq \exp\big(-N(\bar{V}-\delta/2)\big).
\end{equation}

We now let $N >N_1 \vee N_2$ (and large enough for $T_0^{-1} \exp\big(N(\bar{V}-\delta)\big)>1$ for the specific $T_0$ above). Then by Inequality~\eqref{Eq1Lower},
\begin{align}
\mathbb{P}[\tau^{N,x}\leq e^{N(\bar{V} - \delta)}] & \overset{\eqref{EqLowerSumm1},\eqref{EqLowerSumm2},\eqref{EqLowerSumm3}}{\leq} \mathbb{P}[Z^{N,x}_{\sigma_\rho} \not\in O]+ k \sup_{y \in \overline{B(x^*,3 \rho) \setminus B(x^*,2 \rho)}} \mathbb{P}[Z^{N,y}_{\sigma_\rho} \not\in O] \notag \\*
& \qquad + k \sup_{y\in O} \mathbb{P}[\sup_{t \in [0,T_0]} |Z_t^{N,y} - y| \geq \rho] \notag \\
&\overset{\eqref{EqLowerSumm22},\eqref{EqLowerSumm32}}{\leq} \mathbb{P}[Z^{N,x}_{\sigma_\rho} \not\in O] + 4 T_0^{-1} \exp\big(-N\delta/2\big). \label{EqLowerbound}
\end{align}
The right-hand side of Inequality~\eqref{EqLowerbound} tends to zero as $\epsilon \rightarrow 0$ by Lemma~\ref{LemLower2}, finishing the proof for $\bar{V}>0$.

Finally, let us assume that $\bar{V}=0$ and that the assertion is false for a given $x \in O$. Then there exists a $\mu_0\in(0,1/2)$ and a $\delta_0>0$ such that for all $\bar{N}>0$ there exists an $N>\bar{N}$ with
\[
\mu_0 \leq \mathbb{P} [\tau^{N,x} \leq e^{- N\delta_0}].
\]
We fix $\rho>0$ such that $\overline{B(x^*,2 \rho)} \subset O$. Using the strong Markov property of $Z$ and the fact that $\sigma_{\rho}$ is a stopping time again, we have that for all $\bar{N}>0$ there exists an $N>\bar{N}$ with
\begin{align}
\mu_0 &\leq \mathbb{P} [\tau^{N,x} \leq e^{- N\delta_0}] \notag \\
&\leq \mathbb{P}[Z_{\sigma_\rho}^{N,x} \notin \overline{B(x^*,\rho)}]
+ \sup_{y \in O} \mathbb{P}[\sup_{t \in [0, e^{-N\delta_0}]} |Z_t^{N,y}-y|\geq \rho]. \label{EqContradiction}
\end{align}
By Lemma~\ref{LemLower2}, there exists an $N_0$ such that for all $N>N_0$,
\begin{equation} \label{EqContra1}
\mathbb{P}[Z_{\sigma_\rho}^{N,x} \not\in \overline{B(x^*,\rho)}] \leq \frac{\mu_0}{2}.
\end{equation}
We now set $c:=-2 \epsilon_0 \log\frac{\mu_0}{2}$. Then by Lemma~\ref{LemLower3}, there exists a $T=T(c,\rho)>0$ and an $N_1>N_0$ such that for all $N>N_1$,
\begin{equation} \label{EqContra2}
e^{-N\delta_0} < T
\end{equation}
and
\begin{equation} \label{EqContra3}
\sup_{y \in O} \mathbb{P}[\sup_{t \in [0, T]} |Z_t^{N,y}-y|\geq \rho]
\leq e^{-Nc/2} < \frac{\mu_0}{2}.
\end{equation}
Combining Inequalities~\eqref{EqContra1}, \eqref{EqContra2} and~\eqref{EqContra3} yields a contradiction to Inequality~\eqref{EqContradiction}, finishing the proof.

\underline{Expected exit time:} 

We have shown in particular that
$\P(\tau^{N,x}>e^{(\bar{V}-\delta)N})\to1$ as $N\to\infty$. Consequently, from Chebycheff,
\begin{align*}
\E(\tau^{N,x})&\ge e^{(\bar{V}-\delta)N}\P(\tau^{N,x}>e^{(\bar{V}-\delta)N}),\\
\liminf_{N\to\infty}\frac{1}{N}\log\E(\tau^{N,x})&\ge\bar{V}-\delta
\end{align*}
for all $\delta>0$. This together with \eqref{EqUpperBound2} implies the second statement of the Theorem.
\end{proof}
\subsection{The case of a characteristic boundary}\label{subsec_caractBoundary}
Since we are mainly interested in studying the time of exit form the basin of attraction of one local equilibrium to 
that of another, we need to consider situations which do not satisfy the above assumptions.
More precisely, we want to suppress the Assumptions (D2), and keep Assumptions (D1), (D3), (D4) and (D5). 
We assume that there exists a collection of open sets $\{O_\rho,\ \rho>0\}$ which is such that 
\begin{description}
\item{$\bullet$} $\overline O_\rho\subset O$ for any $\rho>0$.
\item{$\bullet$} $d(O_\rho,\tilde{\partial O})\to0$, as $\rho\to0$.
\item{$\bullet$} $O_\rho$ satisfies Assumptions (D1), (D2), (D3), (D4) and (D5) for any $\rho>0$.
\end{description}

We can now establish
\begin{corollary}\label{cor-exit}
Let then $O$ be a domain satisfying Assumptions (D1), (D3), (D4) and (D5), such that there exists a sequence 
$\{O_\rho,\ \rho>0\}$ satisfying the three above conditions. Then the conclusion of Theorem \ref{TheoExitTime}
is still true.
\end{corollary}
\begin{proof}
If we define $\bar V_\rho$ as $\bar V$, but with $O$ replaced by $O_\rho$, it follows from Lemma \ref{LemVCont}
that $\bar V_\rho\to\bar V$ as $\rho\to0$. By an obvious monotonicity property and the continuity of the quasi--potential, the lower bound 
\[ \lim_{N\to\infty} \mathbb{P}\big[\tau^{N,x}>e^{(\bar{V} - \delta)N}\big]=1\]
follows immediately from Theorem \ref{TheoExitTime}. The proof of the upper bound is done as in the proof of Theorem \ref{TheoExitTime}, once \eqref{EqUpperBound1} is established. Let us now explain how this is done.
Let $\tau^{N,x}_\rho$ denote the time of exit from $O_\rho$. 
The same argument used to establish \eqref{EqUpperBound1} above leads to the statement that for any $\rho,\eta>0$, there exists $N_{\rho, \eta}$ such that for all $N\ge N_{\rho, \eta}$,
\[ \inf_{x \in O}\mathbb{P} [\tau^{N,x} \leq T]\geq e^{-N(\bar{V}+\eta)}.\]
Now utilizing (D4), (D5) and the compactness of $\bar{O}\backslash O_\rho$, it is not hard to deduce from 
Theorem \ref{TheoremLDPLowerUnif} that for $\rho >0$ small enough,
\[\liminf_{N\to\infty}\log\inf_{x\in(\bar{O}\backslash O_\rho)\cap A^N}\P_x(\tau^{N,x}\le1)>-\eta.\]
The wished result follows now from the last two lower bounds and the strong Markov property.

Finally the result for $\E(\tau^{N,x})$ now follows from the first part of the result, exactly as in the proof of
Theorem \ref{TheoExitTime}. 
\end{proof}

\section{Example: the SIRS model} \label{AppendixSIRS}

We finally show that Theorem \ref{TheoExitTime} applies to the SIRS model
from Example \ref{ExampleSIRS}. Assumptions~(A) and~(B) have already been verified in Section~\ref{SectionSetup}. For~(C), we note that major problems only occur for the balls centered at the ``corner points'' of the set $A$. Only for a corner point $x$ (with corresponding vector $v$) do we have $v \not \in \mathcal{C}_x$ (recall that we define $\mathcal C_x$ corresponding to the modified rates $\beta^\delta$). For simplicity of exposition, we concentrate on the ball $B$ centered at $x=(1,0)^\top$. The same argument applies to the balls centered at the other corners and in a simpler form to all other balls. For the balls $B_i$ not centered at the corners, the vectors $v_i$ can be represented by $\mu^i$'s for which $\mu^i_j>0$ implies that $\beta_j^\delta(z)>\lambda>0$ (for an appropriate constant $\lambda$ which can be chosen independently of $i$) for all $z\in B_i$. This simplifies the discussion below significantly. In particular, Assumption~(C) is satisfied due to Theorem~\ref{ThLLN}. We first note that for all $x \in A$, $y \in \mathds{R}^d$, $L(x,y)=\tilde L(x,y)$, cf.~Theorem~\ref{Theorem5.26} below. As before, we define the vector 
\[
v=(-1/2,1/4)^\top \quad \text{and} \quad \mu_1=0,\quad \mu_2=\frac 1 2, \quad \mu_3 = \frac 1 4,
\] 
in particular $\mu \in \tilde V_v$ (but $V_{x,v}=\emptyset$!). In order to simplify the notation, we do not normalize $v$. We let $\eta <\eta_0:=1/2$ and note that for $t \in [0,\eta]$,
\[
\beta_2(\phi^x(t))=\gamma \Big(1-\frac t 2 \Big) \geq \frac 3 4 \gamma, \quad \beta_3(\phi^x(t))=\frac \nu 4 t.
\] 
%We readily observe that for appropriate constants $\tilde C_1,\tilde C_2>0$ and $t \in [0, \eta]$,
%\begin{equation}
%\label{EqSIRSC1}
% \ell ( \phi^x(t),\mu)\leq \tilde C_1 + \tilde C_2 |\log t|
%\end{equation}
%and therefore
%\[
%I_{\eta,x}(\phi^x) \leq \int_0^\eta  \ell (\phi^x(t),\mu)dt \rightarrow 0 \quad \text{as } \eta \rightarrow 0.
%\]
%The same argument applies for all $z\in B\cap A$. We note that the bound in~\eqref{EqSIRSC1} can be %chosen independently of $z\in B$. 
%This shows that Assumption~(C1) is satisfied.

Let us prove that Assumption~(C) is satisfied. Let $X^N$ be a Poisson random variable with mean $\mu N$. We note that by Theorem~\ref{ThLLN} for $\xi >1$,
\begin{equation} \label{EqPoissonBound}
\mathbb{P}[  X^N > \xi N \mu] \leq \bar C_1 \exp\big(-N \bar C_2(\xi)), \quad \mathbb{P}[  X^N < \xi^{-1} N \mu] \leq \tilde C_1 \exp\big(-N \tilde C_2(\xi))
\end{equation}
for appropriate constants $\bar C_1$, $\tilde C_1$, $\bar C_2$, $\tilde C_2$ with $\bar C_2(\xi)=O((\xi -1)^2)$
%\footnote{Here we use the following notation:
%\[
%f \in \Theta(g) \Leftrightarrow f \in \mathcal O(g) \text{ and } f \in o(g).
%\]
%}
 as $\xi \downarrow 1$ and $\tilde C_2(\xi)= O((1-\xi^{-1})^2)$ as $\xi \downarrow 1$. The first bound 
 is obtained by applying Theorem~\ref{ThLLN} to $d=k=h_1=1$, $\beta_1\equiv \mu$ and $x=0$. We get
\begin{align*}
\mathbb{P} [  X^N > \xi N \mu ] &=\mathbb{P}\Big[ \frac 1 N  X^N > \xi \mu\Big] \\
&=\mathbb{P}\Big[ \frac 1 N X^N - \mu > (\xi-1) \mu\Big ] \\
&\leq \mathbb{P}\Big[\Big|\frac 1 N X^N - \mu  \Big| > (\xi - 1) \mu  \Big] \\
&=\mathbb{P}\Big[|Z^{N,0}(1) - \mu | > (\xi - 1) \mu  \Big] \\
&\leq \mathbb{P}\Big[\sup_{t\in [0,1]} |Z^{N,0}(t) - \mu t| > (\xi - 1) \mu  \Big].
\end{align*}

Let us define the process $\hat Z^{N,x}$ as the solution of~\eqref{EqPoisson} with constant rates $\mu_j$. For $\epsilon>0$ small enough, we define
\[
X^{N,\epsilon}_2:=\# \text{jumps of type } h_2 \text{ of } \tilde Z^{N,x} \text{ in } [0,\epsilon], \quad
X^{N,\epsilon}_3:=\# \text{jumps of type } h_3 \text{ of } \tilde Z^{N,x} \text{ in } [0,\epsilon]
\]
and
\[
F^{N,\epsilon}_2:=\Big\{\frac{15}{32} N \epsilon < X_2^{N,\epsilon} < \frac{17}{32} N \epsilon\Big\}, \quad
F^{N,\epsilon}_3:=\Big\{\frac{7}{32} N \epsilon < X_3^{N,\epsilon} < \frac{9}{32} N \epsilon\Big\}.
\]
hence
\[
\tilde Z^{N,x}(\epsilon) \in \tilde B :=\Big\{ z \in A \Big|1-\frac{17}{32}\epsilon < z_1 < 1-\frac{15}{32}\epsilon,\; \frac{3}{16}\epsilon < z_2 < \frac{5}{16}\epsilon \Big\} 
\]
and
\[\sup_{t \in [0,\epsilon]}|\tilde Z^{N,x}(t)-\phi^x(t)|<\epsilon
\]
on $F_2^{N, \epsilon} \cap F_3^{N,\epsilon}$.
Furthermore, for $z \in \tilde B$ and $t\in [0, \eta-\epsilon]$,
\[
\dist(\phi^z(t),\partial A) \geq \frac{3}{16} \epsilon
\]
and
\[
|\tilde Z^{N,z}(t)-\phi^z(t)| < \frac{3}{16} \epsilon \Rightarrow |\tilde Z^{N,z}(t)-\phi^{\tilde x}(t)| < \epsilon,
\]
where $\tilde x =\phi^x(\epsilon)=(1-\epsilon/2,\epsilon/4)^\top$.
We compute by using the Markov property of $Z^N$,
\begin{align}
&\mathbb{P} \Big[\sup_{t \in [0, \eta]} |\tilde Z^{N,x}(t) - \phi^x(t) |< \epsilon] \notag \\*
&\qquad \geq \mathbb{P} \Big[\sup_{t \in [0, \eta]} |\tilde Z^{N,x}(t) - \phi^x(t) |< \epsilon; F_2^{N,\epsilon} \cap F_3^{N,\epsilon} \Big]
\notag \\
&\qquad \geq \mathbb{P} \Big[F_2^{N,\epsilon} \cap F_3^{N,\epsilon}\Big] \cdot \inf_{z \in \tilde B}\mathbb{P} \Big[\sup_{t \in [0, \eta - \epsilon]} |\tilde Z^{N,z}(t) - \phi^{z}(t) |< \frac{3}{16}\epsilon\Big] 
\notag \\
&\qquad \geq \mathbb{P} \Big[F_2^{N,\epsilon} \cap F_3^{N,\epsilon}\Big] \cdot \inf_{z \in \tilde B}\mathbb{P} \Big[\sup_{t \in [0, \eta - \epsilon]} |\hat Z^{N,z}(t) - \phi^{z}(t) |< \frac{3}{16}\epsilon\Big] \notag \\
&\qquad \geq 1-\hat C_1 \exp\big(-N \hat C_2(\epsilon)\big) \notag
\end{align}
for appropriate constants  $\hat C_1$, $\hat C_2$ with $\hat C_2(\epsilon) = O(\epsilon^2)$ as $\epsilon \downarrow 0$ by Theorem~\ref{ThLLN} and Inequalities~\eqref{EqPoissonBound} as required. As the rates are vanishing like polynomials, \eqref{EqAssC3.3} is satisfied.

Our model has a disease--free equilibrium $(0,0)$, and  if $\beta>\gamma$ it has a stable endemic equilibrium
(and then the disease--free equilibrium is unstable). We assume from now on that $\beta>\gamma$, and we seek to estimate the time it takes for the random perturbations to drive our system form the stable endemic equilibrium
to the disease--free equilibrium. The characteristic boundary which we want to hit is the set $\{x=0,0\le y\le 1\}$.
We note, however, that not only the Assumption \ref{AssExit} (D2) but also the Assumption \ref{AssExit} (D5) fail to be satisfied here. Consequently we cannot apply Corollary \ref{cor-exit}. We will now show that
if we denote by $\tau^{N,x}=\inf\{t>0,\ Z^N(t)\in\tilde O\}$ where $\tilde O=\{z_1=0\}$, then Theorem \ref{TheoExitTime} applies.
All we have to show is that for any $\delta>0$, $\bar{V}$ being defined as in section \ref{SectionExitTime},
\[
\lim_{N\rightarrow\infty} \mathbb{P}\big[e^{(\bar{V} - \delta)N}<\tau^{N,x}<e^{(\bar{V} + \delta)N}\big]=1.
\]
For any $\eta>0$, let $O_\eta=\{(z_1,z_2)\in [0,1]^2,z_1>\eta, z_1+z_2\le1\}$, 
$\tilde O_\eta=\{z_1=\eta,0\le z_2\le 1-\eta\}$,
 and $\tau^{N,x}_\eta=\inf\{t>0,\ Z^N(t)\in\tilde O_\eta\}$.
We note that all the above assumptions, including \ref{AssExit} (D1), (D2),..., (D5) are satisfied for this new exit problem, so that
\[
\lim_{N\rightarrow\infty} \mathbb{P}\big[e^{(\bar{V}_\eta - \delta/2)N}<\tau^{N,x}_\eta<e^{(\bar{V}_\eta + \delta/2)N}\big]=1,
\]
where $\bar{V}_\eta=\inf_{z\in\tilde O_\eta}V(x^\ast,z)$. Now for $\eta_0>0$ such that whenever $\eta\le\eta_0$, 
$\bar{V}-\delta/2\le \bar{V}_\eta<\bar{V}$. Moreover clearly $\tau^{N,x}_\eta\le\tau^{N,x}$. From these follows clearly the fact that $\P(\tau^{N,x}>e^{(\bar{V} - \delta)N})\to1$ as $N\to\infty$. It remains to establish the upper bound.

The crucial result which allows us to overcome the new difficulty is the
\begin{lemma}\label{le:extinct}
For any $\eta>0$, $t>0$, 
\[\liminf_{N\to\infty}\frac{1}{N}\log\inf_{x\in A^N\cap O_\eta^c}\P(\tau^{N,x}<t)\ge-\eta\log\left(\frac{\beta}{\gamma}\right).\]
\end{lemma}
\begin{proof}
The first component of the process $Z^{N,x}(t)$ is dominated by the  process
\[x_1+\frac{1}{N}P_1\left(N\beta\int_0^tZ^{N,x}_1(s)ds\right)-\frac{1}{N}P_1\left(N\gamma\int_0^tZ^{N,x}_1(s)ds\right),
\] which is a continuous time binary branching process with birth rate $\beta$ and death rate $\gamma$. This process goes extinct before time $t$ with probability (see the formula in the middle of page 108 in \cite{Athreya1972})
\[ \left(\frac{\gamma e^{N(\beta-\gamma)t}-N^{-1}}{\beta e^{N(\beta-\gamma)t}-\gamma}\right)^{Nx_1}.\]
The result follows readily, since $x\in O_\eta^c$ implies that $x_1\le \eta$.
\end{proof}

In order to adapt the proof of the upper bound in Theorem \ref{TheoExitTime}, all we have to do is to extend
the proof of Lemma~\ref{LemAbove1} to the time of extinction in the SIRS model, which we now do. Indeed, from
Lemma~\ref{LemAbove1}, for any $\eta, \delta, \rho>0$,  
there exists $T_0$ such that
\[\liminf_{N\to\infty}\frac{1}{N}\log\inf_{x\in\overline{B(x^\ast,\rho)}}\P(\tau^{N,x}_\eta\le T_0-1)>-(\bar{V}+\delta/2).\]
On the other hand, from Lemma~\ref{le:extinct},  provided $\eta<\frac{\delta}{2\log(\beta/\gamma)}$,
\[\liminf_{N\to\infty}\frac{1}{N}\log\inf_{x\in A^N\cap O_\eta^c}\P(\tau^{N,x}<1)\ge-\delta/2.\]
The statement of Lemma~\ref{LemAbove1} now follows from the strong Markov property and  the last two
estimates.

\appendix

\section{Change of measure} \label{AppendixChange}

We assume that $Z^{N,x}=Z^N$ has rates $\{N \beta_j| j=1,\dots,k\}$ under $\mathbb P$ and rates $\{N \tilde \beta_j| j=1,\dots,k\}$ under $\tilde{\mathbb{P}}$. 
We furthermore assume that for $x \in A$,
\[
\tilde \beta_j(x) >0 \text{ only if } \beta_j(x) >0.
\]
Hence, $\tilde{\mathbb{P}}|_{\mathcal F_t}$ is absolutely continuous with respect to $\mathbb P|_{\mathcal F_t}$ but not necessarily vice versa.

We require Theorem~B.6 of \cite{Shwartz1995} which gives us an important change of measure formula.

\begin{theorem} \label{TheoremB6}
For all $T>0$ and non-negative, $\mathcal F_T$-measurable random variables $X$, we have
\[
\mathbb{E}[\xi_T X] = \tilde{\mathbb{E}}[X],
\]
where $\mathbb{\tilde{E}}$ denotes the expectation with respect to $\tilde{\mathbb{P}}$ and 
\begin{align} 
\xi_T &:= \exp \Big(\sum_\tau \Big[  \log\big(\tilde\beta_{j(\tau)}(Z^N(\tau-))\big)  - \log\big(\beta_{j(\tau)}(Z^N(\tau-)\big)\Big] \notag \\*
& \qquad- N \sum_j\int_0^T \big(\tilde \beta_j(Z^N(t)) - \beta_j(Z^N(t))\big) dt  \Big);\label{EqChangeofMeasure}
\end{align}
here, we sum over the jump times $\tau\in [0,T]$ of $Z^N$; $j(\tau)$ denotes the corresponding type of the jump direction. In other words, we have
\[
\xi_T=\frac{d \tilde{\mathbb{P}}}{d \mathbb{P}}\Big|_{\mathcal F_T}.
\]
\end{theorem}

We observe that (under $\mathbb P$) $\xi_T=0$ if and only if there exists a jump time $\tau \in [0,T]$ (with jump type $j(\tau)$) and $\tilde \beta_{j(\tau)}(Z^N(\tau-))=0$.

We deduce the following result. Note that since $\tilde{\mathbb{P}}[\xi_T=0]=0$,  $\xi_T^{-1}$ is well-defined $\tilde{\mathbb{P}}$-almost surely.

\begin{cor} \label{CorB.6}
For every non-negative measurable function $X\geq 0$,
\[
\mathbb{E}[X] \geq \tilde{\mathbb{E}}[\xi_T^{-1} X]. \]
\end{cor}

\begin{proof}
As $X \geq 0$, we have
\[
\mathbb{E}[X]\geq \mathbb{E}[X \mathds{1}_{\{ \xi_T \not= 0 \}}] =\tilde{\mathbb{E}}[X \mathds{1}_{\{ \xi_T \not= 0\}} \xi_T^{-1}]
=\tilde{\mathbb{E}}[X \xi_T^{-1}].
\]
\end{proof}

\bigskip

\textbf{Acknowledgements:}
This research was supported by the ANR project MANEGE, the DAAD, and the Labex Archim\`ede.

\bibliographystyle{abbrvnat}
\bibliography{bibliographynew}

\begin{thebibliography}{16}
\providecommand{\natexlab}[1]{#1}
\providecommand{\url}[1]{\texttt{#1}}
\expandafter\ifx\csname urlstyle\endcsname\relax
  \providecommand{\doi}[1]{doi: #1}\else
  \providecommand{\doi}{doi: \begingroup \urlstyle{rm}\Url}\fi

\bibitem[Athreya and Ney(1972)]{Athreya1972}
K.~B. Athreya and P.~E. Ney.
\newblock \emph{Branching Processes}.
\newblock Springer, New York, 1972.

\bibitem[Billingsley(1999)]{Billingsley1999}
P.~Billingsley.
\newblock \emph{Convergence of Probability Measures}.
\newblock Wiley, New York, 1999.

\bibitem[Dembo and Zeitouni(2009)]{Dembo2009}
A.~Dembo and O.~Zeitouni.
\newblock \emph{Large deviations techniques and applications}.
\newblock Springer, Berlin, 2009.

\bibitem[Dupuis and Ellis(1997)]{Dupuis1997}
P.~Dupuis and R.~S. Ellis.
\newblock \emph{A weak convergence approach to the theory of large deviations}.
\newblock Wiley, New York, 1997.

\bibitem[Dupuis et~al.(1991)Dupuis, Ellis, and Weiss]{Dupuis1991}
P.~Dupuis, R.~S. Ellis, and A.~Weiss.
\newblock Large deviations for {M}arkov processes with discontinuous
  statistics. {I}. {G}eneral upper bounds.
\newblock \emph{Ann. Probab.}, 19\penalty0 (3):\penalty0 1280--1297, 1991.

\bibitem[Feng and Kurtz(2006)]{Feng2006}
J.~Feng and T.~G. Kurtz.
\newblock \emph{Large deviations for stochastic processes}.
\newblock American Mathematical Society, Providence, 2006.

\bibitem[Freidlin and Wentzell(2012)]{Freidlin2012}
M.~I. Freidlin and A.~D. Wentzell.
\newblock \emph{Random Perturbations of Dynamical Systems}.
\newblock Springer, Berlin, 2012.

\bibitem[Komiya(1988)]{Komiya}
H.~Komiya.
\newblock Elementary proof for {S}ion's minimax theorem.
\newblock \emph{Kodai Math. J.}, 11\penalty0 (1):\penalty0 5--7, 1988.

\bibitem[Kratz et~al.(2015)Kratz, Pardoux, and Samegni~Kepgnou]{kratz2015}
P.~Kratz, E.~Pardoux, and B.~Samegni~Kepgnou.
\newblock Numerical methods in the context of compartmental models in
  epidemiology.
\newblock \emph{ESAIM: Proceedings}, 48:\penalty0 169--189, 2015.

\bibitem[Kribs-Zaleta and Velasco-Hern\'andez(2000)]{Kribs2000}
C.~M. Kribs-Zaleta and Velasco-Hern\'andez.
\newblock A simple vaccination model with multiple endemic states.
\newblock \emph{Math. Biosci.}, 164\penalty0 (2):\penalty0 183--201, 2000.

\bibitem[Kurtz(1978)]{Kurtz1978}
T.~G. Kurtz.
\newblock Strong approximation theorems for density dependent {M}arkov chains.
\newblock \emph{Stochastic Processes and their Applications}, 6\penalty0
  (3):\penalty0 223--240, 1978.

\bibitem[Pakdaman et~al.(2010)Pakdaman, Thieullen, and Wainrib]{Pakdaman2010}
K.~Pakdaman, M.~Thieullen, and G.~Wainrib.
\newblock Diffusion approximation of birth-death processes: Comparison in terms
  of large deviations and exit points.
\newblock \emph{Statistics \& Probability Letters}, 80\penalty0
  (13-14):\penalty0 1121--1127, 2010.

\bibitem[Revuz and Yor(2005)]{RevuzYor1999}
D.~Revuz and M.~Yor.
\newblock \emph{Continuous Martingales and Brownian Motion}.
\newblock Springer, Berlin, 2005.

\bibitem[Roydon(1968)]{Roydon1968}
H.~Roydon.
\newblock \emph{Real Analysis}.
\newblock Collier-Macmillan, London, 1968.

\bibitem[Shwartz and Weiss(1995)]{Shwartz1995}
A.~Shwartz and A.~Weiss.
\newblock \emph{Large Deviations for Performance Analysis}.
\newblock Chapman Hall, London, 1995.

\bibitem[Shwartz and Weiss(2005)]{Shwartz2005}
A.~Shwartz and A.~Weiss.
\newblock Large deviations with diminishing rates.
\newblock \emph{Mathematics of Operations Research}, 30\penalty0 (2):\penalty0
  281--310, 2005.

\end{thebibliography}

\allowdisplaybreaks

\end{document}